\documentclass[twoside,11pt,reqno]{amsart}
\usepackage{amsmath,amsthm,amssymb,amstext,amsfonts,amscd}
\usepackage{graphicx}
\usepackage{multirow}

\newtheorem{theorem}{Theorem}

\newtheorem{definition}{Definition}

\newtheorem{lemma}{Lemma}

\newtheorem{remark}{Remark}

\numberwithin{equation}{section}
\input{tcilatex}

\begin{document}
\title[Few results in connection with sum and product theorems.....]{Few
results in connection with sum and product theorems of relative $\left(
p,q\right) $-$\varphi $ order, relative $\left( p,q\right) $-$\varphi $ type
and relative $\left( p,q\right) $-$\varphi $ weak type of meromorphic
functions with respect to entire functions}
\author[Tanmay Biswas]{Tanmay Biswas}
\address{T. Biswas : Rajbari, Rabindrapalli, R. N. Tagore Road, P.O.-
Krishnagar, Dist-Nadia, PIN-\ 741101, West Bengal, India}
\email{tanmaybiswas\_math@rediffmail.com}

\begin{abstract}
Orders and types of entire and meromorphic functions have been actively
investigated by many authors. In the present paper, we aim at investigating
some basic properties in connection with sum and product of relative $\left(
p,q\right) $-$\varphi $ order, relative $\left( p,q\right) $-$\varphi $
type, and relative $\left( p,q\right) $-$\varphi $ weak type of meromorphic
functions with respect to entire functions where $p,q$ are any two positive
integers and $\varphi $ $:$ $[0,+\infty )\rightarrow (0,+\infty )$ be a
non-decreasing unbounded function.
\end{abstract}

\maketitle

\section{\textbf{Introduction, Definitions and Notations}}

\qquad Let $f$ be an entire function defined in the open complex plane $%
\mathbb{C}
.$ The\ maximum modulus function $M_{f}\left( r\right) $ corresponding to $f$%
\ (see \cite{14}) is defined on $\left\vert z\right\vert =r$ as $M_{f}\left(
r\right) =\QTATOP{\max }{\left\vert z\right\vert =r}\left\vert f\left(
z\right) \right\vert $. A non-constant entire function $f$ is said have the
Property (A) if for any $\sigma >1$ and for all sufficiently large $r,$ $%
\left[ M_{f}\left( r\right) \right] ^{2}\leq M_{f}\left( r^{\sigma }\right) $
holds $\left( \text{see \cite{1}}\right) $. When $f$ is meromorphic, one may
introduce another function $T_{f}\left( r\right) $ known as Nevanlinna's
characteristic function of $f$ (see \cite[p.4]{6}), playing the same role as 
$M_{f}\left( r\right) .$ If $f$ is non-constant entire function, then its
Nevanlinna's characteristic function is strictly increasing and continuous
and therefore there exists its inverse functions $T_{f}^{-1}(r)$ $:$ $\left(
\left\vert f\left( 0\right) \right\vert ,\infty \right) $ $\rightarrow $ $%
\left( 0,\infty \right) $ with $\underset{s\rightarrow \infty }{\lim }%
T_{f}^{-1}\left( s\right) =\infty .$

\qquad However, throughout this paper, we assume that the reader is familiar
with the fundamental results and the standard notations of the Nevanlinna
theory of meromorphic functions which are available in \cite{6, 10, 12, 13}
and therefore we do not explain those in details. Now we define $\exp
^{[k]}x=\exp \left( \exp ^{[k-1]}x\right) $ and $\log ^{[k]}x=\log \left(
\log ^{[k-1]}x\right) $ for $x\in \lbrack 0,\infty )$ and $k\in 
\mathbb{N}
$ where $%
\mathbb{N}
$ be the set of all positive integers$.$ We also denote $\log ^{[0]}x=x,$ $%
\log ^{[-1]}x=\exp x,$ $\exp ^{[0]}x=x$ and $\exp ^{[-1]}x=\log x.$ Further
we assume that throughout the present paper $p$ and $q$ always denote
positive integers.

\qquad Mainly the growth investigation of meromorphic functions has usually
been done through its Nevanlinna's characteristic function in comparison
with those of exponential function. But if one is paying attention to
evaluate the growth rates of any meromorphic function with respect to an
entire function, the notions of relative growth indicators \cite{9} will
come. Extending this notion, Debnath et. al. \cite{3} introduce the
definition of relative $\left( p,q\right) $-th order and relative $\left(
p,q\right) $-th lower order of a meromorphic function $f$ with respect to
another entire function $g$ respectively in the light of index-pair ( detail
about index-pair one may see \cite{3, 7, 8} ). For details about it, one may
see \cite{3}. Extending this notion, recently Biswas \cite{xxx} introduce
the definitions of relative $\left( p,q\right) $-$\varphi $ order and the
relative $\left( p,q\right) $-$\varphi $ lower order of a meromorphic
function $f$ with respect to another entire function $g$ as follows:

\begin{definition}
\label{d1}\cite{xxx} Let $\varphi $ $:$ $[0,+\infty )\rightarrow (0,+\infty
) $ be a non-decreasing unbounded function. The relative $\left( p,q\right) $%
-$\varphi $ order and the relative $\left( p,q\right) $-$\varphi $ lower
order of a meromorphic function $f$ with respect to an entire function $g$
are defined as%
\begin{equation*}
\begin{array}{c}
\rho _{g}^{\left( p,q\right) }\left( f,\varphi \right) \\ 
\lambda _{g}^{\left( p,q\right) }\left( f,\varphi \right)%
\end{array}%
=\underset{r\rightarrow \infty }{\lim }%
\begin{array}{c}
\sup \\ 
\inf%
\end{array}%
\frac{\log ^{\left[ p\right] }T_{g}^{-1}\left( T_{f}\left( r\right) \right) 
}{\log ^{\left[ q\right] }\varphi \left( r\right) }.
\end{equation*}
\end{definition}

\qquad If we consider $\varphi (r)=r$, then the above definition reduce to
the definitions of relative $\left( p,q\right) $-th order and relative $%
\left( p,q\right) $-th lower order of a meromorphic $f$ with respect to an
entire $g,$ introduced by Debnath et. al. \cite{3}.

\qquad If the relative $\left( p,q\right) $-$\varphi $ order and the
relative $\left( p,q\right) $-$\varphi $ lower order of $f$ with respect to $%
g$ are the same, then $f$ is called a function of regular relative $\left(
p,q\right) $-$\varphi $ growth with respect to $g$. Otherwise, $f$ is said
to be\ irregular relative $\left( p,q\right) $-$\varphi $ growth with
respect to $g$.

\qquad Now in order to refine the above growth scale, one may introduce the
definitions of other growth indicators, such as relative $\left( p,q\right) $%
-$\varphi $ type and relative $\left( p,q\right) $-$\varphi $ lower type%
\emph{\ }of entire or meromorphic functions with respect to another entire
function which are as follows:

\begin{definition}
\label{d2}\cite{xxx} Let $\varphi $ $:$ $[0,+\infty )\rightarrow (0,+\infty
) $ be a non-decreasing unbounded function. The relative $\left( p,q\right) $%
-$\varphi $ type and the relative $\left( p,q\right) $-$\varphi $ lower type
of a meromorphic function $f$ with respect to another entire function $g$
having non-zero finite relative $\left( p,q\right) $-$\varphi $ order $\rho
_{g}^{\left( p,q\right) }\left( f,\varphi \right) $ are defined as :%
\begin{equation*}
\begin{array}{c}
\sigma _{g}^{\left( p,q\right) }\left( f,\varphi \right) \\ 
\overline{\sigma }_{g}^{\left( p,q\right) }\left( f,\varphi \right)%
\end{array}%
=\underset{r\rightarrow +\infty }{\lim }%
\begin{array}{c}
\sup \\ 
\inf%
\end{array}%
\frac{\log ^{\left[ p-1\right] }T_{g}^{-1}\left( T_{f}\left( r\right)
\right) }{\left[ \log ^{\left[ q-1\right] }\varphi \left( r\right) \right]
^{\rho _{g}^{\left( p,q\right) }\left( f,\varphi \right) }}.
\end{equation*}
\end{definition}

\qquad Analogously, to determine the relative growth of $f$ having same non
zero finite relative $\left( p,q\right) $-$\varphi $ lower order with
respect to $g$, one can introduce the definition of relative $\left(
p,q\right) $-$\varphi $ weak type $\tau _{g}^{\left( p,q\right) }\left(
f\right) $ and the growth indicator $\overline{\tau }_{g}^{\left( p,q\right)
}\left( f\right) $ of $f$ with respect to $g$ of finite positive relative $%
\left( p,q\right) $-$\varphi $ lower order $\lambda _{g}^{\left( p,q\right)
}\left( f\right) $ in the following way:

\begin{definition}
\label{d3}\cite{xxx} Let $\varphi $ $:$ $[0,+\infty )\rightarrow (0,+\infty
) $ be a non-decreasing unbounded function. The relative $\left( p,q\right) $%
-$\varphi $ weak type $\tau _{g}^{\left( p,q\right) }\left( f,\varphi
\right) $ and the growth indicator $\overline{\tau }_{g}^{\left( p,q\right)
}\left( f,\varphi \right) $ of a meromorphic function $f$ with respect to
another entire function $g$ having non-zero finite relative $\left(
p,q\right) $-$\varphi $ lower order $\lambda _{g}^{\left( p,q\right) }\left(
f,\varphi \right) $ are defined as :%
\begin{equation*}
\begin{array}{c}
\tau _{g}^{\left( p,q\right) }\left( f,\varphi \right) \\ 
\overline{\tau }_{g}^{\left( p,q\right) }\left( f,\varphi \right)%
\end{array}%
=\underset{r\rightarrow +\infty }{\lim }%
\begin{array}{c}
\inf \\ 
\sup%
\end{array}%
\frac{\log ^{\left[ p-1\right] }T_{g}^{-1}\left( T_{f}\left( r\right)
\right) }{\left[ \log ^{\left[ q-1\right] }\varphi \left( r\right) \right]
^{\lambda _{g}^{\left( p,q\right) }\left( f,\varphi \right) }}.
\end{equation*}
\end{definition}

\qquad If we consider $\varphi (r)=r$, then $\sigma _{g}^{\left( p,q\right)
}\left( f,r\right) $ and $\tau _{g}^{\left( p,q\right) }\left( f,r\right) $
are respectively known as relative $\left( p,q\right) $-th type and relative 
$\left( p,q\right) $-th weak type of $f$ with respect to $g$. For details
about relative $\left( p,q\right) $-th type, relative $\left( p,q\right) $%
-th weak type etc., one may see \cite{yyy}.

\qquad Here, in this paper, we aim at investigating some basic properties of
relative $\left( p,q\right) $-$\varphi $ order, relative $\left( p,q\right) $%
-$\varphi $ type and relative $\left( p,q\right) $-$\varphi $ weak type of a
meromorphic function with respect to an entire function under somewhat
different conditions. Throughout this paper, we assume that all the growth
indicators are all nonzero finite.

\section{\textbf{Lemmas}}

\qquad In this section we present some lemmas which will be needed in the
sequel.

\begin{lemma}
\label{l9.2}\cite{1} Let $f$ be an entire function which satisfies the
Property (A) then for any positive integer $n$ and for all sufficiently
large $r,$%
\begin{equation*}
\left[ M_{f}\left( r\right) \right] ^{n}\leq M_{f}\left( r^{\delta }\right)
\end{equation*}%
holds where $\delta >1.$
\end{lemma}

\begin{lemma}
\label{l9.6}\cite[p. 18]{6} Let $f$ be an entire function. Then for all
sufficiently large values of $r,$%
\begin{equation*}
T_{f}\left( r\right) \leq \log M_{f}\left( r\right) \leq 3T_{f}\left(
2r\right) ~.
\end{equation*}
\end{lemma}

\section{\textbf{Main Results}}

\qquad In this section we present some results which will be needed in the
sequel.

\begin{theorem}
\label{t9.2x} Let $f_{1}$, $f_{2}$ be meromorphic functions and $g_{1}$ be
any entire function such that at least $f_{1}$ or $f_{2}$ is of regular
relative $\left( p,q\right) $-$\varphi $ growth with respect to $g_{1}$.
Also let $g_{1}$ has the Property (A). Then%
\begin{equation*}
\lambda _{g_{1}}^{\left( p,q\right) }\left( f_{1}\pm f_{2},\varphi \right)
\leq \max \left\{ \lambda _{g_{1}}^{\left( p,q\right) }\left( f_{1},\varphi
\right) ,\lambda _{g_{1}}^{\left( p,q\right) }\left( f_{2},\varphi \right)
\right\} ~.
\end{equation*}%
The equality holds when $\lambda _{g_{1}}^{\left( p,q\right) }\left(
f_{i},\varphi \right) >\lambda _{g_{1}}^{\left( p,q\right) }\left(
f_{j},\varphi \right) $ with at least $f_{j}$ is of regular relative $\left(
p,q\right) $-$\varphi $ growth with respect to $g_{1}$ where $i,j=1,2$ and $%
i\neq j$.
\end{theorem}

\begin{proof}
The result is obvious when $\lambda _{g_{1}}^{\left( p,q\right) }\left(
f_{1}\pm f_{2},\varphi \right) =0$. So we suppose that $\lambda
_{g_{1}}^{\left( p,q\right) }\left( f_{1}\pm f_{2},\varphi \right) $ $>$ $0$%
. We can clearly assume that $\lambda _{g_{1}}^{\left( p,q\right) }\left(
f_{k},\varphi \right) $ is finite for $k=1,2.$ Now let us consider that $%
\max \left\{ \lambda _{g_{1}}^{\left( p,q\right) }\left( f_{1},\varphi
\right) ,\lambda _{g_{1}}^{\left( p,q\right) }\left( f_{2},\varphi \right)
\right\} =\Delta $ and $f_{2}$ is of regular relative $\left( p,q\right) $-$%
\varphi $ growth with respect to $g_{1}.$

\qquad Now for any arbitrary $\varepsilon >0$ from the definition of $%
\lambda _{g_{1}}^{\left( p,q\right) }\left( f_{1},\varphi \right) $, we have
for a sequence values of $r$ tending to infinity that%
\begin{equation*}
T_{f_{1}}\left( r\right) \leq T_{g_{1}}\left[ \exp ^{\left[ p\right] }\left[
\left( \lambda _{g_{1}}^{\left( p,q\right) }\left( f_{1},\varphi \right)
+\varepsilon \right) \log ^{\left[ q\right] }\varphi \left( r\right) \right] %
\right]
\end{equation*}%
\begin{equation}
i.e.,~T_{f_{1}}\left( r\right) \leq T_{g_{1}}\left[ \exp ^{\left[ p\right] }%
\left[ \left( \Delta +\varepsilon \right) \log ^{\left[ q\right] }\varphi
\left( r\right) \right] \right] ~.  \label{50.1}
\end{equation}

\qquad Also for any arbitrary $\varepsilon >0$ from the definition of $\rho
_{g_{1}}^{\left( p,q\right) }\left( f_{2},\varphi \right) \left( =\lambda
_{g_{1}}^{\left( p,q\right) }\left( f_{2},\varphi \right) \right) $, we
obtain for all sufficiently large values of $r$ that%
\begin{equation}
T_{f_{2}}\left( r\right) \leq T_{g_{1}}\left[ \exp ^{\left[ p\right] }\left[
\left( \lambda _{g_{1}}^{\left( p,q\right) }\left( f_{2},\varphi \right)
+\varepsilon \right) \log ^{\left[ q\right] }\varphi \left( r\right) \right] %
\right]  \label{91.1}
\end{equation}%
\begin{equation}
i.e.,~T_{f_{2}}\left( r\right) \leq T_{g_{1}}\left[ \exp ^{\left[ p\right] }%
\left[ \left( \Delta +\varepsilon \right) \log ^{\left[ q\right] }\varphi
\left( r\right) \right] \right] ~.  \label{9.1c}
\end{equation}

\qquad Since $T_{f_{1}\pm f_{2}}\left( r\right) \leq T_{f_{1}}\left(
r\right) +T_{f_{2}}\left( r\right) +O(1)$ for all large $r,$ so in view of $%
\left( \ref{50.1}\right) $ , $\left( \ref{9.1c}\right) $ and Lemma \ref{l9.6}%
, we obtain for a sequence values of $r$ tending to infinity that%
\begin{equation*}
T_{f_{1}\pm f_{2}}\left( r\right) \leq 2\log M_{g_{1}}\left[ \exp ^{\left[ p%
\right] }\left[ \left( \Delta +\varepsilon \right) \log ^{\left[ q\right]
}\varphi \left( r\right) \right] \right] +O(1)
\end{equation*}%
\begin{equation}
i.e.,~T_{f_{1}\pm f_{2}}\left( r\right) \leq 3\log M_{g_{1}}\left[ \exp ^{%
\left[ p\right] }\left[ \left( \Delta +\varepsilon \right) \log ^{\left[ q%
\right] }\varphi \left( r\right) \right] \right] ~.  \label{9.2}
\end{equation}

\qquad Therefore in view of Lemma \ref{l9.2} and Lemma \ref{l9.6}$,$ we
obtain from $\left( \ref{9.2}\right) $ for a sequence values of $r$ tending
to infinity and $\sigma >1$ that%
\begin{equation*}
T_{f_{1}\pm f_{2}}\left( r\right) \leq \frac{1}{3}\log \left[ M_{g_{1}}\left[
\exp ^{\left[ p\right] }\left[ \left( \Delta +\varepsilon \right) \log ^{%
\left[ q\right] }\varphi \left( r\right) \right] \right] \right] ^{9}
\end{equation*}%
\begin{equation*}
i.e.,~T_{f_{1}\pm f_{2}}\left( r\right) \leq \frac{1}{3}\log M_{g_{1}}\left[ %
\left[ \exp ^{\left[ p\right] }\left[ \left( \Delta +\varepsilon \right)
\log ^{\left[ q\right] }\varphi \left( r\right) \right] \right] ^{\sigma }%
\right]
\end{equation*}%
\begin{equation*}
i.e.,~T_{f_{1}\pm f_{2}}\left( r\right) \leq T_{g_{1}}\left[ 2\left[ \exp ^{%
\left[ p\right] }\left[ \left( \Delta +\varepsilon \right) \log ^{\left[ q%
\right] }\varphi \left( r\right) \right] \right] ^{\sigma }\right] ~.
\end{equation*}

\qquad Now we get from above by letting $\sigma \rightarrow 1^{+}$%
\begin{equation*}
i.e.,~\underset{r\rightarrow \infty }{\lim \inf }\frac{\log ^{\left[ p\right]
}T_{g_{1}}^{-1}\left( T_{f_{1}\pm f_{2}}\left( r\right) \right) }{\log ^{%
\left[ q\right] }\varphi \left( r\right) }<\left( \Delta +\varepsilon
\right) ~.~\ \ \ \ \ \ \ \ \ \ \ 
\end{equation*}%
Since $\varepsilon >0$ is arbitrary,%
\begin{equation*}
\lambda _{g_{1}}^{\left( p,q\right) }\left( f_{1}\pm f_{2},\varphi \right)
\leq \Delta =\max \left\{ \lambda _{g_{1}}^{\left( p,q\right) }\left(
f_{1},\varphi \right) ,\lambda _{g_{1}}^{\left( p,q\right) }\left(
f_{2},\varphi \right) \right\} ~.
\end{equation*}

\qquad Similarly, if we consider that $f_{1}$ is of regular relative $\left(
p,q\right) $-$\varphi $ growth with respect to $g_{1}$ or both $f_{1}$ and $%
f_{2}$ are of regular relative $\left( p,q\right) $-$\varphi $ growth with
respect to $g_{1},$ then one can easily verify that%
\begin{equation}
\lambda _{g_{1}}^{\left( p,q\right) }\left( f_{1}\pm f_{2},\varphi \right)
\leq \Delta =\max \left\{ \lambda _{g_{1}}^{\left( p,q\right) }\left(
f_{1},\varphi \right) ,\lambda _{g_{1}}^{\left( p,q\right) }\left(
f_{2},\varphi \right) \right\} ~.  \label{9.3}
\end{equation}

\qquad Further without loss of any generality, let $\lambda _{g_{1}}^{\left(
p,q\right) }\left( f_{1},\varphi \right) $ $<$ $\lambda _{g_{1}}^{\left(
p,q\right) }\left( f_{2},\varphi \right) $ and $f=f_{1}\pm f_{2}.$ Then in
view of $\left( \ref{9.3}\right) $ we get that $\lambda _{g_{1}}^{\left(
p,q\right) }\left( f,\varphi \right) $ $\leq $ $\lambda _{g_{1}}^{\left(
p,q\right) }\left( f_{2},\varphi \right) .$ As, $f_{2}=\pm \left(
f-f_{1}\right) $ and in this case we obtain that $\lambda _{g_{1}}^{\left(
p,q\right) }\left( f_{2},\varphi \right) $ $\leq $ $\max \left\{ \lambda
_{g_{1}}^{\left( p,q\right) }\left( f,\varphi \right) ,\lambda
_{g_{1}}^{\left( p,q\right) }\left( f_{1},\varphi \right) \right\} ~.$ As we
assume that $\lambda _{g_{1}}^{\left( p,q\right) }\left( f_{1},\varphi
\right) $ $<$ $\lambda _{g_{1}}^{\left( p,q\right) }\left( f_{2},\varphi
\right) ,$ therefore we have $\lambda _{g_{1}}^{\left( p,q\right) }\left(
f_{2},\varphi \right) $ $\leq $ $\lambda _{g_{1}}^{\left( p,q\right) }\left(
f,\varphi \right) $ and hence $\lambda _{g_{1}}^{\left( p,q\right) }\left(
f,\varphi \right) $ $=$ $\lambda _{g_{1}}^{\left( p,q\right) }\left(
f_{2},\varphi \right) $ $=$ $\max \left\{ \lambda _{g_{1}}^{\left(
p,q\right) }\left( f_{1},\varphi \right) ,\lambda _{g_{1}}^{\left(
p,q\right) }\left( f_{2},\varphi \right) \right\} .$ Therefore, $\lambda
_{g_{1}}^{\left( p,q\right) }\left( f_{1}\pm f_{2},\varphi \right) =\lambda
_{g_{1}}^{\left( p,q\right) }\left( f_{i},\varphi \right) \mid i=1,2$
provided $\lambda _{g_{1}}^{\left( p,q\right) }\left( f_{1},\varphi \right)
\neq \lambda _{g_{1}}^{\left( p,q\right) }\left( f_{2},\varphi \right) .$
Thus the theorem is established.
\end{proof}

\begin{theorem}
\label{t9.1x} Let $f_{1}$ and $f_{2}$ be any two meromorphic functions and $%
g_{1}$ be an entire function such that such that $\rho _{g_{1}}^{\left(
p,q\right) }\left( f_{1},\varphi \right) $ and $\rho _{g_{2}}^{\left(
p,q\right) }\left( f_{1},\varphi \right) $ exists . Also let $g_{1}$ has the
Property (A). Then%
\begin{equation*}
\rho _{g_{1}}^{\left( p,q\right) }\left( f_{1}\pm f_{2},\varphi \right) \leq
\max \left\{ \rho _{g_{1}}^{\left( p,q\right) }\left( f_{1},\varphi \right)
,\rho _{g_{1}}^{\left( p,q\right) }\left( f_{2},\varphi \right) \right\} ~.
\end{equation*}%
The equality holds when $\rho _{g_{1}}^{\left( p,q\right) }\left(
f_{1},\varphi \right) \neq \rho _{g_{1}}^{\left( p,q\right) }\left(
f_{2},\varphi \right) $.
\end{theorem}

\qquad We omit the proof of Theorem \ref{t9.1x} as it can easily be carried
out in the line of Theorem \ref{t9.2x}.

\begin{theorem}
\label{t9.3} Let $f_{1}$ be a meromorphic function and $g_{1}$, $g_{2}$ be
any two entire functions such that $\lambda _{g_{1}}^{\left( p,q\right)
}\left( f_{1},\varphi \right) $ and $\lambda _{g_{2}}^{\left( p,q\right)
}\left( f_{1},\varphi \right) $ exists. Also let $g_{1}\pm g_{2}$ has the
Property (A). Then%
\begin{equation*}
\lambda _{g_{1}\pm g_{2}}^{\left( p,q\right) }\left( f_{1},\varphi \right)
\geq \min \left\{ \lambda _{g_{1}}^{\left( p,q\right) }\left( f_{1},\varphi
\right) ,\lambda _{g_{2}}^{\left( p,q\right) }\left( f_{1},\varphi \right)
\right\} ~.
\end{equation*}%
The equality holds when $\lambda _{g_{1}}^{\left( p,q\right) }\left(
f_{1},\varphi \right) \neq \lambda _{g_{2}}^{\left( p,q\right) }\left(
f_{1},\varphi \right) $.
\end{theorem}

\begin{proof}
The result is obvious when $\lambda _{g_{1}\pm g_{2}}^{\left( p,q\right)
}\left( f_{1},\varphi \right) =\infty $. So we suppose that $\lambda
_{g_{1}\pm g_{2}}^{\left( p,q\right) }\left( f_{1},\varphi \right) <\infty $%
. We can clearly assume that $\lambda _{g_{k}}^{\left( p,q\right) }\left(
f_{1},\varphi \right) $ is finite for $k=1,2.$ Further let $\Psi =\min
\left\{ \lambda _{g_{1}}^{\left( p,q\right) }\left( f_{1},\varphi \right)
,\lambda _{g_{2}}^{\left( p,q\right) }\left( f_{1},\varphi \right) \right\}
. $ Now for any arbitrary $\varepsilon >0$ from the definition of $\lambda
_{g_{k}}^{\left( p,q\right) }\left( f_{1},\varphi \right) $, we have for all
sufficiently large values of $r$ that%
\begin{equation}
T_{g_{k}}\left[ \exp ^{\left[ p\right] }\left[ \left( \lambda
_{g_{k}}^{\left( p,q\right) }\left( f_{1},\varphi \right) -\varepsilon
\right) \log ^{\left[ q\right] }\varphi \left( r\right) \right] \right] \leq
T_{f_{1}}\left( r\right) \text{ \ where }k=1,2  \label{99.4x}
\end{equation}%
\begin{equation*}
i.e,~T_{g_{k}}\left[ \exp ^{\left[ p\right] }\left[ \left( \Psi -\varepsilon
\right) \log ^{\left[ q\right] }\varphi \left( r\right) \right] \right] \leq
T_{f_{1}}\left( r\right) \text{ \ where }k=1,2
\end{equation*}

\qquad Since $T_{g_{1}\pm g_{2}}\left( r\right) \leq T_{g_{1}}\left(
r\right) +T_{g_{2}}\left( r\right) +O(1)$ for all large $r,$, we obtain from
above and Lemma \ref{l9.6} for all sufficiently large values of $r$ that%
\begin{equation*}
T_{g_{1}\pm g_{2}}\left[ \exp ^{\left[ p\right] }\left[ \left( \Psi
-\varepsilon \right) \log ^{\left[ q\right] }\varphi \left( r\right) \right] %
\right] \leq 2T_{f_{1}}\left( r\right) +O(1)
\end{equation*}%
\begin{equation*}
i.e.,~T_{g_{1}\pm g_{2}}\left[ \exp ^{\left[ p\right] }\left[ \left( \Psi
-\varepsilon \right) \log ^{\left[ q\right] }\varphi \left( r\right) \right] %
\right] <3T_{f_{1}}\left( r\right) ~.
\end{equation*}

\qquad Therefore in view of Lemma \ref{l9.2} and Lemma \ref{l9.6}$,$ we
obtain from above for all sufficiently large values of $r$ and any $\sigma
>1 $ that

\begin{equation*}
\frac{1}{9}\log M_{g_{1}\pm g_{2}}\left[ \frac{\exp ^{\left[ p\right] }\left[
\left( \Psi -\varepsilon \right) \log ^{\left[ q\right] }\varphi \left(
r\right) \right] }{2}\right] <T_{f_{1}}\left( r\right)
\end{equation*}%
\begin{equation*}
i.e.,~\log M_{g_{1}\pm g_{2}}\left[ \frac{\exp ^{\left[ p\right] }\left[
\left( \Psi -\varepsilon \right) \log ^{\left[ q\right] }\varphi \left(
r\right) \right] }{2}\right] ^{\frac{1}{9}}<T_{f_{1}}\left( r\right)
\end{equation*}%
\begin{equation*}
i.e.,~\log M_{g_{1}\pm g_{2}}\left[ \left( \frac{\exp ^{\left[ p\right] }%
\left[ \left( \Psi -\varepsilon \right) \log ^{\left[ q\right] }\varphi
\left( r\right) \right] }{2}\right) ^{\frac{1}{\sigma }}\right]
<T_{f_{1}}\left( r\right)
\end{equation*}%
\begin{equation*}
i.e.,~T_{g_{1}\pm g_{2}}\left[ \left( \frac{\exp ^{\left[ p\right] }\left[
\left( \Psi -\varepsilon \right) \log ^{\left[ q\right] }\varphi \left(
r\right) \right] }{2}\right) ^{\frac{1}{\sigma }}\right] <T_{f_{1}}\left(
r\right)
\end{equation*}

\qquad As $\varepsilon >0$ is arbitrary, we get from above by letting $%
\sigma \rightarrow 1^{+}$%
\begin{equation}
\lambda _{g_{1}\pm g_{2}}^{\left( p,q\right) }\left( f_{1},\varphi \right)
=\geq \Psi =\min \left\{ \lambda _{g_{1}}^{\left( p,q\right) }\left(
f_{1},\varphi \right) ,\lambda _{g_{2}}^{\left( p,q\right) }\left(
f_{1},\varphi \right) \right\} ~.  \label{99.3}
\end{equation}

\qquad Now without loss of any generality, we may consider that $\lambda
_{g_{1}}^{\left( p,q\right) }\left( f_{1},\varphi \right) <\lambda
_{g_{2}}^{\left( p,q\right) }\left( f_{1},\varphi \right) $ and $g=g_{1}\pm
g_{2}.$ Then in view of $\left( \ref{99.3}\right) $ we get that $\lambda
_{g}^{\left( p,q\right) }\left( f_{1},\varphi \right) \geq \lambda
_{g_{1}}^{\left( p,q\right) }\left( f_{1},\varphi \right) .$ Further, $%
g_{1}=\left( g\pm g_{2}\right) $ and in this case we obtain that $\lambda
_{g_{1}}^{\left( p,q\right) }\left( f_{1},\varphi \right) $ $\geq $ $\min
\left\{ \lambda _{g}^{\left( p,q\right) }\left( f_{1},\varphi \right)
,\lambda _{g_{2}}^{\left( p,q\right) }\left( f_{1},\varphi \right) \right\}
~.$ As we assume that $\lambda _{g_{1}}^{\left( p,q\right) }\left(
f_{1},\varphi \right) <\lambda _{g_{2}}^{\left( p,q\right) }\left(
f_{1},\varphi \right) ,$ therefore we have $\lambda _{g_{1}}^{\left(
p,q\right) }\left( f_{1},\varphi \right) \geq \lambda _{g}^{\left(
p,q\right) }\left( f_{1},\varphi \right) $ and hence $\lambda _{g}^{\left(
p,q\right) }\left( f_{1},\varphi \right) =\lambda _{g_{1}}^{\left(
p,q\right) }\left( f_{1},\varphi \right) =\min \left\{ \lambda
_{g_{1}}^{\left( p,q\right) }\left( f_{1},\varphi \right) ,\lambda
_{g_{2}}^{\left( p,q\right) }\left( f_{1},\varphi \right) \right\} .$
Therefore, $\lambda _{g_{1}\pm g_{2}}^{\left( p,q\right) }\left(
f_{1},\varphi \right) =\lambda _{g_{i}}^{\left( p,q\right) }\left(
f_{1},\varphi \right) \mid i=1,2$ provided $\lambda _{g_{1}}^{\left(
p,q\right) }\left( f_{1},\varphi \right) \neq \lambda _{g_{2}}^{\left(
p,q\right) }\left( f_{1},\varphi \right) .$ Thus the theorem follows.
\end{proof}

\begin{theorem}
\label{t9.4} Let $f_{1}$ be a meromorphic function and $g_{1}$, $g_{2}$ be
any two entire functions such that $f_{1}$ is of regular relative $\left(
p,q\right) $-$\varphi $ growth with respect to at least any one of $g_{1}$
or $g_{2}.$ If $g_{1}\pm g_{2}$ has the Property (A), then%
\begin{equation*}
\rho _{g_{1}\pm g_{2}}^{\left( p,q\right) }\left( f_{1},\varphi \right) \geq
\min \left\{ \rho _{g_{1}}^{\left( p,q\right) }\left( f_{1},\varphi \right)
,\rho _{g_{2}}^{\left( p,q\right) }\left( f_{1},\varphi \right) \right\} ~.
\end{equation*}%
The equality holds when $\rho _{g_{i}}^{\left( p,q\right) }\left(
f_{1},\varphi \right) <\rho _{g_{j}}^{\left( p,q\right) }\left(
f_{1},\varphi \right) $ with at least $f_{1}$ is of regular relative $\left(
p,q\right) $-$\varphi $ growth with respect to $g_{j}$ where $i,j=1,2$ and $%
i\neq j.$
\end{theorem}

\qquad We omit the proof of Theorem \ref{t9.4} as it can easily be carried
out in the line of Theorem \ref{t9.3}.

\begin{theorem}
\label{t9.5} Let $f_{1},f_{2}$ be any two meromorphic functions and $g_{1}$, 
$g_{2}$ be any two entire functions. Also let $g_{1}\pm g_{2}$ has the
Property (A). Then%
\begin{eqnarray*}
&&\rho _{g_{1}\pm g_{2}}^{\left( p,q\right) }\left( f_{1}\pm f_{2},\varphi
\right) \\
&\leq &\max \left[ \min \left\{ \rho _{g_{1}}^{\left( p,q\right) }\left(
f_{1},\varphi \right) ,\rho _{g_{2}}^{\left( p,q\right) }\left(
f_{1},\varphi \right) \right\} ,\min \left\{ \rho _{g_{1}}^{\left(
p,q\right) }\left( f_{2},\varphi \right) ,\rho _{g_{2}}^{\left( p,q\right)
}\left( f_{2},\varphi \right) \right\} \right]
\end{eqnarray*}%
when the following two conditions holds:\newline
$\left( i\right) $ $\rho _{g_{i}}^{\left( p,q\right) }\left( f_{1},\varphi
\right) <\rho _{g_{j}}^{\left( p,q\right) }\left( f_{1},\varphi \right) $
with at least $f_{1}$ is of regular relative $\left( p,q\right) $-$\varphi $
growth with respect to $g_{j}$ for $i$ $=$ $1,$ $2$, $j$ $=$ $1,2$ and $%
i\neq j$; and\newline
$\left( ii\right) $ $\rho _{g_{i}}^{\left( p,q\right) }\left( f_{2},\varphi
\right) <\rho _{g_{j}}^{\left( p,q\right) }\left( f_{2},\varphi \right) $
with at least $f_{2}$ is of regular relative $\left( p,q\right) $-$\varphi $
growth with respect to $g_{j}$ for $i$ $=$ $1,$ $2$, $j$ $=$ $1,2$ and $%
i\neq j$.\newline
The equality holds when $\rho _{g_{1}}^{\left( p,q\right) }\left(
f_{i},\varphi \right) <\rho _{g_{1}}^{\left( p,q\right) }\left(
f_{j},\varphi \right) $ and $\rho _{g_{2}}^{\left( p,q\right) }\left(
f_{i},\varphi \right) <\rho _{g_{2}}^{\left( p,q\right) }\left(
f_{j},\varphi \right) $ holds simultaneously for $i,1,2;$ $j=1,2\ $and $%
i\neq j.$
\end{theorem}

\begin{proof}
Let the conditions $\left( i\right) $ and $\left( ii\right) $ of the theorem
hold. Therefore in view of Theorem \ref{t9.1x} and Theorem \ref{t9.4} we get
that%
\begin{align}
& \max \left[ \min \left\{ \rho _{g_{1}}^{\left( p,q\right) }\left(
f_{1},\varphi \right) ,\rho _{g_{2}}^{\left( p,q\right) }\left(
f_{1},\varphi \right) \right\} ,\min \left\{ \rho _{g_{1}}^{\left(
p,q\right) }\left( f_{2},\varphi \right) ,\rho _{g_{2}}^{\left( p,q\right)
}\left( f_{2},\varphi \right) \right\} \right]  \notag \\
& =\max \left[ \rho _{g_{1}\pm g_{2}}^{\left( p,q\right) }\left(
f_{1},\varphi \right) ,\rho _{g_{1}\pm g_{2}}^{\left( p,q\right) }\left(
f_{2},\varphi \right) \right]  \notag \\
& \geq \rho _{g_{1}\pm g_{2}}^{\left( p,q\right) }\left( f_{1}\pm
f_{2},\varphi \right) ~.  \label{9.xyz}
\end{align}

\qquad Since $\rho _{g_{1}}^{\left( p,q\right) }\left( f_{i},\varphi \right)
<\rho _{g_{1}}^{\left( p,q\right) }\left( f_{j},\varphi \right) $ and $\rho
_{g_{2}}^{\left( p,q\right) }\left( f_{i},\varphi \right) <\rho
_{g_{2}}^{\left( p,q\right) }\left( f_{j},\varphi \right) $ hold
simultaneously for $i=1,2;$ $j=1,2\ $and $i\neq j,$ we obtain that%
\begin{equation*}
\text{either }\min \left\{ \rho _{g_{1}}^{\left( p,q\right) }\left(
f_{1},\varphi \right) ,\rho _{g_{2}}^{\left( p,q\right) }\left(
f_{1},\varphi \right) \right\} >\min \left\{ \rho _{g_{1}}^{\left(
p,q\right) }\left( f_{2},\varphi \right) ,\rho _{g_{2}}^{\left( p,q\right)
}\left( f_{2},\varphi \right) \right\} \text{ or}
\end{equation*}%
\begin{equation*}
\min \left\{ \rho _{g_{1}}^{\left( p,q\right) }\left( f_{2},\varphi \right)
,\rho _{g_{2}}^{\left( p,q\right) }\left( f_{2},\varphi \right) \right\}
>\min \left\{ \rho _{g_{1}}^{\left( p,q\right) }\left( f_{1},\varphi \right)
,\rho _{g_{2}}^{\left( p,q\right) }\left( f_{1},\varphi \right) \right\} 
\text{ holds.}
\end{equation*}

\qquad Now in view of the conditions $\left( i\right) $ and $\left(
ii\right) $ of the theorem, it follows from above that%
\begin{equation*}
\text{either }\rho _{g_{1}\pm g_{2}}^{\left( p,q\right) }\left(
f_{1},\varphi \right) >\rho _{g_{1}\pm g_{2}}^{\left( p,q\right) }\left(
f_{2},\varphi \right) \text{ or }\rho _{g_{1}\pm g_{2}}^{\left( p,q\right)
}\left( f_{2},\varphi \right) >\rho _{g_{1}\pm g_{2}}^{\left( p,q\right)
}\left( f_{1},\varphi \right)
\end{equation*}%
which is the condition for holding equality in $\left( \ref{9.xyz}\right) $.

\qquad Hence the theorem follows.
\end{proof}

\begin{theorem}
\label{t9.6} Let $f_{1},f_{2}$ be any two meromorphic functions and $g_{1}$, 
$g_{2}$ be any two entire functions. Also let $g_{1},g_{2}$ and $g_{1}\pm
g_{2}$ satisfy the Property (A). Then%
\begin{align*}
& \lambda _{g_{1}\pm g_{2}}^{\left( p,q\right) }\left( f_{1}\pm
f_{2},\varphi \right) \\
& \geq \min \left[ \max \left\{ \lambda _{g_{1}}^{\left( p,q\right) }\left(
f_{1},\varphi \right) ,\lambda _{g_{1}}^{\left( p,q\right) }\left(
f_{2},\varphi \right) \right\} ,\max \left\{ \lambda _{g_{2}}^{\left(
p,q\right) }\left( f_{1},\varphi \right) ,\lambda _{g_{2}}^{\left(
p,q\right) }\left( f_{2},\varphi \right) \right\} \right]
\end{align*}%
when the following two conditions holds:\newline
$\left( i\right) $ $\lambda _{g_{1}}^{\left( p,q\right) }\left(
f_{i},\varphi \right) >\lambda _{g_{1}}^{\left( p,q\right) }\left(
f_{j},\varphi \right) $ with at least $f_{j}$ is of regular relative $\left(
p,q\right) $-$\varphi $ growth with respect to $g_{1}$ for $i$ $=$ $1,$ $2$, 
$j$ $=$ $1,2$ and $i\neq j$; and\newline
$\left( ii\right) $ $\lambda _{g_{2}}^{\left( p,q\right) }\left(
f_{i},\varphi \right) >\lambda _{g_{2}}^{\left( p,q\right) }\left(
f_{j},\varphi \right) $ with at least $f_{j}$ is of regular relative $\left(
p,q\right) $-$\varphi $ growth with respect to $g_{2}$ for $i$ $=$ $1,$ $2$, 
$j$ $=$ $1,2$ and $i\neq j$.\newline
The equality holds when $\lambda _{g_{i}}^{\left( p,q\right) }\left(
f_{1},\varphi \right) <\lambda _{g_{j}}^{\left( p,q\right) }\left(
f_{1},\varphi \right) $ and $\lambda _{g_{i}}^{\left( p,q\right) }\left(
f_{2},\varphi \right) <\lambda _{g_{j}}^{\left( p,q\right) }\left(
f_{2},\varphi \right) $ hold simultaneously for $i=1,2;$ $j=1,2\ $and $i\neq
j.$
\end{theorem}

\begin{proof}
Suppose that the conditions $\left( i\right) $ and $\left( ii\right) $ of
the theorem holds. Therefore in view of Theorem \ref{t9.2x} and Theorem \ref%
{t9.3}, we obtain that%
\begin{align}
& \min \left[ \max \left\{ \lambda _{g_{1}}^{\left( p,q\right) }\left(
f_{1},\varphi \right) ,\lambda _{g_{1}}^{\left( p,q\right) }\left(
f_{2},\varphi \right) \right\} ,\max \left\{ \lambda _{g_{2}}^{\left(
p,q\right) }\left( f_{1},\varphi \right) ,\lambda _{g_{2}}^{\left(
p,q\right) }\left( f_{2},\varphi \right) \right\} \right]  \notag \\
& =\min \left[ \lambda _{g_{1}}^{\left( p,q\right) }\left( f_{1}\pm
f_{2},\varphi \right) ,\lambda _{g_{2}}^{\left( p,q\right) }\left( f_{1}\pm
f_{2},\varphi \right) \right]  \notag \\
& \geq \lambda _{g_{1}\pm g_{2}}^{\left( p,q\right) }\left( f_{1}\pm
f_{2},\varphi \right) ~.  \label{9.xya}
\end{align}

\qquad Since $\lambda _{g_{i}}^{\left( p,q\right) }\left( f_{1},\varphi
\right) <\lambda _{g_{j}}^{\left( p,q\right) }\left( f_{1},\varphi \right) $
and $\lambda _{g_{i}}^{\left( p,q\right) }\left( f_{2},\varphi \right)
<\lambda _{g_{j}}^{\left( p,q\right) }\left( f_{2},\varphi \right) $ holds
simultaneously for $i=1,2;$ $j=1,2\ $and $i\neq j$, we get that%
\begin{equation*}
\text{either }\max \left\{ \lambda _{g_{1}}^{\left( p,q\right) }\left(
f_{1},\varphi \right) ,\lambda _{g_{1}}^{\left( p,q\right) }\left(
f_{2},\varphi \right) \right\} <\max \left\{ \lambda _{g_{2}}^{\left(
p,q\right) }\left( f_{1},\varphi \right) ,\lambda _{g_{2}}^{\left(
p,q\right) }\left( f_{2},\varphi \right) \right\} \text{ or}
\end{equation*}%
\begin{equation*}
\max \left\{ \lambda _{g_{2}}^{\left( p,q\right) }\left( f_{1},\varphi
\right) ,\lambda _{g_{2}}^{\left( p,q\right) }\left( f_{2},\varphi \right)
\right\} <\max \left\{ \lambda _{g_{1}}^{\left( p,q\right) }\left(
f_{1},\varphi \right) ,\lambda _{g_{1}}^{\left( p,q\right) }\left(
f_{2},\varphi \right) \right\} \text{ holds.}
\end{equation*}

\qquad Since condition $\left( i\right) $ and $\left( ii\right) $ of the
theorem holds, it follows from above that 
\begin{equation*}
\text{either }\lambda _{g_{1}}^{\left( p,q\right) }\left( f_{1}\pm
f_{2},\varphi \right) <\lambda _{g_{2}}^{\left( p,q\right) }\left( f_{1}\pm
f_{2},\varphi \right) \text{ or }\lambda _{g_{2}}^{\left( p,q\right) }\left(
f_{1}\pm f_{2},\varphi \right) <\lambda _{g_{1}}^{\left( p,q\right) }\left(
f_{1}\pm f_{2},\varphi \right)
\end{equation*}%
which is the condition for holding equality in $\left( \ref{9.xya}\right) $.

\qquad Hence the theorem follows.
\end{proof}

\begin{theorem}
\label{t9.8} Let $f_{1}$, $f_{2}$ be any two meromorphic functions and $%
g_{1} $ be any entire function such that at least $f_{1}$ or $f_{2}$ is of
regular relative $\left( p,q\right) $-$\varphi $ growth with respect to $%
g_{1}$. Also let $g_{1}$ satisfy the Property (A). Then%
\begin{equation*}
\lambda _{g_{1}}^{\left( p,q\right) }\left( f_{1}\cdot f_{2},\varphi \right)
\leq \max \left\{ \lambda _{g_{1}}^{\left( p,q\right) }\left( f_{1},\varphi
\right) ,\lambda _{g_{1}}^{\left( p,q\right) }\left( f_{2},\varphi \right)
\right\} ~.
\end{equation*}%
The equality holds when $\lambda _{g_{1}}^{\left( p,q\right) }\left(
f_{i},\varphi \right) >\lambda _{g_{1}}^{\left( p,q\right) }\left(
f_{j},\varphi \right) $ with at least $f_{j}$ is of regular relative $\left(
p,q\right) $-$\varphi $ growth with respect to $g_{1}$ where $i,j=1,2$ and $%
i\neq j$.
\end{theorem}

\begin{proof}
Since $T_{f_{1}\cdot f_{2}}\left( r\right) \leq T_{f_{1}}\left( r\right)
+T_{f_{2}}\left( r\right) $ for all large $r,$ therefore applying the same
procedure as adopted in Theorem \ref{t9.2x} we get that%
\begin{equation*}
\lambda _{g_{1}}^{\left( p,q\right) }\left( f_{1}\cdot f_{2},\varphi \right)
\leq \max \left\{ \lambda _{g_{1}}^{\left( p,q\right) }\left( f_{1},\varphi
\right) ,\lambda _{g_{1}}^{\left( p,q\right) }\left( f_{2},\varphi \right)
\right\} ~.
\end{equation*}

Now without loss of any generality, let $\lambda _{g_{1}}^{\left( p,q\right)
}\left( f_{1},\varphi \right) <\lambda _{g_{1}}^{\left( p,q\right) }\left(
f_{2},\varphi \right) $ and $f=f_{1}\cdot f_{2}.$ Then $\lambda
_{g_{1}}^{\left( p,q\right) }\left( f,\varphi \right) \leq \lambda
_{g_{1}}^{\left( p,q\right) }\left( f_{2},\varphi \right) .$ Further, $f_{2}=%
\frac{f}{f_{1}}$ and $T_{f_{1}}\left( r\right) =T_{\frac{1}{f_{1}}}\left(
r\right) +O(1).$ Therefore $T_{f_{2}}\left( r\right) \leq T_{f}\left(
r\right) +T_{f_{1}}\left( r\right) +O(1)$ and in this case we obtain that $%
\lambda _{g_{1}}^{\left( p,q\right) }\left( f_{2},\varphi \right) \leq \max
\left\{ \lambda _{g_{1}}^{\left( p,q\right) }\left( f,\varphi \right)
,\lambda _{g_{1}}^{\left( p,q\right) }\left( f_{1},\varphi \right) \right\}
~.$ As we assume that $\lambda _{g_{1}}^{\left( p,q\right) }\left(
f_{1},\varphi \right) <\lambda _{g_{1}}^{\left( p,q\right) }\left(
f_{2},\varphi \right) ,$ therefore we have $\lambda _{g_{1}}^{\left(
p,q\right) }\left( f_{2},\varphi \right) \leq \lambda _{g_{1}}^{\left(
p,q\right) }\left( f,\varphi \right) $ and hence $\lambda _{g_{1}}^{\left(
p,q\right) }\left( f,\varphi \right) $ $=$ $\lambda _{g_{1}}^{\left(
p,q\right) }\left( f_{2},\varphi \right) $ $=$ $\max $ $\{$ $\lambda
_{g_{1}}^{\left( p,q\right) }\left( f_{1},\varphi \right) ,$ $\lambda
_{g_{1}}^{\left( p,q\right) }\left( f_{2},\varphi \right) $ $\}.$ Therefore, 
$\lambda _{g_{1}}^{\left( p,q\right) }\left( f_{1}\cdot f_{2},\varphi
\right) =\lambda _{g_{1}}^{\left( p,q\right) }\left( f_{i},\varphi \right)
\mid i=1,2$ provided $\lambda _{g_{1}}^{\left( p,q\right) }\left(
f_{1},\varphi \right) \neq \lambda _{g_{1}}^{\left( p,q\right) }\left(
f_{2},\varphi \right) .$

\qquad Hence the theorem follows.
\end{proof}

\qquad Next we prove the result for the quotient $\frac{f_{1}}{f_{2}},$
provided $\frac{f_{1}}{f_{2}}$ is meromorphic.

\begin{theorem}
\label{t9.8 A} Let $f_{1}$, $f_{2}$ be any two meromorphic functions and $%
g_{1}$ be any entire function such that at least $f_{1}$ or $f_{2}$ is of
regular relative $\left( p,q\right) $-$\varphi $ growth with respect to $%
g_{1}$. Also let $g_{1}$ satisfy the Property (A). Then%
\begin{equation*}
\lambda _{g_{1}}^{\left( p,q\right) }\left( \frac{f_{1}}{f_{2}},\varphi
\right) \leq \max \left\{ \lambda _{g_{1}}^{\left( p,q\right) }\left(
f_{1},\varphi \right) ,\lambda _{g_{1}}^{\left( p,q\right) }\left(
f_{2},\varphi \right) \right\} ,
\end{equation*}%
provided $\frac{f_{1}}{f_{2}}$ is meromorphic. The equality holds when at
least $f_{2}$ is of regular relative $\left( p,q\right) $-$\varphi $ growth
with respect to $g_{1}$ and $\lambda _{g_{1}}^{\left( p,q\right) }\left(
f_{1},\varphi \right) \neq \lambda _{g_{1}}^{\left( p,q\right) }\left(
f_{2},\varphi \right) $.
\end{theorem}

\begin{proof}
Since $T_{_{f_{2}}}\left( r\right) =T_{_{\frac{1}{f_{2}}}}\left( r\right)
+O(1)$ and $T_{_{\frac{f_{1}}{f_{2}}}}\left( r\right) \leq
T_{_{f_{1}}}\left( r\right) +T_{_{\frac{1}{f_{2}}}}\left( r\right) ,$ we get
in view of Theorem \ref{t9.2x} that%
\begin{equation}
\lambda _{g_{1}}^{\left( p,q\right) }\left( \frac{f_{1}}{f_{2}},\varphi
\right) \leq \max \left\{ \lambda _{g_{1}}^{\left( p,q\right) }\left(
f_{1},\varphi \right) ,\lambda _{g_{1}}^{\left( p,q\right) }\left(
f_{2},\varphi \right) \right\} ~.  \label{50.3}
\end{equation}

\qquad Now in order to prove the equality conditions, we discuss the
following two cases:\medskip \newline
\textbf{Case I. }Suppose $\frac{f_{1}}{f_{2}}\left( =h\right) $ satisfies
the following condition%
\begin{equation*}
\lambda _{g_{1}}^{\left( p,q\right) }\left( f_{1},\varphi \right) <\lambda
_{g_{1}}^{\left( p,q\right) }\left( f_{2},\varphi \right) ,
\end{equation*}%
and $f_{2}$ is of regular relative $\left( p,q\right) $-$\varphi $ growth
with respect to $g_{1}.$

\qquad Now if possible, let $\lambda _{g_{1}}^{\left( p,q\right) }\left( 
\frac{f_{1}}{f_{2}},\varphi \right) <\lambda _{g_{1}}^{\left( p,q\right)
}\left( f_{2},\varphi \right) $. Therefore from $f_{1}=h\cdot f_{2}$ we get
that $\lambda _{g_{1}}^{\left( p,q\right) }\left( f_{1},\varphi \right)
=\lambda _{g_{1}}^{\left( p,q\right) }\left( f_{2},\varphi \right) $ which
is a contradiction. Therefore $\lambda _{g_{1}}^{\left( p,q\right) }\left( 
\frac{f_{1}}{f_{2}},\varphi \right) \geq \lambda _{g_{1}}^{\left( p,q\right)
}\left( f_{2},\varphi \right) $ and in view of $\left( \ref{50.3}\right) $,
we get that%
\begin{equation*}
\lambda _{g_{1}}^{\left( p,q\right) }\left( \frac{f_{1}}{f_{2}},\varphi
\right) =\lambda _{g_{1}}^{\left( p,q\right) }\left( f_{2},\varphi \right) ~.
\end{equation*}%
\medskip \newline
\textbf{Case II. } Suppose $\frac{f_{1}}{f_{2}}\left( =h\right) $ satisfies
the following condition%
\begin{equation*}
\text{ }\lambda _{g_{1}}^{\left( p,q\right) }\left( f_{1},\varphi \right)
>\lambda _{g_{1}}^{\left( p,q\right) }\left( f_{2},\varphi \right) ,
\end{equation*}%
and $f_{2}$ is of regular relative $\left( p,q\right) $-$\varphi $ growth
with respect to $g_{1}.$

\qquad Now from $f_{1}=h\cdot f_{2}$ we get that either $\lambda
_{g_{1}}^{\left( p,q\right) }\left( f_{1},\varphi \right) \leq \lambda
_{g_{1}}^{\left( p,q\right) }\left( \frac{f_{1}}{f_{2}},\varphi \right) $ or 
$\lambda _{g_{1}}^{\left( p,q\right) }\left( f_{1},\varphi \right) \leq
\lambda _{g_{1}}^{\left( p,q\right) }\left( f_{2},\varphi \right) $. But
according to our assumption $\lambda _{g_{1}}^{\left( p,q\right) }\left(
f_{1},\varphi \right) \nleq \lambda _{g_{1}}^{\left( p,q\right) }\left(
f_{2},\varphi \right) $. Therefore $\lambda _{g_{1}}^{\left( p,q\right)
}\left( \frac{f_{1}}{f_{2}},\varphi \right) \geq \lambda _{g_{1}}^{\left(
p,q\right) }\left( f_{1},\varphi \right) $ and in view of $\left( \ref{50.3}%
\right) $, we get that%
\begin{equation*}
\lambda _{g_{1}}^{\left( p,q\right) }\left( \frac{f_{1}}{f_{2}},\varphi
\right) =\lambda _{g_{1}}^{\left( p,q\right) }\left( f_{1},\varphi \right) ~.
\end{equation*}

\qquad Hence the theorem follows.
\end{proof}

\qquad Now we state the following theorem which can easily be carried out in
the line of Theorem \ref{t9.8} and Theorem \ref{t9.8 A} and therefore its
proof is omitted.

\begin{theorem}
\label{t9.7} Let $f_{1}$ and $f_{2}$ be any two meromorphic functions and $%
g_{1}$ be any entire function such that such that $\rho _{g_{1}}^{\left(
p,q\right) }\left( f_{1},\varphi \right) $ and $\rho _{g_{2}}^{\left(
p,q\right) }\left( f_{1},\varphi \right) $ exists. Also let $g_{1}$ satisfy
the Property (A). Then%
\begin{equation*}
\rho _{g_{1}}^{\left( p,q\right) }\left( f_{1}\cdot f_{2},\varphi \right)
\leq \max \left\{ \rho _{g_{1}}^{\left( p,q\right) }\left( f_{1},\varphi
\right) ,\rho _{g_{1}}^{\left( p,q\right) }\left( f_{2},\varphi \right)
\right\} ~.
\end{equation*}%
The equality holds when $\rho _{g_{1}}^{\left( p,q\right) }\left(
f_{1},\varphi \right) \neq \rho _{g_{1}}^{\left( p,q\right) }\left(
f_{2},\varphi \right) $. Similar results hold for the quotient $\frac{f_{1}}{%
f_{2}}$, provided $\frac{f_{1}}{f_{2}}$ is meromorphic.
\end{theorem}

\begin{theorem}
\label{t9.9} Let $f_{1}$ be a meromorphic function and $g_{1}$, $g_{2}$ be
any two entire functions such that $\lambda _{g_{1}}^{\left( p,q\right)
}\left( f_{1},\varphi \right) $ and $\lambda _{g_{2}}^{\left( p,q\right)
}\left( f_{1},\varphi \right) $ exists. Also let $g_{1}\cdot g_{2}$ satisfy
the Property (A). Then%
\begin{equation*}
\lambda _{g_{1}\cdot g_{2}}^{\left( p,q\right) }\left( f_{1},\varphi \right)
\geq \min \left\{ \lambda _{g_{1}}^{\left( p,q\right) }\left( f_{1},\varphi
\right) ,\lambda _{g_{2}}^{\left( p,q\right) }\left( f_{1},\varphi \right)
\right\} ~.
\end{equation*}%
The equality holds when $\lambda _{g_{i}}^{\left( p,q\right) }\left(
f_{1},\varphi \right) <\lambda _{g_{j}}^{\left( p,q\right) }\left(
f_{1},\varphi \right) $ where $i,j=1,2$ and $i\neq j$ and $g_{i}$ satisfy
the Property (A). Similar results hold for the quotient $\frac{g_{1}}{g_{2}}$%
, provided $\frac{g_{1}}{g_{2}}$ is entire and satisfy the Property (A). The
equality holds when $\lambda _{g_{1}}^{\left( p,q\right) }\left(
f_{1},\varphi \right) \neq \lambda _{g_{2}}^{\left( p,q\right) }\left(
f_{1},\varphi \right) $ and $g_{1}$ satisfy the Property (A).
\end{theorem}

\begin{proof}
Since $T_{g_{1}\cdot g_{2}}\left( r\right) \leq T_{g_{1}}\left( r\right)
+T_{g_{2}}\left( r\right) $ for all large $r,$ therefore applying the same
procedure as adopted in Theorem \ref{t9.3} we get that%
\begin{equation*}
\lambda _{g_{1}\cdot g_{2}}^{\left( p,q\right) }\left( f_{1},\varphi \right)
\geq \min \left\{ \lambda _{g_{1}}^{\left( p,q\right) }\left( f_{1},\varphi
\right) ,\lambda _{g_{2}}^{\left( p,q\right) }\left( f_{1},\varphi \right)
\right\} ~.
\end{equation*}

\qquad Now without loss of any generality, we may consider that $\lambda
_{g_{1}}^{\left( p,q\right) }\left( f_{1},\varphi \right) <\lambda
_{g_{2}}^{\left( p,q\right) }\left( f_{1},\varphi \right) $ and $%
g=g_{1}\cdot g_{2}.$ Then $\lambda _{g}^{\left( p,q\right) }\left(
f_{1},\varphi \right) \geq \lambda _{g_{1}}^{\left( p,q\right) }\left(
f_{1},\varphi \right) .$ Further, $g_{1}=\frac{g}{g_{2}}$ and and $%
T_{g_{2}}\left( r\right) =T_{\frac{1}{g_{2}}}\left( r\right) +O(1).$
Therefore $T_{g_{1}}\left( r\right) $ $\leq $ $T_{g}\left( r\right) $ $+$ $%
T_{g_{2}}\left( r\right) $ $+$ $O(1)$ and in this case we obtain that $%
\lambda _{g_{1}}^{\left( p,q\right) }\left( f_{1},\varphi \right) \geq \min
\left\{ \lambda _{g}^{\left( p,q\right) }\left( f_{1},\varphi \right)
,\lambda _{g_{2}}^{\left( p,q\right) }\left( f_{1},\varphi \right) \right\}
~.$ As we assume that $\lambda _{g_{1}}^{\left( p,q\right) }\left(
f_{1},\varphi \right) <\lambda _{g_{2}}^{\left( p,q\right) }\left(
f_{1},\varphi \right) ,$ so we have $\lambda _{g_{1}}^{\left( p,q\right)
}\left( f_{1},\varphi \right) \geq \lambda _{g}^{\left( p,q\right) }\left(
f_{1},\varphi \right) $ and hence $\lambda _{g}^{\left( p,q\right) }\left(
f_{1},\varphi \right) $ $=$ $\lambda _{g_{1}}^{\left( p,q\right) }\left(
f_{1},\varphi \right) $ $=$ $\min $ $\left\{ \lambda _{g_{1}}^{\left(
p,q\right) }\left( f_{1},\varphi \right) ,\lambda _{g_{2}}^{\left(
p,q\right) }\left( f_{1},\varphi \right) \right\} .$ Therefore, $\lambda
_{g_{1}\cdot g_{2}}^{\left( p,q\right) }\left( f_{1},\varphi \right)
=\lambda _{g_{i}}^{\left( p,q\right) }\left( f_{1},\varphi \right) \mid
i=1,2 $ provided $\lambda _{g_{1}}^{\left( p,q\right) }\left( f_{1},\varphi
\right) <\lambda _{g_{2}}^{\left( p,q\right) }\left( f_{1},\varphi \right) $
and $g_{1}$ satisfy the Property (A)$.$ Hence the first part of the theorem
follows.

\qquad Now we prove our results for the quotient $\frac{g_{1}}{g_{2}}$,
provided $\frac{g_{1}}{g_{2}}$ is entire and $\lambda _{g_{1}}^{\left(
p,q\right) }\left( f_{1},\varphi \right) \neq \lambda _{g_{2}}^{\left(
p,q\right) }\left( f_{1},\varphi \right) $. Since $T_{_{g_{2}}}\left(
r\right) =T_{_{\frac{1}{g_{2}}}}\left( r\right) +O(1)$ and $T_{_{\frac{g_{1}%
}{g_{2}}}}\left( r\right) \leq T_{_{g_{1}}}\left( r\right) +T_{_{\frac{1}{%
g_{2}}}}\left( r\right) ,$ we get in view of Theorem \ref{t9.3} that%
\begin{equation}
\lambda _{\frac{g_{1}}{g_{2}}}^{\left( p,q\right) }\left( f_{1},\varphi
\right) \geq \min \left\{ \lambda _{g_{1}}^{\left( p,q\right) }\left(
f_{1},\varphi \right) ,\lambda _{g_{2}}^{\left( p,q\right) }\left(
f_{1},\varphi \right) \right\} ~.  \label{50.11}
\end{equation}

\qquad Now in order to prove the equality conditions, we discuss the
following two cases:\medskip \newline
\textbf{Case I. }Suppose $\frac{g_{1}}{g_{2}}\left( =h\right) $ satisfies
the following condition%
\begin{equation*}
\lambda _{g_{1}}^{\left( p,q\right) }\left( f_{1},\varphi \right) >\lambda
_{g_{2}}^{\left( p,q\right) }\left( f_{1},\varphi \right) ~.
\end{equation*}

\qquad Now if possible, let $\lambda _{\frac{g_{1}}{g_{2}}}^{\left(
p,q\right) }\left( f_{1},\varphi \right) >\lambda _{g_{2}}^{\left(
p,q\right) }\left( f_{1},\varphi \right) $. Therefore from $g_{1}=h\cdot
g_{2}$ we get that $\lambda _{g_{1}}^{\left( p,q\right) }\left(
f_{1},\varphi \right) =\lambda _{g_{2}}^{\left( p,q\right) }\left(
f_{1},\varphi \right) $, which is a contradiction. Therefore $\lambda _{%
\frac{g_{1}}{g_{2}}}^{\left( p,q\right) }\left( f_{1},\varphi \right) \leq
\lambda _{g_{2}}^{\left( p,q\right) }\left( f_{1},\varphi \right) $ and in
view of $\left( \ref{50.11}\right) $, we get that%
\begin{equation*}
\lambda _{\frac{g_{1}}{g_{2}}}^{\left( p,q\right) }\left( f_{1},\varphi
\right) =\lambda _{g_{2}}^{\left( p,q\right) }\left( f_{1},\varphi \right) ~.
\end{equation*}%
\medskip \newline
\textbf{Case II. } Suppose that $\frac{g_{1}}{g_{2}}\left( =h\right) $
satisfies the following condition%
\begin{equation*}
\lambda _{g_{1}}^{\left( p,q\right) }\left( f_{1},\varphi \right) <\lambda
_{g_{2}}^{\left( p,q\right) }\left( f_{1},\varphi \right) ~.
\end{equation*}

\qquad Therefore from $g_{1}=h\cdot g_{2}$, we get that either $\lambda
_{g_{1}}^{\left( p,q\right) }\left( f_{1},\varphi \right) \geq \lambda _{%
\frac{g_{1}}{g_{2}}}^{\left( p,q\right) }\left( f_{1},\varphi \right) $ or $%
\lambda _{g_{1}}^{\left( p,q\right) }\left( f_{1},\varphi \right) \geq
\lambda _{g_{2}}^{\left( p,q\right) }\left( f_{1},\varphi \right) $. But
according to our assumption $\lambda _{g_{1}}^{\left( p,q\right) }\left(
f_{1},\varphi \right) \ngeqslant \lambda _{g_{2}}^{\left( p,q\right) }\left(
f_{1},\varphi \right) $. Therefore $\lambda _{\frac{g_{1}}{g_{2}}}^{\left(
p,q\right) }\left( f_{1},\varphi \right) \leq \lambda _{g_{1}}^{\left(
p,q\right) }\left( f_{1},\varphi \right) $ and in view of $\left( \ref{50.11}%
\right) $, we get that%
\begin{equation*}
\lambda _{\frac{g_{1}}{g_{2}}}^{\left( p,q\right) }\left( f_{1},\varphi
\right) =\lambda _{g_{1}}^{\left( p,q\right) }\left( f_{1},\varphi \right) ~.
\end{equation*}

\qquad Hence the theorem follows.
\end{proof}

\begin{theorem}
\label{t9.10} Let $f_{1}$ be any meromorphic function and $g_{1}$, $g_{2}$
be any two entire functions such that $\rho _{g_{1}}^{\left( p,q\right)
}\left( f_{1},\varphi \right) $ and $\rho _{g_{2}}^{\left( p,q\right)
}\left( f_{1},\varphi \right) $ exists. Further let $f_{1}$ is of regular
relative $\left( p,q\right) $-$\varphi $ growth with respect to at least any
one of $g_{1}$ or $g_{2}.$ Also let $g_{1}\cdot g_{2}$ satisfy the Property
(A). Then%
\begin{equation*}
\rho _{g_{1}\cdot g_{2}}^{\left( p,q\right) }\left( f_{1},\varphi \right)
\geq \min \left\{ \rho _{g_{1}}^{\left( p,q\right) }\left( f_{1},\varphi
\right) ,\rho _{g_{2}}^{\left( p,q\right) }\left( f_{1},\varphi \right)
\right\} ~.
\end{equation*}%
The equality holds when $\rho _{g_{i}}^{\left( p,q\right) }\left(
f_{1},\varphi \right) <\rho _{g_{j}}^{\left( p,q\right) }\left(
f_{1},\varphi \right) $ with at least $f_{1}$ is of regular relative $\left(
p,q\right) $-$\varphi $ growth with respect to $g_{j}$ where $i,j=1,2$ and $%
i\neq j$ and $g_{i}$ satisfy the Property (A).
\end{theorem}

\begin{theorem}
\label{t9.10A} Let $f_{1}$ be any meromorphic function and $g_{1}$, $g_{2}$
be any two entire functions such that $\rho _{g_{1}}^{\left( p,q\right)
}\left( f_{1},\varphi \right) $ and $\rho _{g_{2}}^{\left( p,q\right)
}\left( f_{1},\varphi \right) $ exists. Further let $f_{1}$ is of regular
relative $\left( p,q\right) $-$\varphi $ growth with respect to at least any
one of $g_{1}$ or $g_{2}.$ Then%
\begin{equation*}
\rho _{\frac{g_{1}}{g_{2}}}^{\left( p,q\right) }\left( f_{1},\varphi \right)
\geq \min \left\{ \rho _{g_{1}}^{\left( p,q\right) }\left( f_{1},\varphi
\right) ,\rho _{g_{2}}^{\left( p,q\right) }\left( f_{1},\varphi \right)
\right\} ,
\end{equation*}%
provided $\frac{g_{1}}{g_{2}}$ is entire and satisfy the Property (A). The
equality holds when at least $f_{1}$ is of regular relative $\left(
p,q\right) $-$\varphi $ growth with respect to $g_{2}$, $\rho
_{g_{1}}^{\left( p,q\right) }\left( f_{1},\varphi \right) \neq \rho
_{g_{2}}^{\left( p,q\right) }\left( f_{1},\varphi \right) $ and $g_{1}$
satisfy the Property (A).
\end{theorem}

\qquad We omit the proof of Theorem \ref{t9.10} and Theorem \ref{t9.10A} as
those can easily be carried out in the line of Theorem \ref{t9.9}.

\qquad Now we state the following four theorems without their proofs as
those can easily be carried out in the line of Theorem \ref{t9.5} and
Theorem \ref{t9.6} respectively.

\begin{theorem}
\label{t9.11} Let $f_{1},f_{2}$ be any two meromorphic functions and $g_{1}$%
, $g_{2}$ be any two entire functions. Also let $g_{1}\cdot g_{2}$ be
satisfy the Property (A). Then%
\begin{eqnarray*}
&&\rho _{g_{1}\cdot g_{2}}^{\left( p,q\right) }\left( f_{1}\cdot
f_{2},\varphi \right) \\
&\leq &\max \left[ \min \left\{ \rho _{g_{1}}^{\left( p,q\right) }\left(
f_{1},\varphi \right) ,\rho _{g_{2}}^{\left( p,q\right) }\left(
f_{1},\varphi \right) \right\} ,\min \left\{ \rho _{g_{1}}^{\left(
p,q\right) }\left( f_{2},\varphi \right) ,\rho _{g_{2}}^{\left( p,q\right)
}\left( f_{2},\varphi \right) \right\} \right] ,
\end{eqnarray*}%
when the following two conditions holds:\newline
$\left( i\right) $ $\rho _{g_{i}}^{\left( p,q\right) }\left( f_{1},\varphi
\right) <\rho _{g_{j}}^{\left( p,q\right) }\left( f_{1},\varphi \right) $
with at least $f_{1}$ is of regular relative $\left( p,q\right) $-$\varphi $
growth with respect to $g_{j}$ and $g_{i}$ satisfy the Property (A) for $i$ $%
=$ $1,$ $2$, $j$ $=$ $1,2$ and $i\neq j$; and\newline
$\left( ii\right) $ $\rho _{g_{i}}^{\left( p,q\right) }\left( f_{2},\varphi
\right) <\rho _{g_{j}}^{\left( p,q\right) }\left( f_{2},\varphi \right) $
with at least $f_{2}$ is of regular relative $\left( p,q\right) $-$\varphi $
growth with respect to $g_{j}$ and $g_{i}$ satisfy the Property (A) for $i$ $%
=$ $1,$ $2$, $j$ $=$ $1,2$ and $i\neq j$.\newline
The quality holds when $\rho _{g_{1}}^{\left( p,q\right) }\left(
f_{i},\varphi \right) <\rho _{g_{1}}^{\left( p,q\right) }\left(
f_{j},\varphi \right) $ and $\rho _{g_{2}}^{\left( p,q\right) }\left(
f_{i},\varphi \right) <\rho _{g_{2}}^{\left( p,q\right) }\left(
f_{j},\varphi \right) $ holds simultaneously for $i=1,2;$ $j=1,2\ $and $%
i\neq j.$
\end{theorem}

\begin{theorem}
\label{t9.12} Let $f_{1},f_{2}$ be any two meromorphic functions and $g_{1}$%
, $g_{2}$ be any two entire functions. Also let $g_{1}\cdot g_{2}$, $g_{1}$
and $g_{2}$ be satisfy the Property (A). Then%
\begin{eqnarray*}
&&\lambda _{g_{1}\cdot g_{2}}^{\left( p,q\right) }\left( f_{1}\cdot
f_{2},\varphi \right) \\
&\geq &\min \left[ \max \left\{ \lambda _{g_{1}}^{\left( p,q\right) }\left(
f_{1},\varphi \right) ,\lambda _{g_{1}}^{\left( p,q\right) }\left(
f_{2},\varphi \right) \right\} ,\max \left\{ \lambda _{g_{2}}^{\left(
p,q\right) }\left( f_{1},\varphi \right) ,\lambda _{g_{2}}^{\left(
p,q\right) }\left( f_{2},\varphi \right) \right\} \right]
\end{eqnarray*}%
when the following two conditions holds:\newline
$\left( i\right) $ $\lambda _{g_{1}}^{\left( p,q\right) }\left(
f_{i},\varphi \right) >\lambda _{g_{1}}^{\left( p,q\right) }\left(
f_{j},\varphi \right) $ with at least $f_{j}$ is of regular relative $\left(
p,q\right) $-$\varphi $ growth with respect to $g_{1}$ for $i$ $=$ $1,$ $2$, 
$j$ $=$ $1,2$ and $i\neq j$; and\newline
$\left( ii\right) $ $\lambda _{g_{2}}^{\left( p,q\right) }\left(
f_{i},\varphi \right) >\lambda _{g_{2}}^{\left( p,q\right) }\left(
f_{j},\varphi \right) $ with at least $f_{j}$ is of regular relative $\left(
p,q\right) $-$\varphi $ growth with respect to $g_{2}$ for $i$ $=$ $1,$ $2$, 
$j$ $=$ $1,2$ and $i\neq j$.\newline
The equality holds when $\lambda _{g_{i}}^{\left( p,q\right) }\left(
f_{1},\varphi \right) <\lambda _{g_{j}}^{\left( p,q\right) }\left(
f_{1},\varphi \right) $ and $\lambda _{g_{i}}^{\left( p,q\right) }\left(
f_{2},\varphi \right) <\lambda _{g_{j}}^{\left( p,q\right) }\left(
f_{2},\varphi \right) $ holds simultaneously for $i=1,2;$ $j=1,2\ $and $%
i\neq j.$
\end{theorem}

\begin{theorem}
\label{t9.11A} Let $f_{1},f_{2}$ be any two meromorphic functions and $g_{1}$%
, $g_{2}$ be any two entire functions such that $\frac{f_{1}}{f_{2}}$ is
meromorphic and $\frac{g_{1}}{g_{2}}$ is entire. Also let $\frac{g_{1}}{g_{2}%
}$ satisfy the Property (A). Then%
\begin{eqnarray*}
&&\rho _{\frac{g_{1}}{g_{2}}}^{\left( p,q\right) }\left( \frac{f_{1}}{f_{2}}%
,\varphi \right) \\
&\leq &\max \left[ \min \left\{ \rho _{g_{1}}^{\left( p,q\right) }\left(
f_{1},\varphi \right) ,\rho _{g_{2}}^{\left( p,q\right) }\left(
f_{1},\varphi \right) \right\} ,\min \left\{ \rho _{g_{1}}^{\left(
p,q\right) }\left( f_{2},\varphi \right) ,\rho _{g_{2}}^{\left( p,q\right)
}\left( f_{2},\varphi \right) \right\} \right]
\end{eqnarray*}%
when the following two conditions holds:\newline
$\left( i\right) $ At least $f_{1}$ is of regular relative $\left(
p,q\right) $-$\varphi $ growth with respect to $g_{2}$ and $\rho
_{g_{1}}^{\left( p,q\right) }\left( f_{1},\varphi \right) \neq \rho
_{g_{2}}^{\left( p,q\right) }\left( f_{1},\varphi \right) $; and\newline
$\left( ii\right) $ At least $f_{2}$ is of regular relative $\left(
p,q\right) $-$\varphi $ growth with respect to $g_{2}$ and $\rho
_{g_{1}}^{\left( p,q\right) }\left( f_{2},\varphi \right) \neq \rho
_{g_{2}}^{\left( p,q\right) }\left( f_{2},\varphi \right) $.\newline
The equality holds when $\rho _{g_{1}}^{\left( p,q\right) }\left(
f_{i},\varphi \right) <\rho _{g_{1}}^{\left( p,q\right) }\left(
f_{j},\varphi \right) $ and $\rho _{g_{2}}^{\left( p,q\right) }\left(
f_{i},\varphi \right) <\rho _{g_{2}}^{\left( p,q\right) }\left(
f_{j},\varphi \right) $ holds simultaneously for $i=1,2;$ $j=1,2\ $and $%
i\neq j.$
\end{theorem}

\begin{theorem}
\label{t9.12A} Let $f_{1},f_{2}$ be any two meromorphic functions and $g_{1}$%
, $g_{2}$ be any two entire functions such that $\frac{f_{1}}{f_{2}}$ is
meromorphic and $\frac{g_{1}}{g_{2}}$ is entire. Also let $\frac{g_{1}}{g_{2}%
}$, $g_{1}$ and $g_{2}$ be satisfy the Property (A). Then%
\begin{eqnarray*}
&&\lambda _{\frac{g_{1}}{g_{2}}}^{\left( p,q\right) }\left( \frac{f_{1}}{%
f_{2}},\varphi \right) \\
&\geq &\min \left[ \max \left\{ \lambda _{g_{1}}^{\left( p,q\right) }\left(
f_{1},\varphi \right) ,\lambda _{g_{1}}^{\left( p,q\right) }\left(
f_{2},\varphi \right) \right\} ,\max \left\{ \lambda _{g_{2}}^{\left(
p,q\right) }\left( f_{1},\varphi \right) ,\lambda _{g_{2}}^{\left(
p,q\right) }\left( f_{2},\varphi \right) \right\} \right]
\end{eqnarray*}%
when the following two conditions hold:\newline
$\left( i\right) $ At least $f_{2}$ is of regular relative $\left(
p,q\right) $-$\varphi $ growth with respect to $g_{1}$ and $\lambda
_{g_{1}}^{\left( p,q\right) }\left( f_{1},\varphi \right) \neq \lambda
_{g_{1}}^{\left( p,q\right) }\left( f_{2},\varphi \right) $; and\newline
$\left( ii\right) $ At least $f_{2}$ is of regular relative $\left(
p,q\right) $-$\varphi $ growth with respect to $g_{2}$ and $\lambda
_{g_{2}}^{\left( p,q\right) }\left( f_{1},\varphi \right) \neq \lambda
_{g_{2}}^{\left( p,q\right) }\left( f_{2},\varphi \right) $.\newline
The equality holds when $\lambda _{g_{i}}^{\left( p,q\right) }\left(
f_{1},\varphi \right) <\lambda _{g_{j}}^{\left( p,q\right) }\left(
f_{1},\varphi \right) $ and $\lambda _{g_{i}}^{\left( p,q\right) }\left(
f_{2},\varphi \right) <\lambda _{g_{j}}^{\left( p,q\right) }\left(
f_{2},\varphi \right) $ holds simultaneously for $i=1,2;$ $j=1,2\ $and $%
i\neq j.$
\end{theorem}

\qquad Next we intend to find out the sum and product theorems of relative $%
\left( p,q\right) $-$\varphi $ type ( respectively relative $\left(
p,q\right) $-$\varphi $ lower type) and relative $\left( p,q\right) $-$%
\varphi $ weak type of meromorphic function with respect to an entire
function taking into consideration of the above theorems.

\begin{theorem}
\label{t9.13} Let $f_{1},f_{2}$ be any two meromorphic functions and $g_{1}$%
, $g_{2}$ be any two entire functions. Also let $\rho _{g_{1}}^{\left(
p,q\right) }\left( f_{1},\varphi \right) $, $\rho _{g_{1}}^{\left(
p,q\right) }\left( f_{2},\varphi \right) $, $\rho _{g_{2}}^{\left(
p,q\right) }\left( f_{1},\varphi \right) $ and $\rho _{g_{2}}^{\left(
p,q\right) }\left( f_{2},\varphi \right) $ are all non zero and finite.%
\newline
\textbf{(A) }If $\rho _{g_{1}}^{\left( p,q\right) }\left( f_{i},\varphi
\right) >\rho _{g_{1}}^{\left( p,q\right) }\left( f_{j},\varphi \right) $
for $i,$ $j$ $=$ $1,2$; $i\neq j,$ and $g_{1}$ has the Property (A)$,$ then%
\begin{equation*}
\sigma _{g_{1}}^{\left( p,q\right) }\left( f_{1}\pm f_{2},\varphi \right)
=\sigma _{g_{1}}^{\left( p,q\right) }\left( f_{i},\varphi \right) \text{
and\ }\overline{\sigma }_{g_{1}}^{\left( p,q\right) }\left( f_{1}\pm
f_{2},\varphi \right) =\overline{\sigma }_{g_{1}}^{\left( p,q\right) }\left(
f_{i},\varphi \right) \mid i=1,2.
\end{equation*}%
\textbf{(B)} If $\rho _{g_{i}}^{\left( p,q\right) }\left( f_{1},\varphi
\right) <\rho _{g_{j}}^{\left( p,q\right) }\left( f_{1},\varphi \right) $
with at least $f_{1}$ is of regular relative $\left( p,q\right) $-$\varphi $
growth with respect to $g_{j}$ for $i$, $j$ $=$ $1,2$; $i\neq j$ and $%
g_{1}\pm g_{2}$ has the Property (A)$,$ then%
\begin{equation*}
\sigma _{g_{1}\pm g_{2}}^{\left( p,q\right) }\left( f_{1},\varphi \right)
=\sigma _{g_{i}}^{\left( p,q\right) }\left( f_{1},\varphi \right) \text{
and\ }\overline{\sigma }_{g_{1}\pm g_{2}}^{\left( p,q\right) }\left(
f_{1},\varphi \right) =\overline{\sigma }_{g_{i}}^{\left( p,q\right) }\left(
f_{1},\varphi \right) \mid i=1,2.
\end{equation*}%
\textbf{(C)} Assume the functions $f_{1},f_{2},g_{1}$ and $g_{2}$ satisfy
the following conditions:\newline
$\left( i\right) $ $\rho _{g_{i}}^{\left( p,q\right) }\left( f_{1},\varphi
\right) <\rho _{g_{j}}^{\left( p,q\right) }\left( f_{1},\varphi \right) $
with at least $f_{1}$ is of regular relative $\left( p,q\right) $-$\varphi $
growth with respect to $g_{j}$ for $i$ $=$ $1,$ $2$, $j$ $=$ $1,2$ and $%
i\neq j$;\newline
$\left( ii\right) $ $\rho _{g_{i}}^{\left( p,q\right) }\left( f_{2},\varphi
\right) <\rho _{g_{j}}^{\left( p,q\right) }\left( f_{2},\varphi \right) $
with at least $f_{2}$ is of regular relative $\left( p,q\right) $-$\varphi $
growth with respect to $g_{j}$ for $i$ $=$ $1,$ $2$, $j$ $=$ $1,2$ and $%
i\neq j$;\newline
$\left( iii\right) $ $\rho _{g_{1}}^{\left( p,q\right) }\left( f_{i},\varphi
\right) >\rho _{g_{1}}^{\left( p,q\right) }\left( f_{j},\varphi \right) $
and $\rho _{g_{2}}^{\left( p,q\right) }\left( f_{i},\varphi \right) >\rho
_{g_{2}}^{\left( p,q\right) }\left( f_{j},\varphi \right) $ holds
simultaneously for $i=1,2;$ $j=1,2\ $and $i\neq j$;\newline
$\left( iv\right) \rho _{g_{m}}^{\left( p,q\right) }\left( f_{l},\varphi
\right) =\max \left[ \min \left\{ \rho _{g_{1}}^{\left( p,q\right) }\left(
f_{1},\varphi \right) ,\rho _{g_{2}}^{\left( p,q\right) }\left(
f_{1},\varphi \right) \right\} ,\min \left\{ \rho _{g_{1}}^{\left(
p,q\right) }\left( f_{2},\varphi \right) ,\rho _{g_{2}}^{\left( p,q\right)
}\left( f_{2},\varphi \right) \right\} \right] \mid l,m=1,2$, and $g_{1}\pm
g_{2}$ has the Property (A);\newline
then%
\begin{equation*}
\sigma _{g_{1}\pm g_{2}}^{\left( p,q\right) }\left( f_{1}\pm f_{2},\varphi
\right) =\sigma _{g_{m}}^{\left( p,q\right) }\left( f_{l},\varphi \right)
\mid l,m=1,2
\end{equation*}%
and%
\begin{equation*}
\overline{\sigma }_{g_{1}\pm g_{2}}^{\left( p,q\right) }\left( f_{1}\pm
f_{2},\varphi \right) =\overline{\sigma }_{g_{m}}^{\left( p,q\right) }\left(
f_{l},\varphi \right) \mid l,m=1,2.
\end{equation*}
\end{theorem}

\begin{proof}
From the definition of relative $\left( p,q\right) $-$\varphi $ type and
relative $\left( p,q\right) $-$\varphi $ lower type of meromorphic function
with respect to an entire function, we have for all sufficiently large
values of $r$ that%
\begin{equation}
T_{f_{k}}\left( r\right) \leq T_{g_{l}}\left[ \exp ^{\left[ p-1\right]
}\left\{ \left( \sigma _{g_{l}}^{\left( p,q\right) }\left( f_{k},\varphi
\right) +\varepsilon \right) \left[ \log ^{\left[ q-1\right] }\varphi \left(
r\right) \right] ^{\rho _{g_{l}}^{\left( p,q\right) }\left( f_{k}\right)
}\right\} \right] ,  \label{9.15}
\end{equation}%
\begin{equation}
T_{f_{k}}\left( r\right) \geq T_{g_{l}}\left[ \exp ^{\left[ p-1\right]
}\left\{ \left( \overline{\sigma }_{g_{l}}^{\left( p,q\right) }\left(
f_{k},\varphi \right) -\varepsilon \right) \left[ \log ^{\left[ q-1\right]
}\varphi \left( r\right) \right] ^{\rho _{g_{l}}^{\left( p,q\right) }\left(
f_{k}\right) }\right\} \right]  \label{9.15a}
\end{equation}%
and for a sequence of values of $r$ tending to infinity, we obtain that%
\begin{equation}
T_{f_{k}}\left( r\right) \geq T_{g_{l}}\left[ \exp ^{\left[ p-1\right]
}\left\{ \left( \sigma _{g_{l}}^{\left( p,q\right) }\left( f_{k},\varphi
\right) -\varepsilon \right) \left[ \log ^{\left[ q-1\right] }\varphi \left(
r\right) \right] ^{\rho _{g_{l}}^{\left( p,q\right) }\left( f_{k}\right)
}\right\} \right] ,  \label{9.20}
\end{equation}%
and%
\begin{equation}
T_{f_{k}}\left( r\right) \leq T_{g_{l}}\left[ \exp ^{\left[ p-1\right]
}\left\{ \left( \overline{\sigma }_{g_{l}}^{\left( p,q\right) }\left(
f_{k},\varphi \right) +\varepsilon \right) \left[ \log ^{\left[ q-1\right]
}\varphi \left( r\right) \right] ^{\rho _{g_{l}}^{\left( p,q\right) }\left(
f_{k}\right) }\right\} \right] ,  \label{9.20a}
\end{equation}%
where $\varepsilon >0$ is any arbitrary positive number $k=1,\,2$ and $%
l=1,2. $\medskip \newline
\textbf{Case I.} Suppose that $\rho _{g_{1}}^{\left( p,q\right) }\left(
f_{1},\varphi \right) >\rho _{g_{1}}^{\left( p,q\right) }\left(
f_{2},\varphi \right) $ hold. Also let $\varepsilon \left( >0\right) $ be
arbitrary. Since $T_{f_{1}\pm f_{2}}\left( r\right) \leq T_{f_{1}}\left(
r\right) +T_{f_{2}}\left( r\right) +O(1)$ for all large $r,$ so in view of $%
\left( \ref{9.15}\right) ,$ we get for all sufficiently large values of $r$
that%
\begin{equation*}
T_{f_{1}\pm f_{2}}\left( r\right) \leq ~\ \ \ \ \ \ \ \ \ \ \ \ \ \ \ \ \ \
\ \ \ \ \ \ \ \ \ \ \ \ \ \ \ \ \ \ \ \ \ \ \ \ \ \ \ \ \ \ \ \ \ \ \ \ \ \
\ \ \ \ \ \ \ \ \ \ \ \ \ \ \ \ \ \ \ \ \ \ \ \ \ \ \ \ \ \ \ \ 
\end{equation*}%
\begin{equation}
T_{g_{1}}\left[ \exp ^{\left[ p-1\right] }\left\{ \left( \sigma
_{g_{1}}^{\left( p,q\right) }\left( f_{1},\varphi \right) +\varepsilon
\right) \left[ \log ^{\left[ q-1\right] }\varphi \left( r\right) \right]
^{\rho _{g_{1}}^{\left( p,q\right) }\left( f_{1}\right) }\right\} \right]
\left( 1+A\right) ~.  \label{9.18}
\end{equation}%
where $A=\frac{T_{g_{1}}\left[ \exp ^{\left[ p-1\right] }\left\{ \left(
\sigma _{g_{1}}^{\left( p,q\right) }\left( f_{2},\varphi \right)
+\varepsilon \right) \left[ \log ^{\left[ q-1\right] }\varphi \left(
r\right) \right] ^{\rho _{g_{1}}^{\left( p,q\right) }\left( f_{2},\varphi
\right) }\right\} \right] +O(1)}{T_{g_{1}}\left[ \exp ^{\left[ p-1\right]
}\left\{ \left( \sigma _{g_{1}}^{\left( p,q\right) }\left( f_{1},\varphi
\right) +\varepsilon \right) \left[ \log ^{\left[ q-1\right] }\varphi \left(
r\right) \right] ^{\rho _{g_{1}}^{\left( p,q\right) }\left( f_{1},\varphi
\right) }\right\} \right] },$ and in view of $\rho _{g_{1}}^{\left(
p,q\right) }\left( f_{1},\varphi \right) >\rho _{g_{1}}^{\left( p,q\right)
}\left( f_{2},\varphi \right) $, and for all sufficiently large values of $r$%
, we can make the term $A$ sufficiently small . Hence for any $\alpha
=1+\varepsilon _{1}$, it follows from $\left( \ref{9.18}\right) $ for all
sufficiently large values of $r$ that%
\begin{equation*}
T_{f_{1}\pm f_{2}}\left( r\right) \leq T_{g_{1}}\left[ \exp ^{\left[ p-1%
\right] }\left\{ \left( \sigma _{g_{1}}^{\left( p,q\right) }\left(
f_{1},\varphi \right) +\varepsilon \right) \left[ \log ^{\left[ q-1\right]
}\varphi \left( r\right) \right] ^{\rho _{g_{1}}^{\left( p,q\right) }\left(
f_{1}\right) }\right\} \right] \cdot \left( 1+\varepsilon _{1}\right)
\end{equation*}%
\begin{equation}
i.e.,~T_{f_{1}\pm f_{2}}\left( r\right) \leq T_{g_{1}}\left[ \exp ^{\left[
p-1\right] }\left\{ \left( \sigma _{g_{1}}^{\left( p,q\right) }\left(
f_{1},\varphi \right) +\varepsilon \right) \left[ \log ^{\left[ q-1\right]
}\varphi \left( r\right) \right] ^{\rho _{g_{1}}^{\left( p,q\right) }\left(
f_{1}\right) }\right\} \right] \cdot \alpha ~.  \notag
\end{equation}

\qquad Hence making $\alpha \rightarrow 1+,$ we get in view of Theorem \ref%
{t9.1x}, $\rho _{g_{1}}^{\left( p,q\right) }\left( f_{1},\varphi \right)
>\rho _{g_{1}}^{\left( p,q\right) }\left( f_{2},\varphi \right) $ and above
for all sufficiently large values of $r$ that%
\begin{equation*}
\underset{r\rightarrow \infty }{\lim \sup }\frac{\log ^{\left[ p-1\right]
}T_{g_{1}}^{-1}\left( T_{f_{1}\pm f_{2}}\left( r\right) \right) }{\left[
\log ^{\left[ q-1\right] }\varphi \left( r\right) \right] ^{\rho
_{g_{1}}^{\left( p,q\right) }\left( f_{1}\pm f_{2},\varphi \right) }}\leq
\sigma _{g_{1}}^{\left( p,q\right) }\left( f_{1},\varphi \right)
\end{equation*}%
\begin{equation}
i.e.,~\sigma _{g_{1}}^{\left( p,q\right) }\left( f_{1}\pm f_{2},\varphi
\right) \leq \sigma _{g_{1}}^{\left( p,q\right) }\left( f_{1},\varphi
\right) ~.~\ \ \ \ \ \ \   \label{9.21}
\end{equation}

Now we may consider that $f=f_{1}\pm f_{2}.$ Since $\rho _{g_{1}}^{\left(
p,q\right) }\left( f_{1},\varphi \right) >\rho _{g_{1}}^{\left( p,q\right)
}\left( f_{2},\varphi \right) $ hold. Then $\sigma _{g_{1}}^{\left(
p,q\right) }\left( f,\varphi \right) =\sigma _{g_{1}}^{\left( p,q\right)
}\left( f_{1}\pm f_{2},\varphi \right) \leq \sigma _{g_{1}}^{\left(
p,q\right) }\left( f_{1},\varphi \right) .$ Further, let $f_{1}=\left( f\pm
f_{2}\right) $. Therefore in view of Theorem \ref{t9.1x} and $\rho
_{g_{1}}^{\left( p,q\right) }\left( f_{1},\varphi \right) >\rho
_{g_{1}}^{\left( p,q\right) }\left( f_{2},\varphi \right) $, we obtain that $%
\rho _{g_{1}}^{\left( p,q\right) }\left( f,\varphi \right) >\rho
_{g_{1}}^{\left( p,q\right) }\left( f_{2},\varphi \right) $ holds. Hence in
view of $\left( \ref{9.21}\right) $ $\sigma _{g_{1}}^{\left( p,q\right)
}\left( f_{1},\varphi \right) \leq \sigma _{g_{1}}^{\left( p,q\right)
}\left( f,\varphi \right) =\sigma _{g_{1}}^{\left( p,q\right) }\left(
f_{1}\pm f_{2},\varphi \right) .$ Therefore $\sigma _{g_{1}}^{\left(
p,q\right) }\left( f,\varphi \right) =\sigma _{g_{1}}^{\left( p,q\right)
}\left( f_{1},\varphi \right) \Rightarrow $ $\sigma _{g_{1}}^{\left(
p,q\right) }\left( f_{1}\pm f_{2},\varphi \right) =\sigma _{g_{1}}^{\left(
p,q\right) }\left( f_{1},\varphi \right) $.

\qquad Similarly, if we consider $\rho _{g_{1}}^{\left( p,q\right) }\left(
f_{1},\varphi \right) <\rho _{g_{1}}^{\left( p,q\right) }\left(
f_{2},\varphi \right) ,$ then one can easily verify that $\sigma
_{g_{1}}^{\left( p,q\right) }\left( f_{1}\pm f_{2},\varphi \right) =\sigma
_{g_{1}}^{\left( p,q\right) }\left( f_{2},\varphi \right) $.\medskip \newline
\textbf{Case II.} Let us consider that $\rho _{g_{1}}^{\left( p,q\right)
}\left( f_{1},\varphi \right) >\rho _{g_{1}}^{\left( p,q\right) }\left(
f_{2},\varphi \right) $ hold. Also let $\varepsilon \left( >0\right) $ are
arbitrary. Since $T_{f_{1}\pm f_{2}}\left( r\right) \leq T_{f_{1}}\left(
r\right) +T_{f_{2}}\left( r\right) +O(1)$ for all large $r,$ from $\left( %
\ref{9.15}\right) $ and $\left( \ref{9.20a}\right) ,$ we get for a sequence
of values of $r$ tending to infinity that%
\begin{equation*}
T_{f_{1}\pm f_{2}}\left( r_{n}\right) \leq ~\ \ \ \ \ \ \ \ \ \ \ \ \ \ \ \
\ \ \ \ \ \ \ \ \ \ \ \ \ \ \ \ \ \ \ \ \ \ \ \ \ \ \ \ \ \ \ \ \ \ \ \ \ \
\ \ \ \ \ \ \ \ \ \ \ \ \ \ \ \ \ \ \ \ \ \ \ \ \ \ \ \ \ \ \ \ \ \ \ \ \ \
\ \ \ \ \ \ \ \ \ \ \ \ \ \ 
\end{equation*}%
\begin{equation}
T_{g_{1}}\left[ \exp ^{\left[ p-1\right] }\left\{ \left( \overline{\sigma }%
_{g_{1}}^{\left( p,q\right) }\left( f_{1},\varphi \right) +\varepsilon
\right) \left[ \log ^{\left[ q-1\right] }\varphi \left( r_{n}\right) \right]
^{\rho _{g_{1}}^{\left( p,q\right) }\left( f_{1},\varphi \right) }\right\} %
\right] \left( 1+B\right) ~.  \label{9.185}
\end{equation}%
where $B=\frac{T_{g_{1}}\left[ \exp ^{\left[ p-1\right] }\left\{ \left(
\sigma _{g_{1}}^{\left( p,q\right) }\left( f_{2},\varphi \right)
+\varepsilon \right) \left[ \log ^{\left[ q-1\right] }\varphi \left(
r_{n}\right) \right] ^{\rho _{g_{1}}^{\left( p,q\right) }\left(
f_{2},\varphi \right) }\right\} \right] +O(1)}{T_{g_{1}}\left[ \exp ^{\left[
p-1\right] }\left\{ \left( \overline{\sigma }_{g_{1}}^{\left( p,q\right)
}\left( f_{1},\varphi \right) +\varepsilon \right) \left[ \log ^{\left[ q-1%
\right] }\varphi \left( r_{n}\right) \right] ^{\rho _{g_{1}}^{\left(
p,q\right) }\left( f_{1},\varphi \right) }\right\} \right] },$ and in view
of $\rho _{g_{1}}^{\left( p,q\right) }\left( f_{1},\varphi \right) >\rho
_{g_{1}}^{\left( p,q\right) }\left( f_{2},\varphi \right) $, we can make the
term $B$\ sufficiently small by taking $n$ sufficiently large and therefore
using the similar technique for as executed in the proof of Case I we get
from $\left( \ref{9.185}\right) $ that $\overline{\sigma }_{g_{1}}^{\left(
p,q\right) }\left( f_{1}\pm f_{2},\varphi \right) =\overline{\sigma }%
_{g_{1}}^{\left( p,q\right) }\left( f_{1},\varphi \right) $ when $\rho
_{g_{1}}^{\left( p,q\right) }\left( f_{1},\varphi \right) >\rho
_{g_{1}}^{\left( p,q\right) }\left( f_{2},\varphi \right) $ hold. Likewise,
if we consider $\rho _{g_{1}}^{\left( p,q\right) }\left( f_{1},\varphi
\right) <\rho _{g_{1}}^{\left( p,q\right) }\left( f_{2},\varphi \right) ,$
then one can easily verify that $\overline{\sigma }_{g_{1}}^{\left(
p,q\right) }\left( f_{1}\pm f_{2},\varphi \right) =\overline{\sigma }%
_{g_{1}}^{\left( p,q\right) }\left( f_{2},\varphi \right) $.

\qquad Thus combining Case I and Case II, we obtain the first part of the
theorem.\medskip \newline
\textbf{Case III.} Let us consider that $\rho _{g_{1}}^{\left( p,q\right)
}\left( f_{1},\varphi \right) <\rho _{g_{2}}^{\left( p,q\right) }\left(
f_{1},\varphi \right) $ with at least $f_{1}$ is of regular relative $\left(
p,q\right) $-$\varphi $ growth with respect to $g_{2}.$ We can make the term
\linebreak $C=\frac{T_{g_{2}}\left[ \exp ^{\left[ p-1\right] }\left\{ \left(
\sigma _{g_{1}}^{\left( p,q\right) }\left( f_{1},\varphi \right)
-\varepsilon \right) \left[ \log ^{\left[ q-1\right] }\varphi \left(
r_{n}\right) \right] ^{\rho _{g_{1}}^{\left( p,q\right) }\left(
f_{1},\varphi \right) }\right\} \right] +O(1)}{T_{g_{2}}\left[ \exp ^{\left[
p-1\right] }\left\{ \left( \overline{\sigma }_{g_{2}}^{\left( p,q\right)
}\left( f_{1},\varphi \right) -\varepsilon \right) \left[ \log ^{\left[ q-1%
\right] }\varphi \left( r_{n}\right) \right] ^{\rho _{g_{2}}^{\left(
p,q\right) }\left( f_{1},\varphi \right) }\right\} \right] }$ sufficiently
small by taking $n$ sufficiently large, since $\rho _{g_{1}}^{\left(
p,q\right) }\left( f_{1},\varphi \right) <\rho _{g_{2}}^{\left( p,q\right)
}\left( f_{1},\varphi \right) .$ Hence $C<\varepsilon _{1}.$

\qquad As $T_{g_{1}\pm g_{2}}\left( r\right) \leq T_{g_{1}}\left( r\right)
+T_{g_{2}}\left( r\right) +O(1)$ for all large $r$, we get that%
\begin{equation*}
T_{g_{1}\pm g_{2}}\left( \exp ^{\left[ p-1\right] }\left\{ \left( \sigma
_{g_{1}}^{\left( p,q\right) }\left( f_{1},\varphi \right) -\varepsilon
\right) \left[ \log ^{\left[ q-1\right] }\varphi \left( r_{n}\right) \right]
^{\rho _{g_{1}}^{\left( p,q\right) }\left( f_{1},\varphi \right) }\right\}
\right) \leq
\end{equation*}%
\begin{equation*}
~\ \ \ \ T_{g_{1}}\left[ \exp ^{\left[ p-1\right] }\left\{ \left( \sigma
_{g_{1}}^{\left( p,q\right) }\left( f_{1},\varphi \right) -\varepsilon
\right) \left[ \log ^{\left[ q-1\right] }\varphi \left( r_{n}\right) \right]
^{\rho _{g_{1}}^{\left( p,\right) }\left( f_{1},\varphi \right) }\right\} %
\right] +
\end{equation*}%
\begin{equation*}
T_{g_{2}}\left[ \exp ^{\left[ p-1\right] }\left\{ \left( \sigma
_{g_{1}}^{\left( p,q\right) }\left( f_{1},\varphi \right) -\varepsilon
\right) \left[ \log ^{\left[ q-1\right] }\varphi \left( r_{n}\right) \right]
^{\rho _{g_{1}}^{\left( p,q\right) }\left( f_{1},\varphi \right) }\right\} %
\right] +O(1)~.
\end{equation*}

\qquad Therefore for any $\alpha =1+\varepsilon _{1},$ we obtain in view of $%
C<\varepsilon _{1},$ $\left( \ref{9.15a}\right) $ and $\left( \ref{9.20}%
\right) $ for a sequence of values of $r$ tending to infinity that%
\begin{equation*}
T_{g_{1}\pm g_{2}}\left( \exp ^{\left[ p-1\right] }\left\{ \left( \sigma
_{g_{1}}^{\left( p,q\right) }\left( f_{1},\varphi \right) -\varepsilon
\right) \left[ \log ^{\left[ q-1\right] }\varphi \left( r_{n}\right) \right]
^{\rho _{g_{1}}^{\left( p,q\right) }\left( f_{1},\varphi \right) }\right\}
\right) \leq \alpha T_{f_{1}}\left( r_{n}\right)
\end{equation*}

\qquad Now making $\alpha \rightarrow 1+$, we obtain from above for a
sequence of values of $r$ tending to infinity that%
\begin{equation}
\left( \sigma _{g_{1}}^{\left( p,q\right) }\left( f_{1},\varphi \right)
-\varepsilon \right) \left[ \log ^{\left[ q-1\right] }\varphi \left(
r_{n}\right) \right] ^{\rho _{g_{1}\pm g_{2}}^{\left( p,q\right) }\left(
f_{1},\varphi \right) }<\log ^{\left[ p-1\right] }T_{g_{1}\pm
g_{2}}^{-1}T_{f_{1}}\left( r_{n}\right)  \notag
\end{equation}

\qquad Since $\varepsilon >0$ is arbitrary, we find that%
\begin{equation}
\sigma _{g_{1}\pm g_{2}}^{\left( p,q\right) }\left( f_{1},\varphi \right)
\geq \sigma _{g_{1}}^{\left( p,q\right) }\left( f_{1},\varphi \right) ~.
\label{9.237}
\end{equation}

\qquad Now we may consider that $g=g_{1}\pm g_{2}.$ Also $\rho
_{g_{1}}^{\left( p,q\right) }\left( f_{1},\varphi \right) <\rho
_{g_{2}}^{\left( p,q\right) }\left( f_{1},\varphi \right) $ and at least $%
f_{1}$ is of regular relative $\left( p,q\right) $-$\varphi $ growth with
respect to $g_{2}$. Then $\sigma _{g}^{\left( p,q\right) }\left(
f_{1},\varphi \right) =\sigma _{g_{1}\pm g_{2}}^{\left( p,q\right) }\left(
f_{1},\varphi \right) \geq \sigma _{g_{1}}^{\left( p,q\right) }\left(
f_{1},\varphi \right) .$ Further let $g_{1}=\left( g\pm g_{2}\right) $.
Therefore in view of Theorem \ref{t9.4} and $\rho _{g_{1}}^{\left(
p,q\right) }\left( f_{1},\varphi \right) <\rho _{g_{2}}^{\left( p,q\right)
}\left( f_{1},\varphi \right) $, we obtain that $\rho _{g}^{\left(
p,q\right) }\left( f_{1},\varphi \right) <\rho _{g_{2}}^{\left( p,q\right)
}\left( f_{1},\varphi \right) $ as at least $f_{1}$ is of regular relative $%
\left( p,q\right) $-$\varphi $ growth with respect to $g_{2}$. Hence in view
of $\left( \ref{9.237}\right) $, $\sigma _{g_{1}}^{\left( p,q\right) }\left(
f_{1},\varphi \right) \geq \sigma _{g}^{\left( p,q\right) }\left(
f_{1},\varphi \right) =\sigma _{g_{1}\pm g_{2}}^{\left( p,q\right) }\left(
f_{1},\varphi \right) .$ Therefore $\sigma _{g}^{\left( p,q\right) }\left(
f_{1},\varphi \right) =\sigma _{g_{1}}^{\left( p,q\right) }\left(
f_{1},\varphi \right) \Rightarrow $ $\sigma _{g_{1}\pm g_{2}}^{\left(
p,q\right) }\left( f_{1},\varphi \right) =\sigma _{g_{1}}^{\left( p,q\right)
}\left( f_{1},\varphi \right) $.

\qquad Similarly if we consider $\rho _{g_{1}}^{\left( p,q\right) }\left(
f_{1},\varphi \right) >\rho _{g_{2}}^{\left( p,q\right) }\left(
f_{1},\varphi \right) $ with at least $f_{1}$ is of regular relative $\left(
p,q\right) $-$\varphi $ growth with respect to $g_{1}$, then $\sigma
_{g_{1}\pm g_{2}}^{\left( p,q\right) }\left( f_{1},\varphi \right) =\sigma
_{g_{2}}^{\left( p,q\right) }\left( f_{1},\varphi \right) .$\medskip \newline
\textbf{Case IV. }In this case suppose that $\rho _{g_{1}}^{\left(
p,q\right) }\left( f_{1},\varphi \right) <\rho _{g_{2}}^{\left( p,q\right)
}\left( f_{1},\varphi \right) $ with at least $f_{1}$ is of regular relative 
$\left( p,q\right) $-$\varphi $ growth with respect to $g_{2}.$ we can also
make the term $D=\frac{T_{g_{2}}\left[ \exp ^{\left[ p-1\right] }\left\{
\left( \overline{\sigma }_{g_{1}}^{\left( p,q\right) }\left( f_{1},\varphi
\right) -\varepsilon \right) \left[ \log ^{\left[ q-1\right] }\varphi \left(
r\right) \right] ^{\rho _{g_{1}}^{\left( p,q\right) }\left( f_{1},\varphi
\right) }\right\} \right] +O(1)}{T_{g_{2}}\left[ \exp ^{\left[ p-1\right]
}\left\{ \left( \overline{\sigma }_{g_{2}}^{\left( p,q\right) }\left(
f_{1},\varphi \right) -\varepsilon \right) \left[ \log ^{\left[ q-1\right]
}\varphi \left( r\right) \right] ^{\rho _{g_{2}}^{\left( p,q\right) }\left(
f_{1},\varphi \right) }\right\} \right] }$ sufficiently small by taking $r$
sufficiently large as $\rho _{g_{1}}^{\left( p,q\right) }\left(
f_{1},\varphi \right) <\rho _{g_{2}}^{\left( p,q\right) }\left(
f_{1},\varphi \right) .$ So $D<\varepsilon _{1}$ for sufficiently large $r.$
As $T_{g_{1}\pm g_{2}}\left( r\right) \leq T_{g_{1}}\left( r\right)
+T_{g_{2}}\left( r\right) +O(1)$ for all large $r$, therefore from $\left( %
\ref{9.15a}\right) ,$ we get for all sufficiently large values of $r$ that%
\begin{equation*}
T_{g_{1}\pm g_{2}}\left( \exp ^{\left[ p-1\right] }\left\{ \left( \overline{%
\sigma }_{g_{1}}^{\left( p,q\right) }\left( f_{1},\varphi \right)
-\varepsilon \right) \left[ \log ^{\left[ q-1\right] }\varphi \left(
r\right) \right] ^{\rho _{g_{1}}^{\left( p,q\right) }\left( f_{1},\varphi
\right) }\right\} \right) \leq
\end{equation*}%
\begin{equation*}
~\ \ \ \ T_{g_{1}}\left[ \exp ^{\left[ p-1\right] }\left\{ \left( \overline{%
\sigma }_{g_{1}}^{\left( p,q\right) }\left( f_{1},\varphi \right)
-\varepsilon \right) \left[ \log ^{\left[ q-1\right] }\varphi \left(
r\right) \right] ^{\rho _{g_{1}}^{\left( p,q\right) }\left( f_{1},\varphi
\right) }\right\} \right] +
\end{equation*}%
\begin{equation*}
T_{g_{2}}\left[ \exp ^{\left[ p-1\right] }\left\{ \left( \overline{\sigma }%
_{g_{1}}^{\left( p,q\right) }\left( f_{1},\varphi \right) -\varepsilon
\right) \left[ \log ^{\left[ q-1\right] }\varphi \left( r\right) \right]
^{\rho _{g_{1}}^{\left( p,q\right) }\left( f_{1},\varphi \right) }\right\} %
\right] +O(1)
\end{equation*}%
\begin{equation*}
i.e.,~T_{g_{1}\pm g_{2}}\left( \exp ^{\left[ p-1\right] }\left\{ \left( 
\overline{\sigma }_{g_{1}}^{\left( p,q,t\right) L}\left( f_{1},\varphi
\right) -\varepsilon \right) \left[ \log ^{\left[ q-1\right] }\varphi \left(
r\right) \right] ^{\rho _{g_{1}}^{\left( p,q,t\right) L}\left( f_{1},\varphi
\right) }\right\} \right)
\end{equation*}%
\begin{equation}
~\ \ \ \ \ \ \ \ \ \ \ \ \ \ \ \ \ \ \ \ \ \ \ \ \ \ \ \ \ \ \ \ \ \ \ \ \ \
\ \ \ \ \ \ \ \ \ \ \ \ \ \ \ \ \ \ \ \ \ \ \ \ \ \ \ \ \ \ \ \ \ \ \ \ \ \
\ \ \ \ \ \ \ \ \ \ \ \leq \left( 1+\varepsilon _{1}\right) T_{f_{1}}\left(
r\right) ~.,  \label{9.238}
\end{equation}

and therefore using the similar technique for as executed in the proof of
Case III we get from $\left( \ref{9.238}\right) $ that $\overline{\sigma }%
_{g_{1}\pm g_{2}}^{\left( p,q\right) }\left( f_{1},\varphi \right) =%
\overline{\sigma }_{g_{1}}^{\left( p,q\right) }\left( f_{1},\varphi \right) $
where $\rho _{g_{1}}^{\left( p,q\right) }\left( f_{1},\varphi \right) <\rho
_{g_{2}}^{\left( p,q\right) }\left( f_{1},\varphi \right) $ and at least $%
f_{1}$ is of regular relative $\left( p,q\right) $-$\varphi $ growth with
respect to $g_{2}$.

\qquad Likewise if we consider $\rho _{g_{1}}^{\left( p,q\right) }\left(
f_{1},\varphi \right) >\rho _{g_{2}}^{\left( p,q\right) }\left(
f_{1},\varphi \right) $ with at least $f_{1}$ is of regular relative $\left(
p,q\right) $-$\varphi $ growth with respect to $g_{1}$, then $\overline{%
\sigma }_{g_{1}\pm g_{2}}^{\left( p,q\right) }\left( f_{1},\varphi \right) =%
\overline{\sigma }_{g_{2}}^{\left( p,q\right) }\left( f_{1},\varphi \right)
. $

\qquad Thus combining Case III and Case IV, we obtain the second part of the
theorem.

\qquad The third part of the theorem is a natural consequence of Theorem \ref%
{t9.5} and the first part and second part of the theorem. Hence its proof is
omitted.
\end{proof}

\begin{theorem}
\label{t9.14} Let $f_{1},f_{2}$ be any two meromorphic functions and $g_{1}$%
, $g_{2}$ be any two entire functions. Also let $\lambda _{g_{1}}^{\left(
p,q\right) }\left( f_{1},\varphi \right) $, $\lambda _{g_{1}}^{\left(
p,q\right) }\left( f_{2},\varphi \right) $, $\lambda _{g_{2}}^{\left(
p,q\right) }\left( f_{1},\varphi \right) $ and $\lambda _{g_{2}}^{\left(
p,q\right) }\left( f_{2},\varphi \right) $ are all non zero and finite.%
\newline
\textbf{(A)} If $\lambda _{g_{1}}^{\left( p,q\right) }\left( f_{i},\varphi
\right) >\lambda _{g_{1}}^{\left( p,q\right) }\left( f_{j},\varphi \right) $
with at least $f_{j}$ is of regular relative $\left( p,q\right) $-$\varphi $
growth with respect to $g_{1}$ for $i$, $j$ $=$ $1,2$; $i\neq j$, and $g_{1}$
has the Property (A), then%
\begin{equation*}
\tau _{g_{1}}^{\left( p,q\right) }\left( f_{1}\pm f_{2},\varphi \right)
=\tau _{g_{1}}^{\left( p,q\right) }\left( f_{i},\varphi \right) \text{ and \ 
}\overline{\tau }_{g_{1}}^{\left( p,q\right) }\left( f_{1}\pm f_{2},\varphi
\right) =\overline{\tau }_{g_{1}}^{\left( p,q\right) }\left( f_{i},\varphi
\right) \mid i=1,2~.
\end{equation*}%
\textbf{(B)} If $\lambda _{g_{i}}^{\left( p,q\right) }\left( f_{1},\varphi
\right) <\lambda _{g_{j}}^{\left( p,q\right) }\left( f_{1},\varphi \right) $
for $i$, $j$ $=$ $1,2$; $i\neq j$ and $g_{1}\pm g_{2}$ has the Property (A),
then%
\begin{equation*}
\tau _{g_{1}\pm g_{2}}^{\left( p,q\right) }\left( f_{1},\varphi \right)
=\tau _{g_{i}}^{\left( p,q\right) }\left( f_{1},\varphi \right) \text{ and \ 
}\overline{\tau }_{g_{1}\pm g_{2}}^{\left( p,q\right) }\left( f_{1},\varphi
\right) =\overline{\tau }_{g_{i}}^{\left( p,q\right) }\left( f_{1},\varphi
\right) \mid i=1,2~.
\end{equation*}%
\textbf{(C)} Assume the functions $f_{1},f_{2},g_{1}$ and $g_{2}$ satisfy
the following conditions:\newline
$\left( i\right) $ $\rho _{g_{1}}^{\left( p,q\right) }\left( f_{i},\varphi
\right) >\rho _{g_{1}}^{\left( p,q\right) }\left( f_{j},\varphi \right) $
with at least $f_{j}$ is of regular relative $\left( p,q\right) $-$\varphi $
growth with respect to $g_{1}$ for $i$, $j$ $=$ $1,2$ and $i\neq j$;\newline
$\left( ii\right) $ $\rho _{g_{2}}^{\left( p,q\right) }\left( f_{i},\varphi
\right) >\rho _{g_{2}}^{\left( p,q\right) }\left( f_{j},\varphi \right) $
with at least $f_{j}$ is of regular relative $\left( p,q\right) $-$\varphi $
growth with respect to $g_{2}$ for $i$, $j$ $=$ $1,2$ and $i\neq j$;\newline
$\left( iii\right) $ $\rho _{g_{i}}^{\left( p,q\right) }\left( f_{1},\varphi
\right) <\rho _{g_{j}}^{\left( p,q\right) }\left( f_{1},\varphi \right) $
and $\rho _{g_{i}}^{\left( p,q\right) }\left( f_{2},\varphi \right) <\rho
_{g_{j}}^{\left( p,q\right) }\left( f_{2},\varphi \right) $ holds
simultaneously for $i,$ $j=1,2\ $and $i\neq j$;\newline
$\left( iv\right) \lambda _{g_{m}}^{\left( p,q\right) }\left( f_{l},\varphi
\right) =\min \left[ \max \left\{ \lambda _{g_{1}}^{\left( p,q\right)
}\left( f_{1},\varphi \right) ,\lambda _{g_{1}}^{\left( p,q\right) }\left(
f_{2},\varphi \right) \right\} ,\max \left\{ \lambda _{g_{2}}^{\left(
p,q\right) }\left( f_{1},\varphi \right) ,\lambda _{g_{2}}^{\left(
p,q\right) }\left( f_{2},\varphi \right) \right\} \right] \mid l,m=1,2$ and $%
g_{1}\pm g_{2}$ has the Property (A)\newline
then we have%
\begin{equation*}
\tau _{g_{1}\pm g_{2}}^{\left( p,q\right) }\left( f_{1}\pm f_{2},\varphi
\right) =\tau _{g_{m}}^{\left( p,q\right) }\left( f_{l},\varphi \right) \mid
l,m=1,2
\end{equation*}%
and%
\begin{equation*}
\overline{\tau }_{g_{1}\pm g_{2}}^{\left( p,q\right) }\left( f_{1}\pm
f_{2},\varphi \right) =\overline{\tau }_{g_{m}}^{\left( p,q\right) }\left(
f_{l},\varphi \right) \mid l,m=1,2~.
\end{equation*}
\end{theorem}

\begin{proof}
For any arbitrary positive number $\varepsilon (>0)$, we have for all
sufficiently large values of $r$ that%
\begin{equation}
T_{f_{k}}\left( r\right) \leq T_{g_{l}}\left[ \exp ^{\left[ p-1\right]
}\left\{ \left( \overline{\tau }_{g_{l}}^{\left( p,q\right) }\left(
f_{k},\varphi \right) +\varepsilon \right) \left[ \log ^{\left[ q-1\right]
}\varphi \left( r\right) \right] ^{\lambda _{g_{l}}^{\left( p,q\right)
}\left( f_{k},\varphi \right) }\right\} \right] ,  \label{9.15x}
\end{equation}%
\begin{equation}
T_{f_{k}}\left( r\right) \geq T_{g_{l}}\left[ \exp ^{\left[ p-1\right]
}\left\{ \left( \tau _{g_{l}}^{\left( p,q\right) }\left( f_{k},\varphi
\right) -\varepsilon \right) \left[ \log ^{\left[ q-1\right] }\varphi \left(
r\right) \right] ^{\lambda _{g_{l}}^{\left( p,q\right) }\left( f_{k},\varphi
\right) }\right\} \right] ,  \label{9.15ax}
\end{equation}%
and for a sequence of values of $r$ tending to infinity we obtain that%
\begin{equation}
T_{f_{k}}\left( r\right) \geq T_{g_{l}}\left[ \exp ^{\left[ p-1\right]
}\left\{ \left( \overline{\tau }_{g_{l}}^{\left( p,q\right) }\left(
f_{k},\varphi \right) -\varepsilon \right) \left[ \log ^{\left[ q-1\right]
}\varphi \left( r\right) \right] ^{\lambda _{g_{l}}^{\left( p,q\right)
}\left( f_{k},\varphi \right) }\right\} \right]  \label{9.20x}
\end{equation}%
and%
\begin{equation}
T_{f_{k}}\left( r\right) \leq T_{g_{l}}\left[ \exp ^{\left[ p-1\right]
}\left\{ \left( \tau _{g_{l}}^{\left( p,q\right) }\left( f_{k},\varphi
\right) +\varepsilon \right) \left[ \log ^{\left[ q-1\right] }\varphi \left(
r\right) \right] ^{\lambda _{g_{l}}^{\left( p,q\right) }\left( f_{k},\varphi
\right) }\right\} \right] ,  \label{9.20ax}
\end{equation}%
where $k=1,2$ and $l=1,2.$\medskip \newline
\textbf{Case I.} Let $\lambda _{g_{1}}^{\left( p,q\right) }\left(
f_{1},\varphi \right) >\lambda _{g_{1}}^{\left( p,q\right) }\left(
f_{2},\varphi \right) $ with at least $f_{2}$ is of regular relative $\left(
p,q\right) $-$\varphi $ growth with respect to $g_{1}$. Also let $%
\varepsilon \left( >0\right) $ be arbitrary. Since $T_{f_{1}\pm f_{2}}\left(
r\right) \leq T_{f_{1}}\left( r\right) +T_{f_{2}}\left( r\right) +O(1)$ for
all large $r,$ we get from $\left( \ref{9.15x}\right) $ and $\left( \ref%
{9.20ax}\right) ,$ for a sequence $\left\{ r_{n}\right\} $ of values of $r$
tending to infinity that%
\begin{equation*}
T_{f_{1}\pm f_{2}}\left( r_{n}\right) \leq ~\ \ \ \ \ \ \ \ \ \ \ \ \ \ \ \
\ \ \ \ \ \ \ \ \ \ \ \ \ \ \ \ \ \ \ \ \ \ \ \ \ \ \ \ \ \ \ \ \ \ \ \ \ \
\ \ \ \ \ \ \ \ \ \ \ \ \ \ \ \ \ \ \ \ \ \ \ \ \ \ \ \ \ \ \ \ \ \ \ \ \ \
\ \ \ \ \ \ \ \ \ \ \ \ 
\end{equation*}%
\begin{equation}
T_{g_{1}}\left[ \exp ^{\left[ p-1\right] }\left\{ \left( \tau
_{g_{1}}^{\left( p,q\right) }\left( f_{1},\varphi \right) +\varepsilon
\right) \left[ \log ^{\left[ q-1\right] }\varphi \left( r_{n}\right) \right]
^{\lambda _{g_{1}}^{\left( p,q\right) }\left( f_{1},\varphi \right)
}\right\} \right] \left( 1+E\right) ~.  \label{9.18x}
\end{equation}%
where $E=\frac{T_{g_{1}}\left[ \exp ^{\left[ p-1\right] }\left\{ \left( 
\overline{\tau }_{g_{1}}^{\left( p,q\right) }\left( f_{2},\varphi \right)
+\varepsilon \right) \left[ \log ^{\left[ q-1\right] }\varphi \left(
r_{n}\right) \right] ^{\lambda _{g_{1}}^{\left( p,q\right) }\left(
f_{2},\varphi \right) }\right\} \right] +O(1)}{T_{g_{1}}\left[ \exp ^{\left[
p-1\right] }\left\{ \left( \tau _{g_{1}}^{\left( p,q\right) }\left(
f_{1},\varphi \right) +\varepsilon \right) \left[ \log ^{\left[ q-1\right]
}\varphi \left( r_{n}\right) \right] ^{\lambda _{g_{1}}^{\left( p,q\right)
}\left( f_{1},\varphi \right) }\right\} \right] }$ and in view of $\lambda
_{g_{1}}^{\left( p,q\right) }\left( f_{1}\right) >\lambda _{g_{1}}^{\left(
p,q\right) }\left( f_{2}\right) $, we can make the term $E$ sufficiently
small by taking $n$ sufficiently large. Now with the help of Theorem \ref%
{t9.2x} and using the similar technique of Case I of Theorem \ref{t9.13}, we
get from $\left( \ref{9.18x}\right) $ that%
\begin{equation}
\tau _{g_{1}}^{\left( p,q\right) }\left( f_{1}\pm f_{2},\varphi \right) \leq
\tau _{g_{1}}^{\left( p,q\right) }\left( f_{1},\varphi \right) ~.
\label{9.21x}
\end{equation}

\qquad Further, we may consider that $f=f_{1}\pm f_{2}.$ Also suppose that $%
\lambda _{g_{1}}^{\left( p,q\right) }\left( f_{1},\varphi \right) >\lambda
_{g_{1}}^{\left( p,q\right) }\left( f_{2},\varphi \right) $ and at least $%
f_{2}$ is of regular relative $\left( p,q\right) $-$\varphi $ growth with
respect to $g_{1}$. Then $\tau _{g_{1}}^{\left( p,q\right) }\left( f,\varphi
\right) =\tau _{g_{1}}^{\left( p,q\right) }\left( f_{1}\pm f_{2},\varphi
\right) \leq \tau _{g_{1}}^{\left( p,q\right) }\left( f_{1},\varphi \right)
. $ Now let $f_{1}=\left( f\pm f_{2}\right) $. Therefore in view of Theorem %
\ref{t9.2x}, $\lambda _{g_{1}}^{\left( p,q\right) }\left( f_{1},\varphi
\right) >\lambda _{g_{1}}^{\left( p,q\right) }\left( f_{2},\varphi \right) $
and at least $f_{2}$ is of regular relative $\left( p,q\right) $-$\varphi $
growth with respect to $g_{1},$ we obtain that $\lambda _{g_{1}}^{\left(
p,q\right) }\left( f,\varphi \right) >\lambda _{g_{1}}^{\left( p,q\right)
}\left( f_{2},\varphi \right) $ holds. Hence in view of $\left( \ref{9.21x}%
\right) $, $\tau _{g_{1}}^{\left( p,q\right) }\left( f_{1},\varphi \right)
\leq \tau _{g_{1}}^{\left( p,q\right) }\left( f,\varphi \right) =\tau
_{g_{1}}^{\left( p,q\right) }\left( f_{1}\pm f_{2},\varphi \right) .$
Therefore $\tau _{g_{1}}^{\left( p,q\right) }\left( f,\varphi \right) =\tau
_{g_{1}}^{\left( p,q\right) }\left( f_{1},\varphi \right) \Rightarrow $ $%
\tau _{g_{1}}^{\left( p,q\right) }\left( f_{1}\pm f_{2},\varphi \right)
=\tau _{g_{1}}^{\left( p,q\right) }\left( f_{1},\varphi \right) $.

\qquad Similarly, if we consider $\lambda _{g_{1}}^{\left( p,q\right)
}\left( f_{1},\varphi \right) <\lambda _{g_{1}}^{\left( p,q\right) }\left(
f_{2},\varphi \right) $ with at least $f_{1}$ is of regular relative $\left(
p,q\right) $-$\varphi $ growth with respect to $g_{1}$ then one can easily
verify that $\tau _{g_{1}}^{\left( p,q\right) }\left( f_{1}\pm f_{2},\varphi
\right) =\tau _{g_{1}}^{\left( p,q\right) }\left( f_{2},\varphi \right) $%
.\medskip \newline
\textbf{Case II.} Let us consider that $\lambda _{g_{1}}^{\left( p,q\right)
}\left( f_{1},\varphi \right) >\lambda _{g_{1}}^{\left( p,q\right) }\left(
f_{2},\varphi \right) $ with at least $f_{2}$ is of regular relative $\left(
p,q\right) $-$\varphi $ growth with respect to $g_{1}$. Also let $%
\varepsilon \left( >0\right) $ be arbitrary. As $T_{f_{1}\pm f_{2}}\left(
r\right) \leq T_{f_{1}}\left( r\right) +T_{f_{2}}\left( r\right) +O(1)$ for
all large $r,$ we obtain from $\left( \ref{9.15x}\right) $ for all
sufficiently large values of $r$ that%
\begin{equation*}
T_{f_{1}\pm f_{2}}\left( r\right) \leq ~\ \ \ \ \ \ \ \ \ \ \ \ \ \ \ \ \ \
\ \ \ \ \ \ \ \ \ \ \ \ \ \ \ \ \ \ \ \ \ \ \ \ \ \ \ \ \ \ \ \ \ \ \ \ \ \
\ \ \ \ \ \ \ \ \ \ \ \ \ \ \ \ \ \ \ \ \ \ \ \ \ \ \ \ \ \ \ \ \ \ \ \ \ \
\ \ \ \ \ \ \ \ 
\end{equation*}%
\begin{equation}
T_{g_{1}}\left[ \exp ^{\left[ p-1\right] }\left\{ \left( \overline{\tau }%
_{g_{1}}^{\left( p,q\right) }\left( f_{1},\varphi \right) +\varepsilon
\right) \left[ \log ^{\left[ q-1\right] }\varphi \left( r\right) \right]
^{\lambda _{g_{1}}^{\left( p,q\right) }\left( f_{1},\varphi \right)
}\right\} \right] \left( 1+F\right) ~.  \label{9.185x}
\end{equation}%
where $F=\frac{T_{g_{1}}\left[ \exp ^{\left[ p-1\right] }\left\{ \left( 
\overline{\tau }_{g_{1}}^{\left( p,q\right) }\left( f_{2},\varphi \right)
+\varepsilon \right) \left[ \log ^{\left[ q-1\right] }\varphi \left(
r\right) \right] ^{\lambda _{g_{i}}^{\left( p,q\right) }\left( f_{2},\varphi
\right) }\right\} \right] +O(1)}{T_{g_{1}}\left[ \exp ^{\left[ p-1\right]
}\left\{ \left( \overline{\tau }_{g_{1}}^{\left( p,q\right) }\left(
f_{1},\varphi \right) +\varepsilon \right) \left[ \log ^{\left[ q-1\right]
}\varphi \left( r\right) \right] ^{\lambda _{g_{1}}^{\left( p,q\right)
}\left( f_{1,\varphi }\right) }\right\} \right] },$ and in view of $\lambda
_{g_{1}}^{\left( p,q\right) }\left( f_{1},\varphi \right) >\lambda
_{g_{1}}^{\left( p,q\right) }\left( f_{2},\varphi \right) $, we can make the
term $F$ sufficiently small by taking $r$ sufficiently large and therefore
for similar reasoning of Case I we get from $\left( \ref{9.185x}\right) $
that $\overline{\tau }_{g_{1}}^{\left( p,q\right) }\left( f_{1}\pm
f_{2},\varphi \right) =\overline{\tau }_{g_{1}}^{\left( p,q\right) }\left(
f_{1},\varphi \right) $ when $\lambda _{g_{1}}^{\left( p,q\right) }\left(
f_{1},\varphi \right) >\lambda _{g_{1}}^{\left( p,q\right) }\left(
f_{2},\varphi \right) $ and at least $f_{2}$ is of regular relative $\left(
p,q\right) $-$\varphi $ growth with respect to $g_{1}$.

\qquad Likewise, if we consider $\lambda _{g_{1}}^{\left( p,q\right) }\left(
f_{1},\varphi \right) <\lambda _{g_{1}}^{\left( p,q\right) }\left(
f_{2},\varphi \right) $ with at least $f_{1}$ is of regular relative $\left(
p,q\right) $-$\varphi $ growth with respect to $g_{1}$ then one can easily
verify that $\overline{\tau }_{g_{1}}^{\left( p,q\right) }\left( f_{1}\pm
f_{2},\varphi \right) =\overline{\tau }_{g_{1}}^{\left( p,q\right) }\left(
f_{2},\varphi \right) .$

\qquad Thus combining Case I and Case II, we obtain the first part of the
theorem.\medskip \newline
\textbf{Case III.} Let us consider that $\lambda _{g_{1}}^{\left( p,q\right)
}\left( f_{1},\varphi \right) <\lambda _{g_{2}}^{\left( p,q\right) }\left(
f_{1},\varphi \right) $. Therefore we can make the term $G=\frac{T_{g_{2}}%
\left[ \exp ^{\left[ p-1\right] }\left\{ \left( \tau _{g_{1}}^{\left(
p,q\right) }\left( f_{1},\varphi \right) -\varepsilon \right) \left[ \log ^{%
\left[ q-1\right] }\varphi \left( r\right) \right] ^{\lambda
_{g_{1}}^{\left( p,q\right) }\left( f_{1},\varphi \right) }\right\} \right]
+O(1)}{T_{g_{2}}\left[ \exp ^{\left[ p-1\right] }\left\{ \left( \tau
_{g_{2}}^{\left( p,q\right) }\left( f_{1},\varphi \right) -\varepsilon
\right) \left[ \log ^{\left[ q-1\right] }\varphi \left( r\right) \right]
^{\lambda _{g_{2}}^{\left( p,q\right) }\left( f_{1},\varphi \right)
}\right\} \right] }$ sufficiently small by taking $r$ sufficiently large
since $\lambda _{g_{1}}^{\left( p,q\right) }\left( f_{1},\varphi \right) $ $%
< $ $\lambda _{g_{2}}^{\left( p,q\right) }\left( f_{1},\varphi \right) .$ So 
$G<\varepsilon _{1}.$ Since $T_{g_{1}\pm g_{2}}\left( r\right) \leq
T_{g_{1}}\left( r\right) +T_{g_{2}}\left( r\right) +O(1)$ for all large $r,$
we get from $\left( \ref{9.15ax}\right) $ for all sufficiently large values
of $r$ that%
\begin{equation*}
T_{g_{1}\pm g_{2}}\left( \exp ^{\left[ p-1\right] }\left\{ \left( \tau
_{g_{1}}^{\left( p,q\right) }\left( f_{1},\varphi \right) -\varepsilon
\right) \left[ \log ^{\left[ q-1\right] }\varphi \left( r\right) \right]
^{\lambda _{g_{1}}^{\left( p,q\right) }\left( f_{1},\varphi \right)
}\right\} \right) \leq
\end{equation*}%
\begin{equation*}
~\ T_{g_{1}}\left[ \exp ^{\left[ p-1\right] }\left\{ \left( \tau
_{g_{1}}^{\left( p,q\right) }\left( f_{1},\varphi \right) -\varepsilon
\right) \left[ \log ^{\left[ q-1\right] }\varphi \left( r\right) \right]
^{\lambda _{g_{1}}^{\left( p,q\right) }\left( f_{1},\varphi \right)
}\right\} \right] +
\end{equation*}%
\begin{equation*}
T_{g_{2}}\left[ \exp ^{\left[ p-1\right] }\left\{ \left( \tau
_{g_{1}}^{\left( p,q\right) }\left( f_{1},\varphi \right) -\varepsilon
\right) \left[ \log ^{\left[ q-1\right] }\varphi \left( r\right) \right]
^{\lambda _{g_{1}}^{\left( p,q\right) }\left( f_{1},\varphi \right)
}\right\} \right] +O(1)
\end{equation*}%
\begin{equation*}
i.e.,~T_{g_{1}\pm g_{2}}\left( \exp ^{\left[ p-1\right] }\left\{ \left( \tau
_{g_{1}}^{\left( p,q\right) }\left( f_{1},\varphi \right) -\varepsilon
\right) \left[ \log ^{\left[ q-1\right] }\varphi \left( r\right) \right]
^{\lambda _{g_{1}}^{\left( p,q\right) }\left( f_{1},\varphi \right)
}\right\} \right)
\end{equation*}%
\begin{equation}
~\ \ \ \ \ \ \ \ \ \ \ \ \ \ \ \ \ \ \ \ \ \ \ \ \ \ \ \ \ \ \ \ \ \ \ \ \ \
\ \ \ \ \ \ \ \ \ \ \ \ \ \ \ \ \ \ \ \ \ \ \ \ \ \ \ \ \ \ \ \ \ \ \ \ \ \
\ \ \ \ \ \ \ \ \ \ \ \leq \left( 1+\varepsilon _{1}\right) T_{f_{1}}\left(
r\right) ~.  \label{9.236x}
\end{equation}

\qquad Therefore in view of Theorem \ref{t9.3} and using the similar
technique of Case III of Theorem \ref{t9.13}, we get from $\left( \ref%
{9.236x}\right) $ that%
\begin{equation}
\tau _{g_{1}\pm g_{2}}^{\left( p,q\right) }\left( f_{1},\varphi \right) \geq
\tau _{g_{1}}^{\left( p,q\right) }\left( f_{1},\varphi \right) ~.
\label{9.237x}
\end{equation}

\qquad Further, we may consider that $g=g_{1}\pm g_{2}.$ As $\lambda
_{g_{1}}^{\left( p,q\right) }\left( f_{1},\varphi \right) <\lambda
_{g_{2}}^{\left( p,q\right) }\left( f_{1},\varphi \right) $, so $\tau
_{g}^{\left( p,q\right) }\left( f_{1},\varphi \right) =\tau _{g_{1}\pm
g_{2}}^{\left( p,q\right) }\left( f_{1},\varphi \right) \geq \tau
_{g_{1}}^{\left( p,q\right) }\left( f_{1},\varphi \right) $. Further let $%
g_{1}=\left( g\pm g_{2}\right) $. Therefore in view of Theorem \ref{t9.3}
and $\lambda _{g_{1}}^{\left( p,q\right) }\left( f_{1},\varphi \right)
<\lambda _{g_{2}}^{\left( p,q\right) }\left( f_{1},\varphi \right) $ we
obtain that $\lambda _{g}^{\left( p,q\right) }\left( f_{1},\varphi \right)
<\lambda _{g_{2}}^{\left( p,q\right) }\left( f_{1},\varphi \right) $ holds.
Hence in view of $\left( \ref{9.237x}\right) $ $\tau _{g_{1}}^{\left(
p,q\right) }\left( f_{1},\varphi \right) \geq \tau _{g}^{\left( p,q\right)
}\left( f_{1},\varphi \right) =\tau _{g_{1}\pm g_{2}}^{\left( p,q\right)
}\left( f_{1},\varphi \right) .$ Therefore $\tau _{g}^{\left( p,q\right)
}\left( f_{1},\varphi \right) =\tau _{g_{1}}^{\left( p,q\right) }\left(
f_{1},\varphi \right) \Rightarrow $ $\tau _{g_{1}\pm g_{2}}^{\left(
p,q\right) }\left( f_{1},\varphi \right) =\tau _{g_{1}}^{\left( p,q\right)
}\left( f_{1},\varphi \right) $.

\qquad Likewise, if we consider that $\lambda _{g_{1}}^{\left( p,q\right)
}\left( f_{1},\varphi \right) >\lambda _{g_{2}}^{\left( p,q\right) }\left(
f_{1},\varphi \right) ,$ then one can easily verify that $\tau _{g_{1}\pm
g_{2}}^{\left( p,q\right) }\left( f_{1},\varphi \right) =\tau
_{g_{2}}^{\left( p,q\right) }\left( f_{1},\varphi \right) .$\medskip \newline
\textbf{Case IV. }In this case further we consider $\lambda _{g_{1}}^{\left(
p,q\right) }\left( f_{1},\varphi \right) <\lambda _{g_{2}}^{\left(
p,q\right) }\left( f_{1},\varphi \right) $. Further we can make the term $H=%
\frac{T_{g_{2}}\left[ \exp ^{\left[ p-1\right] }\left\{ \left( \overline{%
\tau }_{g_{1}}^{\left( p,q\right) }\left( f_{1},\varphi \right) -\varepsilon
\right) \left[ \log ^{\left[ q-1\right] }\varphi \left( r_{n}\right) \right]
^{\lambda _{g_{1}}^{\left( p,q\right) }\left( f_{1},\varphi \right)
}\right\} \right] +O(1)}{T_{g_{2}}\left[ \exp ^{\left[ p-1\right] }\left\{
\left( \tau _{g_{2}}^{\left( p,q\right) }\left( f_{1},\varphi \right)
-\varepsilon \right) \left[ \log ^{\left[ q-1\right] }\varphi \left(
r_{n}\right) \right] ^{\lambda _{g_{2}}^{\left( p,q\right) }\left(
f_{1},\varphi \right) }\right\} \right] }$ sufficiently small by taking $n$
sufficiently large, since $\lambda _{g_{1}}^{\left( p,q\right) }\left(
f_{1},\varphi \right) $ $<$ $\lambda _{g_{2}}^{\left( p,q\right) }\left(
f_{1},\varphi \right) .$ Therefore $H<\varepsilon _{1}$ for sufficiently
large $n.$ As $T_{g_{1}\pm g_{2}}\left( r\right) \leq T_{g_{1}}\left(
r\right) +T_{g_{2}}\left( r\right) +O(1)$ for all large $r,$ we obtain from $%
\left( \ref{9.15ax}\right) $ and $\left( \ref{9.20x}\right) ,$ we obtain for
a sequence $\left\{ r_{n}\right\} $ of values of $r$ tending to infinity that%
\begin{equation*}
T_{g_{1}\pm g_{2}}\left( \exp ^{\left[ p-1\right] }\left\{ \left( \overline{%
\tau }_{g_{1}}^{\left( p,q\right) }\left( f_{1},\varphi \right) -\varepsilon
\right) \left[ \log ^{\left[ q-1\right] }\varphi \left( r_{n}\right) \right]
^{\lambda _{g_{1}}^{\left( p,q\right) }\left( f_{1},\varphi \right)
}\right\} \right) \leq
\end{equation*}%
\begin{equation*}
~\ \ \ \ T_{g_{1}}\left[ \exp ^{\left[ p-1\right] }\left\{ \left( \overline{%
\tau }_{g_{1}}^{\left( p,q\right) }\left( f_{1},\varphi \right) -\varepsilon
\right) \left[ \log ^{\left[ q-1\right] }\varphi \left( r_{n}\right) \right]
^{\lambda _{g_{1}}^{\left( p,q\right) }\left( f_{1},\varphi \right)
}\right\} \right] +
\end{equation*}%
\begin{equation*}
T_{g_{2}}\left[ \exp ^{\left[ p-1\right] }\left\{ \left( \overline{\tau }%
_{g_{1}}^{\left( p,q\right) }\left( f_{1},\varphi \right) -\varepsilon
\right) \left[ \log ^{\left[ q-1\right] }\varphi \left( r_{n}\right) \right]
^{\lambda _{g_{1}}^{\left( p,q\right) }\left( f_{1},\varphi \right)
}\right\} \right] +O(1)
\end{equation*}%
\begin{equation*}
i.e.,~T_{g_{1}\pm g_{2}}\left( \exp ^{\left[ p-1\right] }\left\{ \left( 
\overline{\tau }_{g_{1}}^{\left( p,q\right) }\left( f_{1},\varphi \right)
-\varepsilon \right) \left[ \log ^{\left[ q-1\right] }\varphi \left(
r_{n}\right) \right] ^{\lambda _{g_{1}}^{\left( p,q\right) }\left(
f_{1},\varphi \right) }\right\} \right)
\end{equation*}%
\begin{equation}
~\ \ \ \ \ \ \ \ \ \ \ \ \ \ \ \ \ \ \ \ \ \ \ \ \ \ \ \ \ \ \ \ \ \ \ \ \ \
\ \ \ \ \ \ \ \ \ \ \ \ \ \ \ \ \ \ \ \ \ \ \ \ \ \ \ \ \ \ \ \ \ \ \ \ \ \
\ \ \ \ \ \ \ \ \ \ \ \leq \left( 1+\varepsilon _{1}\right) T_{f_{1}}\left(
r\right) ,  \label{9.238x}
\end{equation}

and therefore using the similar technique for as executed in the proof of
Case IV of Theorem \ref{t9.13}, we get from $\left( \ref{9.238x}\right) $
that $\overline{\tau }_{g_{1}\pm g_{2}}^{\left( p,q\right) }\left(
f_{1},\varphi \right) =\overline{\tau }_{g_{1}}^{\left( p,q\right) }\left(
f_{1},\varphi \right) $ when $\lambda _{g_{1}}^{\left( p,q\right) }\left(
f_{1},\varphi \right) <\lambda _{g_{2}}^{\left( p,q\right) }\left(
f_{1},\varphi \right) $.

\qquad Similarly, if we consider that $\lambda _{g_{1}}^{\left( p,q\right)
}\left( f_{1},\varphi \right) >\lambda _{g_{2}}^{\left( p,q\right) }\left(
f_{1},\varphi \right) ,$ then one can easily verify that $\overline{\tau }%
_{g_{1}\pm g_{2}}^{\left( p,q\right) }\left( f_{1},\varphi \right) =%
\overline{\tau }_{g_{2}}^{\left( p,q\right) }\left( f_{1},\varphi \right) .$

\qquad Thus combining Case III and Case IV, we obtain the second part of the
theorem.

\qquad The proof of the third part of the Theorem is omitted as it can be
carried out in view of Theorem \ref{t9.6} and the above cases.
\end{proof}

\qquad In the next two theorems we reconsider the equalities in Theorem \ref%
{t9.2x} to Theorem \ref{t9.4} under somewhat different conditions.

\begin{theorem}
\label{t9.15} Let $f_{1},\,f_{2}\,$be any two meromorphic functions and $%
g_{1}$, $g_{2}$ be any two entire functions.\newline
\textbf{(A) }The following condition is assumed to be satisfied:\newline
$\left( i\right) $ Either $\sigma _{g_{1}}^{\left( p,q\right) }\left(
f_{1},\varphi \right) \neq \sigma _{g_{1}}^{\left( p,q\right) }\left(
f_{2},\varphi \right) $ or $\overline{\sigma }_{g_{1}}^{\left( p,q\right)
}\left( f_{1},\varphi \right) \neq \overline{\sigma }_{g_{1}}^{\left(
p,q\right) }\left( f_{2},\varphi \right) $ holds and $g_{1}$ has the
Property (A), then%
\begin{equation*}
\rho _{g_{1}}^{\left( p,q\right) }\left( f_{1}\pm f_{2},\varphi \right)
=\rho _{g_{1}}^{\left( p,q\right) }\left( f_{1},\varphi \right) =\rho
_{g_{1}}^{\left( p,q\right) }\left( f_{2},\varphi \right) ~.
\end{equation*}%
\textbf{(B)} The following conditions are assumed to be satisfied:\newline
$\left( i\right) $ Either $\sigma _{g_{1}}^{\left( p,q\right) }\left(
f_{1},\varphi \right) \neq \sigma _{g_{2}}^{\left( p,q\right) }\left(
f_{1},\varphi \right) $ or $\overline{\sigma }_{g_{1}}^{\left( p,q\right)
}\left( f_{1},\varphi \right) \neq \overline{\sigma }_{g_{2}}^{\left(
p,q\right) }\left( f_{1},\varphi \right) $ holds and $g_{1}\pm g_{2}$ has
the Property (A);\newline
$\left( ii\right) $ $f_{1}$ is of regular relative $\left( p,q\right) $-$%
\varphi $ growth with respect to at least any one of $g_{1}$ or $g_{2}$, then%
\begin{equation*}
\rho _{g_{1}\pm g_{2}}^{\left( p,q\right) }\left( f_{1},\varphi \right)
=\rho _{g_{1}}^{\left( p,q\right) }\left( f_{1},\varphi \right) =\rho
_{g_{2}}^{\left( p,q\right) }\left( f_{1},\varphi \right) ~.
\end{equation*}
\end{theorem}

\begin{proof}
Let $f_{1},\,f_{2},\,g_{1}$ and $g_{2}$ be any four entire functions
satisfying the conditions of the theorem.\medskip \newline
\textbf{Case I.} Suppose that $\rho _{g_{1}}^{\left( p,q\right) }\left(
f_{1},\varphi \right) =\rho _{g_{1}}^{\left( p,q\right) }\left(
f_{2},\varphi \right) $ $(0<$ $\rho _{g_{1}}^{\left( p,q\right) }\left(
f_{1},\varphi \right) ,$ $\rho _{g_{1}}^{\left( p,q\right) }\left(
f_{2},\varphi \right) $ $<\infty )$. Now in view of Theorem \ref{t9.1x} it
is easy to see that $\rho _{g_{1}}^{\left( p,q\right) }\left( f_{1}\pm
f_{2},\varphi \right) \leq \rho _{g_{1}}^{\left( p,q\right) }\left(
f_{1},\varphi \right) =\rho _{g_{1}}^{\left( p,q\right) }\left(
f_{2},\varphi \right) ~.$ If possible let 
\begin{equation}
\rho _{g_{1}}^{\left( p,q\right) }\left( f_{1}\pm f_{2},\varphi \right)
<\rho _{g_{1}}^{\left( p,q\right) }\left( f_{1},\varphi \right) =\rho
_{g_{1}}^{\left( p,q\right) }\left( f_{2},\varphi \right) ~.  \label{93.1x}
\end{equation}

\qquad Let\textbf{\ }$\sigma _{g_{1}}^{\left( p,q\right) }\left(
f_{1},\varphi \right) \neq \sigma _{g_{1}}^{\left( p,q\right) }\left(
f_{2},\varphi \right) .$ Then in view of the first part of Theorem \ref%
{t9.13} and $\left( \ref{93.1x}\right) $ we obtain that $\sigma
_{g_{1}}^{\left( p,q\right) }\left( f_{1},\varphi \right) =\sigma
_{g_{1}}^{\left( p,q\right) }\left( f_{1}\pm f_{2}\mp f_{2},\varphi \right)
=\sigma _{g_{1}}^{\left( p,q\right) }\left( f_{2},\varphi \right) $ which is
a contradiction. Hence $\rho _{g_{1}}^{\left( p,q\right) }\left( f_{1}\pm
f_{2},\varphi \right) $ $=$ $\rho _{g_{1}}^{\left( p,q\right) }\left(
f_{1},\varphi \right) $ $=$ $\rho _{g_{1}}^{\left( p,q\right) }\left(
f_{2},\varphi \right) ~.$ Similarly with the help of the first part of
Theorem \ref{t9.13}, one can obtain the same conclusion under the hypothesis 
$\overline{\sigma }_{g_{1}}^{\left( p,q\right) }\left( f_{1},\varphi \right)
\neq \overline{\sigma }_{g_{1}}^{\left( p,q\right) }\left( f_{2},\varphi
\right) .$ This proves the first part of the theorem.\medskip \newline
\textbf{Case II. }Let us consider that $\rho _{g_{1}}^{\left( p,q\right)
}\left( f_{1},\varphi \right) =\rho _{g_{2}}^{\left( p,q\right) }\left(
f_{1},\varphi \right) $ $(0<$ $\rho _{g_{1}}^{\left( p,q\right) }\left(
f_{1},\varphi \right) ,$ $\rho _{g_{2}}^{\left( p,q\right) }\left(
f_{1},\varphi \right) $ $<\infty )$, $f_{1}$ is of regular relative $\left(
p,q\right) $-$\varphi $ growth with respect to at least any one of $g_{1}$
or $g_{2}$ and $\left( g_{1}\pm g_{2}\right) $ and $g_{1}\pm g_{2}$ satisfy
the Property (A). Therefore in view of Theorem \ref{t9.4}, it follows that $%
\rho _{g_{1}\pm g_{2}}^{\left( p,q\right) }\left( f_{1},\varphi \right) \geq
\rho _{g_{1}}^{\left( p,q\right) }\left( f_{1},\varphi \right) =\rho
_{g_{2}}^{\left( p,q\right) }\left( f_{1},\varphi \right) $ and if possible
let 
\begin{equation}
\rho _{g_{1}\pm g_{2}}^{\left( p,q\right) }\left( f_{1},\varphi \right)
>\rho _{g_{1}}^{\left( p,q\right) }\left( f_{1},\varphi \right) =\rho
_{g_{2}}^{\left( p,q\right) }\left( f_{1},\varphi \right) ~.  \label{93.3x}
\end{equation}

\qquad Let us consider that $\sigma _{g_{1}}^{\left( p,q\right) }\left(
f_{1},\varphi \right) \neq \sigma _{g_{2}}^{\left( p,q\right) }\left(
f_{1},\varphi \right) .$ Then. in view of the proof of the second part of
Theorem \ref{t9.13} and $\left( \ref{93.3x}\right) $ we obtain that $\sigma
_{g_{1}}^{\left( p,q\right) }\left( f_{1},\varphi \right) =\sigma _{g_{1}\pm
g_{2}\mp g_{2}}^{\left( p,q\right) }\left( f_{1},\varphi \right) =\sigma
_{g_{2}}^{\left( p,q\right) }\left( f_{1},\varphi \right) $ which is a
contradiction. Hence $\rho _{g_{1}\pm g_{2}}^{\left( p,q\right) }\left(
f_{1},\varphi \right) =\rho _{g_{1}}^{\left( p,q\right) }\left(
f_{1},\varphi \right) =\rho _{g_{2}}^{\left( p,q\right) }\left(
f_{1},\varphi \right) ~.$ Also in view of the proof of second part of
Theorem \ref{t9.13} one can derive the same conclusion for the condition $%
\overline{\sigma }_{g_{1}}^{\left( p,q\right) }\left( f_{1},\varphi \right)
\neq \overline{\sigma }_{g_{2}}^{\left( p,q\right) }\left( f_{1},\varphi
\right) $ and therefore the second part of the theorem is established.
\end{proof}

\begin{theorem}
\label{t9.15A} Let $f_{1},\,f_{2}$ be any two meromorphic functions and $%
g_{1}$, $g_{2}$ be any two entire functions.\newline
\textbf{(A) }The following conditions are assumed to be satisfied:\newline
$\left( i\right) $ $\left( f_{1}\pm f_{2}\right) $ is of regular relative $%
\left( p,q\right) $-$\varphi $ growth with respect to at least any one of $%
g_{1}$ or $g_{2}$, and $g_{1},$ $g_{2}$ , $g_{1}\pm g_{2}$ have the Property
(A);\newline
$\left( ii\right) $ Either $\sigma _{g_{1}}^{\left( p,q\right) }\left(
f_{1}\pm f_{2},\varphi \right) \neq \sigma _{g_{2}}^{\left( p,q\right)
}\left( f_{1}\pm f_{2},\varphi \right) $ or $\overline{\sigma }%
_{g_{1}}^{\left( p,q\right) }\left( f_{1}\pm f_{2},\varphi \right) \neq 
\overline{\sigma }_{g_{2}}^{\left( p,q\right) }\left( f_{1}\pm f_{2},\varphi
\right) $;\newline
$\left( iii\right) $ Either $\sigma _{g_{1}}^{\left( p,q\right) }\left(
f_{1},\varphi \right) \neq \sigma _{g_{1}}^{\left( p,q\right) }\left(
f_{2},\varphi \right) $ or $\overline{\sigma }_{g_{1}}^{\left( p,q\right)
}\left( f_{1},\varphi \right) \neq \overline{\sigma }_{g_{1}}^{\left(
p,q\right) }\left( f_{2},\varphi \right) $;\newline
$\left( iv\right) $ Either $\sigma _{g_{2}}^{\left( p,q\right) }\left(
f_{1},\varphi \right) \neq \sigma _{g_{2}}^{\left( p,q\right) }\left(
f_{2},\varphi \right) $ or $\overline{\sigma }_{g_{2}}^{\left( p,q\right)
}\left( f_{1},\varphi \right) \neq \overline{\sigma }_{g_{2}}^{\left(
p,q\right) }\left( f_{2},\varphi \right) $; then%
\begin{equation*}
\rho _{g_{1}\pm g_{2}}^{\left( p,q\right) }\left( f_{1}\pm f_{2},\varphi
\right) =\rho _{g_{1}}^{\left( p,q\right) }\left( f_{1},\varphi \right)
=\rho _{g_{1}}^{\left( p,q\right) }\left( f_{2},\varphi \right) =\rho
_{g_{2}}^{\left( p,q\right) }\left( f_{1},\varphi \right) =\rho
_{g_{2}}^{\left( p,q\right) }\left( f_{2},\varphi \right) ~.
\end{equation*}%
\textbf{(B) }The following conditions are assumed to be satisfied:\newline
$\left( i\right) $ $f_{1}$ and $f_{2}$ are of regular relative $\left(
p,q\right) $-$\varphi $ growth with respect to at least any one\ of $g_{1}$
or $g_{2},$ and $g_{1}\pm g_{2}$ has the Property (A);\newline
$\left( ii\right) $ Either $\sigma _{g_{1}\pm g_{2}}^{\left( p,q\right)
}\left( f_{1},\varphi \right) \neq \sigma _{g_{1}\pm g_{2}}^{\left(
p,q\right) }\left( f_{2},\varphi \right) $ or $\overline{\sigma }_{g_{1}\pm
g_{2}}^{\left( p,q\right) }\left( f_{1},\varphi \right) \neq \overline{%
\sigma }_{g_{1}\pm g_{2}}^{\left( p,q\right) }\left( f_{2},\varphi \right) $;%
\newline
$\left( iii\right) $ Either $\sigma _{g_{1}}^{\left( p,q\right) }\left(
f_{1},\varphi \right) \neq \sigma _{g_{2}}^{\left( p,q\right) }\left(
f_{1},\varphi \right) $ or $\overline{\sigma }_{g_{1}}^{\left( p,q\right)
}\left( f_{1},\varphi \right) \neq \overline{\sigma }_{g_{2}}^{\left(
p,q\right) }\left( f_{1},\varphi \right) $;\newline
$\left( iv\right) $ Either $\sigma _{g_{1}}^{\left( p,q\right) }\left(
f_{2},\varphi \right) \neq \sigma _{g_{2}}^{\left( p,q\right) }\left(
f_{2},\varphi \right) $ or $\overline{\sigma }_{g_{1}}^{\left( p,q\right)
}\left( f_{2},\varphi \right) \neq \overline{\sigma }_{g_{2}}^{\left(
p,q\right) }\left( f_{2},\varphi \right) $; then%
\begin{equation*}
\rho _{g_{1}\pm g_{2}}^{\left( p,q\right) }\left( f_{1}\pm f_{2},\varphi
\right) =\rho _{g_{1}}^{\left( p,q\right) }\left( f_{1},\varphi \right)
=\rho _{g_{1}}^{\left( p,q\right) }\left( f_{2},\varphi \right) =\rho
_{g_{2}}^{\left( p,q\right) }\left( f_{1},\varphi \right) =\rho
_{g_{2}}^{\left( p,q\right) }\left( f_{2},\varphi \right) ~.
\end{equation*}
\end{theorem}

\qquad We omit the proof of Theorem \ref{t9.15A} as it is a natural
consequence of Theorem \ref{t9.15}.

\begin{theorem}
\label{t9.16} Let $f_{1},\,f_{2}\,$be ant two meromorphic functions and $%
g_{1}$,$g_{2}$ be any two entire functions.\newline
\textbf{(A)} The following conditions are assumed to be satisfied:\newline
$\left( i\right) $ At least any one of $f_{1}$ or $f_{2}$ is of regular
relative $\left( p,q\right) $-$\varphi $ growth with respect to $g_{1}$;%
\newline
$\left( ii\right) $ Either $\tau _{g_{1}}^{\left( p,q\right) }\left(
f_{1},\varphi \right) \neq \tau _{g_{1}}^{\left( p,q\right) }\left(
f_{2},\varphi \right) $ or $\overline{\tau }_{g_{1}}^{\left( p,q\right)
}\left( f_{1},\varphi \right) \neq \overline{\tau }_{g_{1}}^{\left(
p,q\right) }\left( f_{2},\varphi \right) $ holds and $g_{1}$ has the
Property (A), then%
\begin{equation*}
\lambda _{g_{1}}^{\left( p,q\right) }\left( f_{1}\pm f_{2},\varphi \right)
=\lambda _{g_{1}}^{\left( p,q\right) }\left( f_{1},\varphi \right) =\lambda
_{g_{1}}^{\left( p,q\right) }\left( f_{2},\varphi \right) ~.
\end{equation*}%
\textbf{(B)} The following conditions are assumed to be satisfied:\newline
$\left( i\right) $ $f_{1},$ $g_{1}$ and $g_{2}$ be any three entire
functions such that $\lambda _{g_{1}}^{\left( p,q\right) }\left(
f_{1},\varphi \right) $ and $\lambda _{g_{2}}^{\left( p,q\right) }\left(
f_{1},\varphi \right) $ exists;\newline
$\left( ii\right) $ Either $\tau _{g_{1}}^{\left( p,q\right) }\left(
f_{1},\varphi \right) \neq \tau _{g_{2}}^{\left( p,q\right) }\left(
f_{1},\varphi \right) $ or $\overline{\tau }_{g_{1}}^{\left( p,q\right)
}\left( f_{1},\varphi \right) \neq \overline{\tau }_{g_{2}}^{\left(
p,q\right) }\left( f_{1},\varphi \right) $ holds and $g_{1}\pm g_{2}$ has
the Property (A), then%
\begin{equation*}
\lambda _{g_{1}\pm g_{2}}^{\left( p,q\right) }\left( f_{1},\varphi \right)
=\lambda _{g_{1}}^{\left( p,q\right) }\left( f_{1},\varphi \right) =\lambda
_{g_{2}}^{\left( p,q\right) }\left( f_{1},\varphi \right) ~.
\end{equation*}
\end{theorem}

\begin{proof}
Let $f_{1},\,f_{2},\,g_{1}$ and $g_{2}$ be any four entire functions
satisfying the conditions of the theorem.\medskip \newline
\textbf{Case I.} Let $\lambda _{g_{1}}^{\left( p,q\right) }\left(
f_{1},\varphi \right) =\lambda _{g_{1}}^{\left( p,q\right) }\left(
f_{2},\varphi \right) $ $(0<\lambda _{g_{1}}^{\left( p,q\right) }\left(
f_{1},\varphi \right) ,\lambda _{g_{1}}^{\left( p,q\right) }\left(
f_{2},\varphi \right) <\infty )$ and at least $f_{1}$ or $f_{2}$ and $\left(
f_{1}\pm f_{2}\right) $ are of regular relative $\left( p,q\right) $-$%
\varphi $ growth with respect to $g_{1}$. Now, in view of Theorem \ref{t9.2x}%
, it is easy to see that $\lambda _{g_{1}}^{\left( p,q\right) }\left(
f_{1}\pm f_{2},\varphi \right) \leq \lambda _{g_{1}}^{\left( p,q\right)
}\left( f_{1},\varphi \right) =\lambda _{g_{1}}^{\left( p,q\right) }\left(
f_{2},\varphi \right) .$ If possible let 
\begin{equation}
\lambda _{g_{1}}^{\left( p,q\right) }\left( f_{1}\pm f_{2},\varphi \right)
<\lambda _{g_{1}}^{\left( p,q\right) }\left( f_{1},\varphi \right) =\lambda
_{g_{1}}^{\left( p,q\right) }\left( f_{2},\varphi \right) ~.  \label{96.19xy}
\end{equation}

\qquad Let\textbf{\ }$\tau _{g_{1}}^{\left( p,q\right) }\left( f_{1},\varphi
\right) \neq \tau _{g_{1}}^{\left( p,q\right) }\left( f_{2},\varphi \right)
. $ Then in view of the proof of the first part of Theorem \ref{t9.14} and $%
\left( \ref{96.19xy}\right) $ we obtain that $\tau _{g_{1}}^{\left(
p,q\right) }\left( f_{1},\varphi \right) =\tau _{g_{1}}^{\left( p,q\right)
}\left( f_{1}\pm f_{2}\mp f_{2},\varphi \right) =\tau _{g_{1}}^{\left(
p,q\right) }\left( f_{2},\varphi \right) $ which is a contradiction. Hence $%
\lambda _{g_{1}}^{\left( p,q\right) }\left( f_{1}\pm f_{2},\varphi \right) $ 
$=$ $\lambda _{g_{1}}^{\left( p,q\right) }\left( f_{1},\varphi \right) $ $=$ 
$\lambda _{g_{1}}^{\left( p,q\right) }\left( f_{2},\varphi \right) ~.$
Similarly in view of the proof of the first part of Theorem \ref{t9.14} ,
one can establish the same conclusion under the hypothesis $\overline{\tau }%
_{g_{1}}^{\left( p,q\right) }\left( f_{1},\varphi \right) \neq \overline{%
\tau }_{g_{1}}^{\left( p,q\right) }\left( f_{2},\varphi \right) .$ This
proves the first part of the theorem.\medskip \newline
\textbf{Case II.} Let us consider that $\lambda _{g_{1}}^{\left( p,q\right)
}\left( f_{1},\varphi \right) =\lambda _{g_{2}}^{\left( p,q\right) }\left(
f_{1},\varphi \right) $ $(0<\lambda _{g_{1}}^{\left( p,q\right) }\left(
f_{1},\varphi \right) ,\lambda _{g_{2}}^{\left( p,q\right) }\left(
f_{1},\varphi \right) <\infty .$ Therefore in view of Theorem \ref{t9.3}, it
follows that $\lambda _{g_{1}\pm g_{2}}^{\left( p,q\right) }\left(
f_{1},\varphi \right) \geq \lambda _{g_{1}}^{\left( p,q\right) }\left(
f_{1},\varphi \right) =\lambda _{g_{2}}^{\left( p,q\right) }\left(
f_{1},\varphi \right) $ and if possible let 
\begin{equation}
\lambda _{g_{1}\pm g_{2}}^{\left( p,q\right) }\left( f_{1},\varphi \right)
>\lambda _{g_{1}}^{\left( p,q\right) }\left( f_{1},\varphi \right) =\lambda
_{g_{2}}^{\left( p,q\right) }\left( f_{1},\varphi \right) ~.  \label{96.2t}
\end{equation}

\qquad Suppose\textbf{\ }$\tau _{g_{1}}^{\left( p,q\right) }\left(
f_{1},\varphi \right) \neq \tau _{g_{2}}^{\left( p,q\right) }\left(
f_{1},\varphi \right) .$ Then in view of the second part of Theorem \ref%
{t9.14} and $\left( \ref{96.2t}\right) $, we obtain that $\tau
_{g_{1}}^{\left( p,q\right) }\left( f_{1},\varphi \right) =\tau _{g_{1}\pm
g_{2}\mp g_{2}}^{\left( p,q\right) }\left( f_{1},\varphi \right) =\tau
_{g_{2}}^{\left( p,q\right) }\left( f_{1},\varphi \right) $ which is a
contradiction. Hence $\lambda _{g_{1}\pm g_{2}}^{\left( p,q\right) }\left(
f_{1},\varphi \right) $ $=$ $\lambda _{g_{1}}^{\left( p,q\right) }\left(
f_{1},\varphi \right) $ $=$ $\lambda _{g_{2}}^{\left( p,q\right) }\left(
f_{1},\varphi \right) ~.$ Analogously with the help of the second part of
Theorem \ref{t9.14}, the same conclusion can also be derived under the
condition $\overline{\tau }_{g_{1}}^{\left( p,q\right) }\left( f_{1},\varphi
\right) \neq \overline{\tau }_{g_{2}}^{\left( p,q\right) }\left(
f_{1},\varphi \right) $ and therefore the second part of the theorem is
established.
\end{proof}

\begin{theorem}
\label{t9.16A} Let $f_{1},\,f_{2}\,$be any two meromorphic functions and $%
g_{1}$, $g_{2}$ be any two entire functions.\newline
\textbf{(A)} The following conditions are assumed to be satisfied:\newline
$\left( i\right) $ At least any one of $f_{1}$ or $f_{2}$ is of regular
relative $\left( p,q\right) $-$\varphi $ growth with respect to $g_{1}$ and $%
g_{2}$. Also $g_{1},$ $g_{2},$ $g_{1}\pm g_{2}$ have satisfy the Property
(A);\newline
$\left( ii\right) $ Either $\tau _{g_{1}}^{\left( p,q\right) }\left(
f_{1}\pm f_{2},\varphi \right) \neq \tau _{g_{2}}^{\left( p,q\right) }\left(
f_{1}\pm f_{2},\varphi \right) $ or $\overline{\tau }_{g_{1}}^{\left(
p,q\right) }\left( f_{1}\pm f_{2},\varphi \right) \neq \overline{\tau }%
_{g_{2}}^{\left( p,q\right) }\left( f_{1}\pm f_{2},\varphi \right) $;\newline
$\left( iii\right) $ Either $\tau _{g_{1}}^{\left( p,q\right) }\left(
f_{1},\varphi \right) \neq \tau _{g_{1}}^{\left( p,q\right) }\left(
f_{2},\varphi \right) $ or $\overline{\tau }_{g_{1}}^{\left( p,q\right)
}\left( f_{1},\varphi \right) \neq \overline{\tau }_{g_{1}}^{\left(
p,q\right) }\left( f_{2},\varphi \right) $;\newline
$\left( iv\right) $ Either $\tau _{g_{2}}^{\left( p,q\right) }\left(
f_{1},\varphi \right) \neq \tau _{g_{2}}^{\left( p,q\right) }\left(
f_{2},\varphi \right) $ or $\overline{\tau }_{g_{2}}^{\left( p,q\right)
}\left( f_{1},\varphi \right) \neq \overline{\tau }_{g_{2}}^{\left(
p,q\right) }\left( f_{2},\varphi \right) $; then%
\begin{equation*}
\lambda _{g_{1}\pm g_{2}}^{\left( p,q\right) }\left( f_{1}\pm f_{2},\varphi
\right) =\lambda _{g_{1}}^{\left( p,q\right) }\left( f_{1},\varphi \right)
=\lambda _{g_{1}}^{\left( p,q\right) }\left( f_{2},\varphi \right) =\lambda
_{g_{2}}^{\left( p,q\right) }\left( f_{1},\varphi \right) =\lambda
_{g_{2}}^{\left( p,q\right) }\left( f_{2},\varphi \right) ~.
\end{equation*}%
\textbf{(B)} The following conditions are assumed to be satisfied:\newline
$\left( i\right) $ At least any one of $f_{1}$ or $f_{2}$ are of regular
relative $\left( p,q\right) $-$\varphi $ growth with respect to $g_{1}\pm
g_{2}$, and $g_{1}\pm g_{2}$ has satisfy the Property (A);\newline
$\left( ii\right) $ Either $\tau _{g_{1}\pm g_{2}}^{\left( p,q\right)
}\left( f_{1},\varphi \right) \neq \tau _{g_{1}\pm g_{2}}^{\left( p,q\right)
}\left( f_{2},\varphi \right) $ or $\overline{\tau }_{g_{1}\pm
g_{2}}^{\left( p,q\right) }\left( f_{1},\varphi \right) \neq \overline{\tau }%
_{g_{1}\pm g_{2}}^{\left( p,q\right) }\left( f_{2},\varphi \right) $ holds;%
\newline
$\left( iii\right) $ Either $\tau _{g_{1}}^{\left( p,q\right) }\left(
f_{1},\varphi \right) \neq \tau _{g_{2}}^{\left( p,q\right) }\left(
f_{1},\varphi \right) $ or $\overline{\tau }_{g_{1}}^{\left( p,q\right)
}\left( f_{1},\varphi \right) \neq \overline{\tau }_{g_{2}}^{\left(
p,q\right) }\left( f_{1},\varphi \right) $ holds;\newline
$\left( iv\right) $ Either $\tau _{g_{1}}^{\left( p,q\right) }\left(
f_{2},\varphi \right) \neq \tau _{g_{2}}^{\left( p,q\right) }\left(
f_{2},\varphi \right) $ or $\overline{\tau }_{g_{1}}^{\left( p,q\right)
}\left( f_{2},\varphi \right) \neq \overline{\tau }_{g_{2}}^{\left(
p,q\right) }\left( f_{2},\varphi \right) $ holds, then%
\begin{equation*}
\lambda _{g_{1}\pm g_{2}}^{\left( p,q\right) }\left( f_{1}\pm f_{2},\varphi
\right) =\lambda _{g_{1}}^{\left( p,q\right) }\left( f_{1},\varphi \right)
=\lambda _{g_{1}}^{\left( p,q\right) }\left( f_{2},\varphi \right) =\lambda
_{g_{2}}^{\left( p,q\right) }\left( f_{1},\varphi \right) =\lambda
_{g_{2}}^{\left( p,q\right) }\left( f_{2},\varphi \right) ~.
\end{equation*}
\end{theorem}

\qquad We omit the proof of Theorem \ref{t9.16A} as it is a natural
consequence of Theorem \ref{t9.16}.

\begin{theorem}
\label{t9.17} Let $f_{1},f_{2}$ be any two meromorphic functions and $g_{1}$%
, $g_{2}$ be any two entire functions. Also let $\rho _{g_{1}}^{\left(
p,q\right) }\left( f_{1},\varphi \right) $, $\rho _{g_{1}}^{\left(
p,q\right) }\left( f_{2},\varphi \right) $, $\rho _{g_{2}}^{\left(
p,q\right) }\left( f_{1},\varphi \right) $ and $\rho _{g_{2}}^{\left(
p,q\right) }\left( f_{2},\varphi \right) $ are all non zero and finite.%
\newline
\textbf{(A)} Assume the functions $f_{1},f_{2}$ and $g_{1}$ satisfy the
following conditions:\newline
$\left( i\right) $ $\rho _{g_{1}}^{\left( p,q\right) }\left( f_{i},\varphi
\right) >\rho _{g_{1}}^{\left( p,q\right) }\left( f_{j},\varphi \right) $
for $i$, $j$ $=$ $1,2$ and $i\neq j$;\newline
$\left( ii\right) $ $g_{1}$ satisfies the Property (A), then%
\begin{equation*}
\sigma _{g_{1}}^{\left( p,q\right) }\left( f_{1}\cdot f_{2},\varphi \right)
=\sigma _{g_{1}}^{\left( p,q\right) }\left( f_{i},\varphi \right) \text{ and
\ }\overline{\sigma }_{g_{1}}^{\left( p,q\right) }\left( f_{1}\cdot
f_{2},\varphi \right) =\overline{\sigma }_{g_{1}}^{\left( p,q\right) }\left(
f_{i},\varphi \right) \mid i=1,2\text{ }~.
\end{equation*}%
Similarly,%
\begin{equation*}
\sigma _{g_{1}}^{\left( p,q\right) }\left( \frac{f_{1}}{f_{2}},\varphi
\right) =\sigma _{g_{1}}^{\left( p,q\right) }\left( f_{i},\varphi \right) 
\text{ and \ }\overline{\sigma }_{g_{1}}^{\left( p,q\right) }\left( \frac{%
f_{1}}{f_{2}},\varphi \right) =\overline{\sigma }_{g_{1}}^{\left( p,q\right)
}\left( f_{i},\varphi \right) \mid i=1,2
\end{equation*}%
holds provided $\left( i\right) $ $\frac{f_{1}}{f_{2}}$ is meromorphic, $%
\left( ii\right) $ $\rho _{g_{1}}^{\left( p,q\right) }\left( f_{i},\varphi
\right) $ $>$ $\rho _{g_{1}}^{\left( p,q\right) }\left( f_{j},\varphi
\right) $ $\mid $ $i$, $1,2;$ $j$ $=$ $1,2;$ $i$ $\neq $ $j$ and $\left(
iii\right) $ $g_{1}$ satisfy the Property (A).\newline
\textbf{(B) }Assume the functions $g_{1},g_{2}$ and $f_{1}$ satisfy the
following conditions:\newline
$\left( i\right) $ $\rho _{g_{i}}^{\left( p,q\right) }\left( f_{1},\varphi
\right) <\rho _{g_{j}}^{\left( p,q\right) }\left( f_{1},\varphi \right) $
with at least $f_{1}$ is of regular relative $\left( p,q\right) $-$\varphi $
growth with respect to $g_{j}$ for $i$, $j$ $=$ $1,2$ and $i\neq j,$ and $%
g_{i}$ satisfy the Property (A);\newline
$\left( ii\right) $ $g_{1}\cdot g_{2}$ satisfy the Property (A), then%
\begin{equation*}
\sigma _{g_{1}\cdot g_{2}}^{\left( p,q\right) }\left( f_{1},\varphi \right)
=\sigma _{g_{i}}^{\left( p,q\right) }\left( f_{1},\varphi \right) \text{ and
\ }\overline{\sigma }_{g_{1}\cdot g_{2}}^{\left( p,q\right) }\left(
f_{1},\varphi \right) =\overline{\sigma }_{g_{i}}^{\left( p,q\right) }\left(
f_{1},\varphi \right) \mid i=1,2~.
\end{equation*}%
Similarly,%
\begin{equation*}
\sigma _{\frac{g_{1}}{g_{2}}}^{\left( p,q\right) }\left( f_{1},\varphi
\right) =\sigma _{g_{i}}^{\left( p,q\right) }\left( f_{1},\varphi \right) 
\text{ and \ }\overline{\sigma }_{\frac{g_{1}}{g_{2}}}^{\left( p,q\right)
}\left( f_{1},\varphi \right) =\overline{\sigma }_{gi}^{\left( p,q\right)
}\left( f_{1},\varphi \right) \mid i=1,2
\end{equation*}%
holds provided $\left( i\right) $ $\frac{g_{1}}{g_{2}}$ is entire and
satisfy the Property (A), $\left( ii\right) $ At least $f_{1}$ is of regular
relative $\left( p,q\right) $-$\varphi $ growth with respect to $g_{2}$, $%
\left( iii\right) $ $\rho _{g_{i}}^{\left( p,q\right) }\left( f_{1},\varphi
\right) <\rho _{g_{j}}^{\left( p,q\right) }\left( f_{1},\varphi \right) $ $%
\mid $ $i$ $=$ $1,2;$ $j$ $=$ $1,2;$ $i$ $\neq $ $j$ and $\left( iv\right) $ 
$g_{1}$ satisfy the Property (A).\newline
\textbf{(C)} Assume the functions $f_{1},f_{2}$, $g_{1}$ and $g_{2}$ satisfy
the following conditions:\newline
$\left( i\right) $ $g_{1}\cdot g_{2}$ satisfy the Property (A);\newline
$\left( ii\right) $ $\rho _{g_{i}}^{\left( p,q\right) }\left( f_{1},\varphi
\right) <\rho _{g_{j}}^{\left( p,q\right) }\left( f_{1},\varphi \right) $
with at least $f_{1}$ is of regular relative $\left( p,q\right) $-$\varphi $
growth with respect to $g_{j}$ for $i$ $=$ $1,$ $2$, $j$ $=$ $1,2$ and $%
i\neq j$;\newline
$\left( iii\right) $ $\rho _{g_{i}}^{\left( p,q\right) }\left( f_{2},\varphi
\right) <\rho _{g_{j}}^{\left( p,q\right) }\left( f_{2},\varphi \right) $
with at least $f_{2}$ is of regular relative $\left( p,q\right) $-$\varphi $
growth with respect to $g_{j}$ for $i$ $=$ $1,$ $2$, $j$ $=$ $1,2$ and $%
i\neq j$;\newline
$\left( iv\right) $ $\rho _{g_{1}}^{\left( p,q\right) }\left( f_{i},\varphi
\right) >\rho _{g_{1}}^{\left( p,q\right) }\left( f_{j},\varphi \right) $
and $\rho _{g_{2}}^{\left( p,q\right) }\left( f_{i},\varphi \right) >\rho
_{g_{2}}^{\left( p,q\right) }\left( f_{j},\varphi \right) $ holds
simultaneously for $i=1,2;$ $j=1,2\ $and $i\neq j$; \newline
$\left( v\right) $ $\rho _{g_{m}}^{\left( p,q\right) }\left( f_{l},\varphi
\right) =\max \left[ \min \left\{ \rho _{g_{1}}^{\left( p,q\right) }\left(
f_{1},\varphi \right) ,\rho _{g_{2}}^{\left( p,q\right) }\left(
f_{1},\varphi \right) \right\} ,\min \left\{ \rho _{g_{1}}^{\left(
p,q\right) }\left( f_{2},\varphi \right) ,\rho _{g_{2}}^{\left( p,q\right)
}\left( f_{2},\varphi \right) \right\} \right] \mid l,m=1,2$; then%
\begin{equation*}
\sigma _{g_{1}\cdot g_{2}}^{\left( p,q\right) }\left( f_{1}\cdot
f_{2},\varphi \right) =\sigma _{g_{m}}^{\left( p,q\right) }\left(
f_{l},\varphi \right) \text{ and }\overline{\sigma }_{g_{1}\cdot
g_{2}}^{\left( p,q\right) }\left( f_{1}\cdot f_{2},\varphi \right) =%
\overline{\sigma }_{g_{m}}^{\left( p,q\right) }\left( f_{l},\varphi \right)
\mid l,m=1,2~.
\end{equation*}%
Similarly,%
\begin{equation*}
\sigma _{\frac{g_{1}}{g_{2}}}^{\left( p,q\right) }\left( \frac{f_{1}}{f_{2}}%
,\varphi \right) =\sigma _{g_{m}}^{\left( p,q\right) }\left( f_{l},\varphi
\right) \text{ and }\overline{\sigma }_{\frac{g_{1}}{g_{2}}}^{\left(
p,q\right) }\left( \frac{f_{1}}{f_{2}},\varphi \right) =\overline{\sigma }%
_{g_{m}}^{\left( p,q\right) }\left( f_{l},\varphi \right) \mid l,m=1,2.
\end{equation*}%
holds provided $\frac{f_{1}}{f_{2}}$ is meromorphic function and $\frac{g_{1}%
}{g_{2}}$ is entire function which satisfy the following conditions:\newline
$\left( i\right) $ $\frac{g_{1}}{g_{2}}$ satisfy the Property (A);\newline
$\left( ii\right) $ At least $f_{1}$ is of regular relative $\left(
p,q\right) $-$\varphi $ growth with respect to $g_{2}$ and $\rho
_{g_{1}}^{\left( p,q\right) }\left( f_{1},\varphi \right) \neq \rho
_{g_{2}}^{\left( p,q\right) }\left( f_{1},\varphi \right) $;\newline
$\left( iii\right) $ At least $f_{2}$ is of regular relative $\left(
p,q\right) $-$\varphi $ growth with respect to $g_{2}$ and $\rho
_{g_{1}}^{\left( p,q\right) }\left( f_{2},\varphi \right) \neq \rho
_{g_{2}}^{\left( p,q\right) }\left( f_{2},\varphi \right) $;\newline
$\left( iv\right) $ $\rho _{g_{1}}^{\left( p,q\right) }\left( f_{i},\varphi
\right) <\rho _{g_{1}}^{\left( p,q\right) }\left( f_{j},\varphi \right) $
and $\rho _{g_{2}}^{\left( p,q\right) }\left( f_{i},\varphi \right) <\rho
_{g_{2}}^{\left( p,q\right) }\left( f_{j},\varphi \right) $ holds
simultaneously for $i=1,2;$ $j=1,2\ $and $i\neq j$;\newline
$\left( v\right) $ $\rho _{g_{m}}^{\left( p,q\right) }\left( f_{l},\varphi
\right) =\max \left[ \min \left\{ \rho _{g_{1}}^{\left( p,q\right) }\left(
f_{1},\varphi \right) ,\rho _{g_{2}}^{\left( p,q\right) }\left(
f_{1},\varphi \right) \right\} ,\min \left\{ \rho _{g_{1}}^{\left(
p,q\right) }\left( f_{2},\varphi \right) ,\rho _{g_{2}}^{\left( p,q\right)
}\left( f_{2},\varphi \right) \right\} \right] \mid l,m=1,2$.
\end{theorem}

\begin{proof}
Let us suppose that $\rho _{g_{1}}^{\left( p,q\right) }\left( f_{1},\varphi
\right) $, $\rho _{g_{1}}^{\left( p,q\right) }\left( f_{2},\varphi \right) $%
, $\rho _{g_{2}}^{\left( p,q\right) }\left( f_{1},\varphi \right) $ and $%
\rho _{g_{2}}^{\left( p,q\right) }\left( f_{2},\varphi \right) $ are all non
zero and finite.\medskip \textbf{\newline
Case I.} Suppose that $\rho _{g_{1}}^{\left( p,q\right) }\left(
f_{1},\varphi \right) >\rho _{g_{1}}^{\left( p,q\right) }\left(
f_{2},\varphi \right) $. Also let $g_{1}$ satisfy the Property (A). \ Since $%
T_{f_{1}\cdot f_{2}}\left( r\right) \leq T_{f_{1}}\left( r\right)
+T_{f_{2}}\left( r\right) $ for all large $r,$ therefore applying the same
procedure as adopted in Case I of Theorem \ref{t9.13} we get that%
\begin{equation}
\sigma _{g_{1}}^{\left( p,q\right) }\left( f_{1}\cdot f_{2},\varphi \right)
\leq \sigma _{g_{1}}^{\left( p,q\right) }\left( f_{1},\varphi \right) ~.
\label{77.3}
\end{equation}

\qquad Further without loss of any generality, let $f=f_{1}\cdot f_{2}$ and $%
\rho _{g_{1}}^{\left( p,q\right) }\left( f_{2},\varphi \right) $ $<$ $\rho
_{g_{1}}^{\left( p,q\right) }\left( f_{1},\varphi \right) $ $=$ $\rho
_{g_{1}}^{\left( p,q\right) }\left( f,\varphi \right) .$ Then in view of $%
\left( \ref{77.3}\right) ,$ we obtain that $\sigma _{g_{1}}^{\left(
p,q\right) }\left( f,\varphi \right) $ $=$ $\sigma _{g_{1}}^{\left(
p,q\right) }\left( f_{1}\cdot f_{2},\varphi \right) $ $\leq $ $\sigma
_{g_{1}}^{\left( p,q\right) }\left( f_{1},\varphi \right) .$ Also $f_{1}=%
\frac{f}{f_{2}}$ and $T_{f_{2}}\left( r\right) $ $=$ $T_{\frac{1}{f_{2}}%
}\left( r\right) $ $+$ $O(1).$ Therefore $T_{f_{1}}\left( r\right) \leq
T_{f}\left( r\right) +T_{f_{2}}\left( r\right) +O(1)$ and in this case also
we obtain from $\left( \ref{77.3}\right) $ that $\sigma _{g_{1}}^{\left(
p,q\right) }\left( f_{1},\varphi \right) $ $\leq $ $\sigma _{g_{1}}^{\left(
p,q\right) }\left( f,\varphi \right) $ $=$ $\sigma _{g_{1}}^{\left(
p,q\right) }\left( f_{1}\cdot f_{2},\varphi \right) .$ Hence $\sigma
_{g_{1}}^{\left( p,q\right) }\left( f,\varphi \right) $ $=$ $\sigma
_{g_{1}}^{\left( p,q\right) }\left( f_{1},\varphi \right) $ $\Rightarrow $ $%
\sigma _{g_{1}}^{\left( p,q\right) }\left( f_{1}\cdot f_{2},\varphi \right) $
$=$ $\sigma _{g_{1}}^{\left( p,q\right) }\left( f_{1},\varphi \right) .$

\qquad Similarly, if we consider $\rho _{g_{1}}^{\left( p,q\right) }\left(
f_{1},\varphi \right) <\rho _{g_{1}}^{\left( p,q\right) }\left(
f_{2},\varphi \right) ,$ then one can verify that $\sigma _{g_{1}}^{\left(
p,q\right) }\left( f_{1}\cdot f_{2},\varphi \right) $ $=$ $\sigma
_{g_{1}}^{\left( p,q\right) }\left( f_{2},\varphi \right) .$

\qquad Next we may suppose that $f=\frac{f_{1}}{f_{2}}$ with $f_{1},$ $f_{2}$
and $f$ are all meromorphic functions.\medskip \newline
\textbf{Sub Case I}$_{\mathbf{A}}$\textbf{.} Let $\rho _{g_{1}}^{\left(
p,q\right) }\left( f_{2},\varphi \right) $ $<$ $\rho _{g_{1}}^{\left(
p,q\right) }\left( f_{1},\varphi \right) $. Therefore in view of Theorem \ref%
{t9.7}, $\rho _{g_{1}}^{\left( p,q\right) }\left( f_{2},\varphi \right) $ $<$
$\rho _{g_{1}}^{\left( p,q\right) }\left( f_{1},\varphi \right) $ $=$ $\rho
_{g_{1}}^{\left( p,q\right) }\left( f,\varphi \right) $. We have $%
f_{1}=f\cdot f_{2}$. So, $\sigma _{g_{1}}^{\left( p,q\right) }\left(
f_{1},\varphi \right) $ $=$ $\sigma _{g_{1}}^{\left( p,q\right) }\left(
f,\varphi \right) $ $=$ $\sigma _{g_{1}}^{\left( p,q\right) }\left( \frac{%
f_{1}}{f_{2}},\varphi \right) $.\medskip \newline
\textbf{Sub Case I}$_{\mathbf{B}}$\textbf{. }Let $\rho _{g_{1}}^{\left(
p,q\right) }\left( f_{2},\varphi \right) $ $>$ $\rho _{g_{1}}^{\left(
p,q\right) }\left( f_{1},\varphi \right) $. Therefore in view of Theorem \ref%
{t9.7}, $\rho _{g_{1}}^{\left( p,q\right) }\left( f_{1},\varphi \right) $ $<$
$\rho _{g_{1}}^{\left( p,q\right) }\left( f_{2},\varphi \right) $ $=$ $\rho
_{g_{1}}^{\left( p,q\right) }\left( f,\varphi \right) $. Since $T_{f}\left(
r\right) =T_{\frac{1}{f}}\left( r\right) +O(1)=T_{\frac{f_{2}}{f_{1}}}\left(
r\right) +O(1),$ So $\sigma _{g_{1}}^{\left( p,q\right) }\left( \frac{f_{1}}{%
f_{2}},\varphi \right) $ $=$ $\sigma _{g_{1}}^{\left( p,q\right) }\left(
f_{2},\varphi \right) $.\medskip \newline
\textbf{Case II. }Let $\rho _{g_{1}}^{\left( p,q\right) }\left(
f_{1},\varphi \right) >\rho _{g_{1}}^{\left( p,q\right) }\left(
f_{2},\varphi \right) $. Also let $g_{1}$ satisfy the Property (A). As $%
T_{f_{1}\cdot f_{2}}\left( r\right) \leq T_{f_{1}}\left( r\right)
+T_{f_{2}}\left( r\right) $ for all large $r,$ therefore applying the same
procedure as explored in Case II of Theorem \ref{t9.13}, one can easily
verify that $\overline{\sigma }_{g_{1}}^{\left( p,q\right) }\left(
f_{1}\cdot f_{2},\varphi \right) =\overline{\sigma }_{g_{1}}^{\left(
p,q\right) }\left( f_{1},\varphi \right) $ and $\overline{\sigma }%
_{g_{1}}^{\left( p,q\right) }\left( \frac{f_{1}}{f_{2}},\varphi \right) $ $=$
$\overline{\sigma }_{g_{1}}^{\left( p,q\right) }\left( f_{i},\varphi \right)
\mid i=1,2$ \ under the conditions specified in the theorem.

\qquad Similarly, if we consider $\rho _{g_{1}}^{\left( p,q\right) }\left(
f_{1},\varphi \right) <\rho _{g_{1}}^{\left( p,q\right) }\left(
f_{2},\varphi \right) ,$ then one can verify that $\overline{\sigma }%
_{g_{1}}^{\left( p,q\right) }\left( f_{1}\cdot f_{2},\varphi \right) $ $=$ $%
\overline{\sigma }_{g_{1}}^{\left( p,q\right) }\left( f_{2},\varphi \right) $
and $\overline{\sigma }_{g_{1}}^{\left( p,q\right) }\left( \frac{f_{1}}{f_{2}%
},\varphi \right) $ $=$ $\overline{\sigma }_{g_{1}}^{\left( p,q\right)
}\left( f_{2},\varphi \right) .$

\qquad Therefore the first part of theorem follows from Case I and Case
II.\medskip \newline
\textbf{Case III. }Let $g_{1}\cdot g_{2}$ satisfy the Property (A) and $\rho
_{g_{1}}^{\left( p,q\right) }\left( f_{1},\varphi \right) <\rho
_{g_{2}}^{\left( p,q\right) }\left( f_{1},\varphi \right) $ with at least $%
f_{1}$ is of regular relative $\left( p,q\right) $-$\varphi $ growth with
respect to $g_{2}.$ Since $T_{g_{1}\cdot g_{2}}\left( r\right) \leq
T_{g_{1}}\left( r\right) +T_{g_{2}}\left( r\right) $ for all large $r,$
therefore applying the same procedure as adopted in Case III of Theorem \ref%
{t9.13} we get that%
\begin{equation}
\sigma _{g_{1}\cdot g_{2}}^{\left( p,q\right) }\left( f_{1},\varphi \right)
\geq \sigma _{g_{1}}^{\left( p,q\right) }\left( f_{1},\varphi \right) ~.
\label{7.6j}
\end{equation}

\qquad Further without loss of any generality, let $g=g_{1}\cdot g_{2}$ and $%
\rho _{g}^{\left( p,q\right) }\left( f_{1},\varphi \right) $ $=$ $\rho
_{g_{1}}^{\left( p,q\right) }\left( f_{1},\varphi \right) $ $<$ $\rho
_{g_{2}}^{\left( p,q\right) }\left( f_{1},\varphi \right) .$ Then in view of 
$\left( \ref{7.6j}\right) ,$ we obtain that $\sigma _{g}^{\left( p,q\right)
}\left( f_{1},\varphi \right) $ $=$ $\sigma _{g_{1}\cdot g_{2}}^{\left(
p,q\right) }\left( f_{1},\varphi \right) $ $\geq $ $\sigma _{g_{1}}^{\left(
p,q\right) }\left( f_{1},\varphi \right) $. Also $g_{1}=\frac{g}{g_{2}}$ and 
$T_{g_{2}}\left( r\right) $ $=$ $T_{\frac{1}{g_{2}}}\left( r\right) $ $+$ $%
O(1).$ Therefore $T_{g_{1}}\left( r\right) \leq T_{g}\left( r\right)
+T_{g_{2}}\left( r\right) +O(1)$ and in this case we obtain from $\left( \ref%
{7.6j}\right) $ that $\sigma _{g_{1}}^{\left( p,q\right) }\left(
f_{1},\varphi \right) $ $\geq $ \ $\sigma _{g}^{\left( p,q\right) }\left(
f_{1},\varphi \right) $ $=$ $\sigma _{g_{1}\cdot g_{2}}^{\left( p,q\right)
}\left( f_{1},\varphi \right) $. Hence $\sigma _{g}^{\left( p,q\right)
}\left( f_{1},\varphi \right) $ $=$ $\sigma _{g_{1}}^{\left( p,q\right)
}\left( f_{1},\varphi \right) $ $\Rightarrow $ $\sigma _{g_{1}\cdot
g_{2}}^{\left( p,q\right) }\left( f_{1},\varphi \right) $ $=$ $\sigma
_{g_{1}}^{\left( p,q\right) }\left( f_{1},\varphi \right) $.

\qquad Similarly, if we consider $\rho _{g_{1}}^{\left( p,q\right) }\left(
f_{1},\varphi \right) >\rho _{g_{2}}^{\left( p,q\right) }\left(
f_{1},\varphi \right) $ with at least $f_{1}$ is of regular relative $\left(
p,q\right) $-$\varphi $ growth with respect to $g_{1}$, then one can verify
that $\sigma _{g_{1}\cdot g_{2}}^{\left( p,q\right) }\left( f_{1},\varphi
\right) $ $=$ $\sigma _{g_{2}}^{\left( p,q\right) }\left( f_{1},\varphi
\right) $.

\qquad Next we may suppose that $g=\frac{g_{1}}{g_{2}}$ with $g_{1},$ $%
g_{2}, $ $g$ are all entire functions satisfying the conditions specified in
the theorem.\medskip \newline
\textbf{Sub Case III}$_{\mathbf{A}}$\textbf{.} Let $\rho _{g_{1}}^{\left(
p,q\right) }\left( f_{1},\varphi \right) $ $<$ $\rho _{g_{2}}^{\left(
p,q\right) }\left( f_{1},\varphi \right) $. Therefore in view of Theorem \ref%
{t9.10A}, $\rho _{g}^{\left( p,q\right) }\left( f_{1},\varphi \right) $ $=$ $%
\rho _{g_{1}}^{\left( p,q\right) }\left( f_{1},\varphi \right) $ $<$ $\rho
_{g_{2}}^{\left( p,q\right) }\left( f_{1},\varphi \right) $. We have $%
g_{1}=g\cdot g_{2}$. So $\sigma _{g_{1}}^{\left( p,q\right) }\left(
f_{1},\varphi \right) $ $=$ $\sigma _{g}^{\left( p,q\right) }\left(
f_{1},\varphi \right) $ $=\sigma _{\frac{g_{1}}{g_{2}}}^{\left( p,q\right)
}\left( f_{1},\varphi \right) $.\medskip \newline
\textbf{Sub Case III}$_{\mathbf{B}}$\textbf{. }Let $\rho _{g_{1}}^{\left(
p,q\right) }\left( f_{1},\varphi \right) $ $>$ $\rho _{g_{2}}^{\left(
p,q\right) }\left( f_{1},\varphi \right) $. Therefore in view of Theorem \ref%
{t9.10A}, $\rho _{g}^{\left( p,q\right) }\left( f_{1},\varphi \right) $ $=$ $%
\rho _{g_{2}}^{\left( p,q\right) }\left( f_{1},\varphi \right) $ $<$ $\rho
_{g_{1}}^{\left( p,q\right) }\left( f_{1},\varphi \right) $. Since $%
T_{g}\left( r\right) =T_{\frac{1}{g}}\left( r\right) +O(1)=T_{\frac{g_{2}}{%
g_{1}}}\left( r\right) +O(1),$ So $\sigma _{\frac{g_{1}}{g_{2}}}^{\left(
p,q\right) }\left( f_{1},\varphi \right) $ $=$ $\sigma _{g_{2}}^{\left(
p,q\right) }\left( f_{1},\varphi \right) $.\medskip \newline
\textbf{Case IV. }Suppose $g_{1}\cdot g_{2}$ satisfy the Property (A). Also
let $\rho _{g_{1}}^{\left( p,q\right) }\left( f_{1},\varphi \right) <\rho
_{g_{2}}^{\left( p,q\right) }\left( f_{1},\varphi \right) $ with at least $%
f_{1}$ is of regular relative $\left( p,q\right) $-$\varphi $ growth with
respect to $g_{2}.$ As $T_{g_{1}\cdot g_{2}}\left( r\right) \leq
T_{g_{1}}\left( r\right) +T_{g_{2}}\left( r\right) $ for all large $r,$ the
same procedure as explored in Case IV of Theorem \ref{t9.13}, one can easily
verify that $\overline{\sigma }_{g_{1}\cdot g_{2}}^{\left( p,q\right)
}\left( f_{1},\varphi \right) $ $=$ $\overline{\sigma }_{g_{1}}^{\left(
p,q\right) }\left( f_{1},\varphi \right) $ and $\overline{\sigma }_{\frac{%
g_{1}}{g_{2}}}^{\left( p,q\right) }\left( f_{1},\varphi \right) $ $=$ $%
\overline{\sigma }_{g_{i}}^{\left( p,q\right) }\left( f_{1},\varphi \right)
\mid i=1,2$ under the conditions specified in the theorem.

\qquad Likewise, if we consider $\rho _{g_{1}}^{\left( p,q\right) }\left(
f_{1},\varphi \right) >\rho _{g_{2}}^{\left( p,q\right) }\left(
f_{1},\varphi \right) $ with at least $f_{1}$ is of regular relative $\left(
p,q\right) $-$\varphi $ growth with respect to $g_{1}$, then one can verify
that $\overline{\sigma }_{g_{1}\cdot g_{2}}^{\left( p,q\right) }\left(
f_{1},\varphi \right) $ $=$ $\overline{\sigma }_{g_{2}}^{\left( p,q\right)
}\left( f_{1},\varphi \right) $ and $\overline{\sigma }_{\frac{g_{1}}{g_{2}}%
}^{\left( p,q\right) }\left( f_{1},\varphi \right) $ $=$ $\overline{\sigma }%
_{g_{2}}^{\left( p,q\right) }\left( f_{1},\varphi \right) $. Therefore the
second part of theorem follows from Case III and Case IV.

\qquad Proof of the third part of the Theorem is omitted as it can be
carried out in view of Theorem \ref{t9.11} and Theorem \ref{t9.11A} and the
above cases.
\end{proof}

\begin{theorem}
\label{t9.18} Let $f_{1},f_{2}$ be any two meromorphic functions and $g_{1}$%
, $g_{2}$ be any two entire functions. Also let $\lambda _{g_{1}}^{\left(
p,q\right) }\left( f_{1},\varphi \right) $, $\lambda _{g_{1}}^{\left(
p,q\right) }\left( f_{2},\varphi \right) $, $\lambda _{g_{2}}^{\left(
p,q\right) }\left( f_{1},\varphi \right) $ and $\lambda _{g_{2}}^{\left(
p,q\right) }\left( f_{2},\varphi \right) $ are all non zero and finite.%
\newline
\textbf{(A)} Assume the functions $f_{1},f_{2}$ and $g_{1}$ satisfy the
following conditions:\newline
$\left( i\right) $ $\lambda _{g_{1}}^{\left( p,q\right) }\left(
f_{i},\varphi \right) >\lambda _{g_{1}}^{\left( p,q\right) }\left(
f_{j},\varphi \right) $ with at least $f_{j}$ is of regular relative $\left(
p,q\right) $-$\varphi $ growth with respect to $g_{1}$ for $i$, $j$ $=$ $1,2$
and $i\neq j$;\newline
$\left( ii\right) $ $g_{1}$ satisfy the Property (A)$,$ then%
\begin{equation*}
\tau _{g_{1}}^{\left( p,q\right) }\left( f_{1}\cdot f_{2},\varphi \right)
=\tau _{g_{1}}^{\left( p,q\right) }\left( f_{i},\varphi \right) \text{ and \ 
}\overline{\tau }_{g_{1}}^{\left( p,q\right) }\left( f_{1}\cdot
f_{2},\varphi \right) =\overline{\tau }_{g_{1}}^{\left( p,q\right) }\left(
f_{i},\varphi \right) \mid i=1,2~.
\end{equation*}%
Similarly,%
\begin{equation*}
\tau _{g_{1}}^{\left( p,q\right) }\left( \frac{f_{1}}{f_{2}},\varphi \right)
=\tau _{g_{1}}^{\left( p,q\right) }\left( f_{i},\varphi \right) \text{ and \ 
}\overline{\tau }_{g_{1}}^{\left( p,q\right) }\left( \frac{f_{1}}{f_{2}}%
,\varphi \right) =\overline{\tau }_{g_{1}}^{\left( p,q\right) }\left(
f_{i},\varphi \right) \mid i=1,2
\end{equation*}%
holds provided $\frac{f_{1}}{f_{2}}$ is meromorphic, at least $f_{2}$ is of
regular relative $\left( p,q\right) $-$\varphi $ growth with respect to $%
g_{1}$ where $g_{1}$ satisfy the Property (A) and $\lambda _{g_{1}}^{\left(
p,q\right) }\left( f_{i},\varphi \right) $ $>$ $\lambda _{g_{1}}^{\left(
p,q\right) }\left( f_{j},\varphi \right) $ $\mid $ $i$ $=$ $1,2;$ $j$ $=$ $%
1,2;$ $i$ $\neq $ $j$.\newline
\textbf{(B) }Assume the functions $g_{1},g_{2}$ and $f_{1}$ satisfy the
following conditions:\newline
$\left( i\right) $ $\lambda _{g_{i}}^{\left( p,q\right) }\left(
f_{1},\varphi \right) <\lambda _{g_{j}}^{\left( p,q\right) }\left(
f_{1},\varphi \right) $ for $i$, $j$ $=$ $1,2$, $i\neq j$; and $g_{i}$
satisfy the Property (A)\newline
$\left( ii\right) $ $g_{1}\cdot g_{2}$ satisfy the Property (A), then%
\begin{equation*}
\tau _{g_{1}\cdot g_{2}}^{\left( p,q\right) }\left( f_{1},\varphi \right)
=\tau _{g_{i}}^{\left( p,q\right) }\left( f_{1},\varphi \right) \text{ and \ 
}\overline{\tau }_{g_{1}\cdot g_{2}}^{\left( p,q\right) }\left(
f_{1},\varphi \right) =\overline{\tau }_{g_{i}}^{\left( p,q\right) }\left(
f_{1},\varphi \right) \mid i=1,2~.
\end{equation*}%
Similarly,%
\begin{equation*}
\tau _{\frac{g_{1}}{g_{2}}}^{\left( p,q\right) }\left( f_{1},\varphi \right)
=\tau _{g_{i}}^{\left( p,q\right) }\left( f_{1},\varphi \right) \text{ and \ 
}\overline{\tau }_{\frac{g_{1}}{g_{2}}}^{\left( p,q\right) }\left(
f_{1},\varphi \right) =\overline{\tau }_{gi}^{\left( p,q\right) }\left(
f_{1},\varphi \right) \mid i=1,2
\end{equation*}%
holds provided $\frac{g_{1}}{g_{2}}$ is entire and satisfy the Property (A), 
$g_{1}$ satisfy the Property (A) and $\lambda _{g_{i}}^{\left( p,q\right)
}\left( f_{1},\varphi \right) <\lambda _{g_{j}}^{\left( p,q\right) }\left(
f_{1},\varphi \right) $ $\mid $ $i$ $=$ $1,2;$ $j$ $=$ $1,2;$ $i$ $\neq $ $j$%
.\newline
\textbf{(C)} Assume the functions $f_{1},f_{2}$, $g_{1}$ and $g_{2}$ satisfy
the following conditions:\newline
$\left( i\right) $ $g_{1}\cdot g_{2}$, $g_{1}$ and $g_{2}$ are satisfy the
Property (A)$;$\newline
$\left( ii\right) $ $\lambda _{g_{1}}^{\left( p,q\right) }\left(
f_{i},\varphi \right) >\lambda _{g_{1}}^{\left( p,q\right) }\left(
f_{j},\varphi \right) $ with at least $f_{j}$ is of regular relative $\left(
p,q\right) $-$\varphi $ growth with respect to $g_{1}$ for $i$ $=$ $1,$ $2$, 
$j$ $=$ $1,2$ and $i\neq j$;\newline
$\left( iii\right) $ $\lambda _{g_{2}}^{\left( p,q\right) }\left(
f_{i},\varphi \right) >\lambda _{g_{2}}^{\left( p,q\right) }\left(
f_{j},\varphi \right) $ with at least $f_{j}$ is of regular relative $\left(
p,q\right) $-$\varphi $ growth with respect to $g_{2}$ for $i$ $=$ $1,$ $2$, 
$j$ $=$ $1,2$ and $i\neq j$;\newline
$\left( iv\right) $ $\lambda _{g_{i}}^{\left( p,q\right) }\left(
f_{1},\varphi \right) <\lambda _{g_{j}}^{\left( p,q\right) }\left(
f_{1},\varphi \right) $ and $\lambda _{g_{i}}^{\left( p,q\right) }\left(
f_{2},\varphi \right) <\lambda _{g_{j}}^{\left( p,q\right) }\left(
f_{2},\varphi \right) $ holds simultaneously for $i=1,2;$ $j=1,2\ $and $%
i\neq j$;\newline
$\left( v\right) $ $\lambda _{g_{m}}^{\left( p,q\right) }\left(
f_{l},\varphi \right) =\min \left[ \max \left\{ \lambda _{g_{1}}^{\left(
p,q\right) }\left( f_{1},\varphi \right) ,\lambda _{g_{1}}^{\left(
p,q\right) }\left( f_{2},\varphi \right) \right\} ,\max \left\{ \lambda
_{g_{2}}^{\left( p,q\right) }\left( f_{1},\varphi \right) ,\lambda
_{g_{2}}^{\left( p,q\right) }\left( f_{2},\varphi \right) \right\} \right]
\mid l,m=1,2$; then%
\begin{equation*}
\tau _{g_{1}\cdot g_{2}}^{\left( p,q\right) }\left( f_{1}\cdot f_{2},\varphi
\right) =\tau _{g_{m}}^{\left( p,q\right) }\left( f_{l},\varphi \right) 
\text{ and }\overline{\tau }_{g_{1}\cdot g_{2}}^{\left( p,q\right) }\left(
f_{1}\cdot f_{2},\varphi \right) =\overline{\tau }_{g_{m}}^{\left(
p,q\right) }\left( f_{l},\varphi \right) \mid l,m=1,2~.
\end{equation*}%
Similarly,%
\begin{equation*}
\tau _{\frac{g_{1}}{g_{2}}}^{\left( p,q\right) }\left( \frac{f_{1}}{f_{2}}%
,\varphi \right) =\tau _{g_{m}}^{\left( p,q\right) }\left( f_{l},\varphi
\right) \text{ and }\overline{\tau }_{\frac{g_{1}}{g_{2}}}^{\left(
p,q\right) }\left( \frac{f_{1}}{f_{2}},\varphi \right) =\overline{\tau }%
_{g_{m}}^{\left( p,q\right) }\left( f_{l},\varphi \right) \mid l,m=1,2~.
\end{equation*}%
holds provided $\frac{f_{1}}{f_{2}}$ is meromorphic and $\frac{g_{1}}{g_{2}}$
is entire functions which satisfy the following conditions:\newline
$\left( i\right) $ $\frac{g_{1}}{g_{2}}$, $g_{1}$ and $g_{2}$ satisfy the
Property (A);\newline
$\left( ii\right) $ At least $f_{2}$ is of regular relative $\left(
p,q\right) $-$\varphi $ growth with respect to $g_{1}$ and $\lambda
_{g_{1}}^{\left( p,q\right) }\left( f_{1},\varphi \right) \neq \lambda
_{g_{1}}^{\left( p,q\right) }\left( f_{2},\varphi \right) $;\newline
$\left( iii\right) $ At least $f_{2}$ is of regular relative $\left(
p,q\right) $-$\varphi $ growth with respect to $g_{2}$ and $\lambda
_{g_{2}}^{\left( p,q\right) }\left( f_{1},\varphi \right) \neq \lambda
_{g_{2}}^{\left( p,q\right) }\left( f_{2},\varphi \right) $;\newline
$\left( iv\right) $ $\lambda _{g_{i}}^{\left( p,q\right) }\left(
f_{1},\varphi \right) <\lambda _{g_{j}}^{\left( p,q\right) }\left(
f_{1},\varphi \right) $ and $\lambda _{g_{i}}^{\left( p,q\right) }\left(
f_{2},\varphi \right) <\lambda _{g_{j}}^{\left( p,q\right) }\left(
f_{2},\varphi \right) $ holds simultaneously for $i=1,2;$ $j=1,2\ $and $%
i\neq j$;\newline
$\left( v\right) $ $\lambda _{g_{m}}^{\left( p,q\right) }\left(
f_{l},\varphi \right) =\min \left[ \max \left\{ \lambda _{g_{1}}^{\left(
p,q\right) }\left( f_{1},\varphi \right) ,\lambda _{g_{1}}^{\left(
p,q\right) }\left( f_{2},\varphi \right) \right\} ,\max \left\{ \lambda
_{g_{2}}^{\left( p,q\right) }\left( f_{1},\varphi \right) ,\lambda
_{g_{2}}^{\left( p,q\right) }\left( f_{2},\varphi \right) \right\} \right]
\mid l,m=1,2$.
\end{theorem}

\begin{proof}
Let us consider that $\lambda _{g_{1}}^{\left( p,q\right) }\left(
f_{1},\varphi \right) $, $\lambda _{g_{1}}^{\left( p,q\right) }\left(
f_{2},\varphi \right) $, $\lambda _{g_{2}}^{\left( p,q\right) }\left(
f_{1},\varphi \right) $ and $\lambda _{g_{2}}^{\left( p,q\right) }\left(
f_{2},\varphi \right) $ are all non zero and finite.\textbf{\newline
Case I.} Suppose $\lambda _{g_{1}}^{\left( p,q\right) }\left( f_{1},\varphi
\right) >\lambda _{g_{1}}^{\left( p,q\right) }\left( f_{2},\varphi \right) $
with at least $f_{2}$ is of regular relative $\left( p,q\right) $-$\varphi $
growth with respect to $g_{1}$ and $g_{1}$ satisfy the Property (A). Since $%
T_{f_{1}\cdot f_{2}}\left( r\right) \leq T_{f_{1}}\left( r\right)
+T_{f_{2}}\left( r\right) $ for all large $r,$ therefore applying the same
procedure as adopted in Case I of Theorem \ref{t9.14} we get that%
\begin{equation}
\tau _{g_{1}}^{\left( p,q\right) }\left( f_{1}\cdot f_{2},\varphi \right)
\leq \tau _{g_{1}}^{\left( p,q\right) }\left( f_{1},\varphi \right) ~.
\label{77.90}
\end{equation}

\qquad Further without loss of any generality, let $f=f_{1}\cdot f_{2}$ and $%
\lambda _{g_{1}}^{\left( p,q\right) }\left( f_{2},\varphi \right) $ $<$ $%
\lambda _{g_{1}}^{\left( p,q\right) }\left( f_{1},\varphi \right) $ $=$ $%
\lambda _{g_{1}}^{\left( p,q\right) }\left( f,\varphi \right) .$ Then in
view of $\left( \ref{77.90}\right) ,$ we obtain that $\tau _{g_{1}}^{\left(
p,q\right) }\left( f,\varphi \right) $ $=$ $\tau _{g_{1}}^{\left( p,q\right)
}\left( f_{1}\cdot f_{2},\varphi \right) $ $\leq $ $\tau _{g_{1}}^{\left(
p,q\right) }\left( f_{1},\varphi \right) .$ Also $f_{1}=\frac{f}{f_{2}}$ and 
$T_{f_{2}}\left( r\right) $ $=$ $T_{\frac{1}{f_{2}}}\left( r\right) $ $+$ $%
O(1).$ Therefore $T_{f_{1}}\left( r\right) \leq T_{f}\left( r\right)
+T_{f_{2}}\left( r\right) +O(1)$ and in this case we obtain from the above
arguments that $\tau _{g_{1}}^{\left( p,q\right) }\left( f_{1},\varphi
\right) $ $\leq $ $\tau _{g_{1}}^{\left( p,q\right) }\left( f,\varphi
\right) $ $=$ $\tau _{g_{1}}^{\left( p,q\right) }\left( f_{1}\cdot
f_{2},\varphi \right) .$ Hence $\tau _{g_{1}}^{\left( p,q\right) }\left(
f,\varphi \right) $ $=$ $\tau _{g_{1}}^{\left( p,q\right) }\left(
f_{1},\varphi \right) $ $\Rightarrow $ $\tau _{g_{1}}^{\left( p,q\right)
}\left( f_{1}\cdot f_{2},\varphi \right) $ $=$ $\tau _{g_{1}}^{\left(
p,q\right) }\left( f_{1},\varphi \right) .$

\qquad Similarly, if we consider $\lambda _{g_{1}}^{\left( p,q\right)
}\left( f_{1},\varphi \right) <\lambda _{g_{1}}^{\left( p,q\right) }\left(
f_{2},\varphi \right) $ with at least $f_{1}$ is of regular relative $\left(
p,q\right) $-$\varphi $ growth with respect to $g_{1}$, then one can easily
verify that $\tau _{g_{1}}^{\left( p,q\right) }\left( f_{1}\cdot
f_{2},\varphi \right) $ $=$ $\tau _{g_{1}}^{\left( p,q\right) }\left(
f_{2},\varphi \right) .$

\qquad Next we may suppose that $f=\frac{f_{1}}{f_{2}}$ with $f_{1},$ $f_{2}$
and $f$ are all meromorphic functions satisfying the conditions specified in
the theorem.\medskip \newline
\textbf{Sub Case I}$_{\mathbf{A}}$\textbf{.} Let $\lambda _{g_{1}}^{\left(
p,q\right) }\left( f_{2},\varphi \right) $ $<\lambda $ $_{g_{1}}^{\left(
p,q\right) }\left( f_{1},\varphi \right) $. Therefore in view of Theorem \ref%
{t9.8 A}, $\lambda _{g_{1}}^{\left( p,q\right) }\left( f_{2},\varphi \right) 
$ $<$ $\lambda _{g_{1}}^{\left( p,q\right) }\left( f_{1},\varphi \right) $ $%
= $ $\lambda _{g_{1}}^{\left( p,q\right) }\left( f,\varphi \right) $. We
have $f_{1}=f\cdot f_{2}$. So $\tau _{g_{1}}^{\left( p,q\right) }\left(
f_{1},\varphi \right) $ $=$ $\tau _{g_{1}}^{\left( p,q\right) }\left(
f,\varphi \right) $ $=$ $\tau _{g_{1}}^{\left( p,q\right) }\left( \frac{f_{1}%
}{f_{2}},\varphi \right) $.\medskip \newline
\textbf{Sub Case I}$_{\mathbf{B}}$\textbf{. }Let $\lambda _{g_{1}}^{\left(
p,q\right) }\left( f_{2},\varphi \right) $ $>$ $\lambda _{g_{1}}^{\left(
p,q\right) }\left( f_{1},\varphi \right) $. Therefore in view of Theorem \ref%
{t9.8 A}, $\lambda _{g_{1}}^{\left( p,q\right) }\left( f_{1},\varphi \right) 
$ $<$ $\lambda _{g_{1}}^{\left( p,q\right) }\left( f_{2},\varphi \right) $ $%
= $ $\lambda _{g_{1}}^{\left( p,q\right) }\left( f,\varphi \right) $. Since $%
T_{f}\left( r\right) =T_{\frac{1}{f}}\left( r\right) +O(1)=T_{\frac{f_{2}}{%
f_{1}}}\left( r\right) +O(1),$ So $\tau _{g_{1}}^{\left( p,q\right) }\left( 
\frac{f_{1}}{f_{2}},\varphi \right) $ $=$ $\tau _{g_{1}}^{\left( p,q\right)
}\left( f_{2},\varphi \right) $.\medskip \newline
\textbf{Case II. }Let $\lambda _{g_{1}}^{\left( p,q\right) }\left(
f_{1},\varphi \right) >\lambda _{g_{1}}^{\left( p,q\right) }\left(
f_{2},\varphi \right) $ with at least $f_{2}$ is of regular relative $\left(
p,q\right) $-$\varphi $ growth with respect to $g_{1}$ where $g_{1}$ satisfy
the Property (A). As $T_{f_{1}\cdot f_{2}}\left( r\right) \leq
T_{f_{1}}\left( r\right) +T_{f_{2}}\left( r\right) $ for all large $r,$ so
applying the same procedure as adopted in Case II of Theorem \ref{t9.14} we
can easily verify that $\overline{\tau }_{g_{1}}^{\left( p,q\right) }\left(
f_{1}\cdot f_{2},\varphi \right) =\overline{\tau }_{g_{1}}^{\left(
p,q\right) }\left( f_{1},\varphi \right) $ and\ $\overline{\tau }_{\frac{%
g_{1}}{g_{2}}}^{\left( p,q\right) }\left( f_{1},\varphi \right) =\overline{%
\tau }_{gi}^{\left( p,q\right) }\left( f_{1},\varphi \right) \mid i=1,2$
under the conditions specified in the theorem.

\qquad Similarly, if we consider $\lambda _{g_{1}}^{\left( p,q\right)
}\left( f_{1},\varphi \right) <\lambda _{g_{1}}^{\left( p,q\right) }\left(
f_{2},\varphi \right) $ with at least $f_{1}$ is of regular relative $\left(
p,q\right) $-$\varphi $ growth with respect to $g_{1}$, then one can easily
verify that $\overline{\tau }_{g_{1}}^{\left( p,q\right) }\left( f_{1}\cdot
f_{2},\varphi \right) $ $=$ $\overline{\tau }_{g_{1}}^{\left( p,q\right)
}\left( f_{2},\varphi \right) .$

\qquad Therefore the first part of theorem follows Case I and Case II.%
\newline
\textbf{Case III. }Let $\lambda _{g_{1}}^{\left( p,q\right) }\left(
f_{1},\varphi \right) <\lambda _{g_{2}}^{\left( p,q\right) }\left(
f_{1},\varphi \right) $ and $g_{1}\cdot g_{2}$ satisfy the Property
(A).Since $T_{g_{1}\cdot g_{2}}\left( r\right) \leq T_{g_{1}}\left( r\right)
+T_{g_{2}}\left( r\right) $ for all large $r,$ therefore applying the same
procedure as adopted in Case III of Theorem \ref{t9.14} we get that 
\begin{equation}
\tau _{g_{1}\cdot g_{2}}^{\left( p,q\right) }\left( f_{1},\varphi \right)
\leq \tau _{g_{1}}^{\left( p,q\right) }\left( f_{1},\varphi \right) ~.
\label{79.0}
\end{equation}

\qquad Further without loss of any generality, let $g=g_{1}\cdot g_{2}$ and $%
\lambda _{g}^{\left( p,q\right) }\left( f_{1},\varphi \right) $ $=$ $\lambda
_{g_{1}}^{\left( p,q\right) }\left( f_{1},\varphi \right) $ $<$ $\lambda
_{g_{2}}^{\left( p,q\right) }\left( f_{1},\varphi \right) .$ Then in view of 
$\left( \ref{79.0}\right) ,$ we obtain that $\tau _{g}^{\left( p,q\right)
}\left( f_{1},\varphi \right) $ $=$ $\tau _{g_{1}\cdot g_{2}}^{\left(
p,q\right) }\left( f_{1},\varphi \right) $ $\geq $ $\tau _{g_{1}}^{\left(
p,q\right) }\left( f_{1},\varphi \right) $. Also $g_{1}=\frac{g}{g_{2}}$ and 
$T_{g_{2}}\left( r\right) $ $=$ $T_{\frac{1}{g_{2}}}\left( r\right) $ $+$ $%
O(1).$ Therefore $T_{g_{1}}\left( r\right) \leq T_{g}\left( r\right)
+T_{g_{2}}\left( r\right) +O(1)$ and in this case we obtain from above
arguments that $\tau _{g_{1}}^{\left( p,q\right) }\left( f_{1},\varphi
\right) $ $\geq $ \ $\tau _{g}^{\left( p,q\right) }\left( f_{1},\varphi
\right) $ $=$ $\tau _{g_{1}\cdot g_{2}}^{\left( p,q\right) }\left(
f_{1},\varphi \right) $. Hence $\tau _{g}^{\left( p,q\right) }\left(
f_{1},\varphi \right) $ $=$ $\tau _{g_{1}}^{\left( p,q\right) }\left(
f_{1},\varphi \right) $ $\Rightarrow $ $\tau _{g_{1}\cdot g_{2}}^{\left(
p,q\right) }\left( f_{1},\varphi \right) $ $=$ $\tau _{g_{1}}^{\left(
p,q\right) }\left( f_{1},\varphi \right) $.

\qquad If $\lambda _{g_{1}}^{\left( p,q\right) }\left( f_{1},\varphi \right)
>\lambda _{g_{2}}^{\left( p,q\right) }\left( f_{1},\varphi \right) ,$ then
one can easily verify that $\tau _{g_{1}\cdot g_{2}}^{\left( p,q\right)
}\left( f_{1},\varphi \right) $ $=$ $\tau _{g_{2}}^{\left( p,q\right)
}\left( f_{1},\varphi \right) $.

\qquad Next we may suppose that $g=\frac{g_{1}}{g_{2}}$ with $g_{1},$ $%
g_{2}, $ $g$ are all entire functions satisfying the conditions specified in
the theorem.\medskip \newline
\textbf{Sub Case III}$_{\mathbf{A}}$\textbf{.} Let $\lambda _{g_{1}}^{\left(
p,q\right) }\left( f_{1},\varphi \right) $ $<$ $\lambda _{g_{2}}^{\left(
p,q\right) }\left( f_{1},\varphi \right) $. Therefore in view of Theorem \ref%
{t9.9}, $\lambda _{g}^{\left( p,q\right) }\left( f_{1},\varphi \right) $ $=$ 
$\lambda _{g_{1}}^{\left( p,q\right) }\left( f_{1},\varphi \right) $ $<$ $%
\lambda _{g_{2}}^{\left( p,q\right) }\left( f_{1},\varphi \right) $. We have 
$g_{1}=g\cdot g_{2}$. So $\tau _{g_{1}}^{\left( p,q\right) }\left(
f_{1},\varphi \right) $ $=$ $\tau _{g}^{\left( p,q\right) }\left(
f_{1},\varphi \right) $ $=\tau _{\frac{g_{1}}{g_{2}}}^{\left( p,q\right)
}\left( f_{1},\varphi \right) $.\medskip \newline
\textbf{Sub Case III}$_{\mathbf{B}}$\textbf{. }Let $\lambda _{g_{1}}^{\left(
p,q\right) }\left( f_{1},\varphi \right) $ $>$ $\lambda _{g_{2}}^{\left(
p,q\right) }\left( f_{1},\varphi \right) $. Therefore in view of Theorem \ref%
{t9.9}, $\lambda _{g}^{\left( p,q\right) }\left( f_{1},\varphi \right) $ $=$ 
$\lambda _{g_{2}}^{\left( p,q\right) }\left( f_{1},\varphi \right) $ $<$ $%
\lambda _{g_{1}}^{\left( p,q\right) }\left( f_{1},\varphi \right) $. Since $%
T_{g}\left( r\right) =T_{\frac{1}{g}}\left( r\right) +O(1)=T_{\frac{g_{2}}{%
g_{1}}}\left( r\right) +O(1),$ So $\tau _{\frac{g_{1}}{g_{2}}}^{\left(
p,q\right) }\left( f_{1},\varphi \right) $ $=$ $\tau _{g_{2}}^{\left(
p,q\right) }\left( f_{1},\varphi \right) $.\medskip \newline
\textbf{Case IV. }Suppose $\lambda _{g_{1}}^{\left( p,q\right) }\left(
f_{1},\varphi \right) <\lambda _{g_{2}}^{\left( p,q\right) }\left(
f_{1},\varphi \right) $ and $g_{1}\cdot g_{2}$ satisfy the Property (A).
Since $T_{g_{1}\cdot g_{2}}\left( r\right) \leq T_{g_{1}}\left( r\right)
+T_{g_{2}}\left( r\right) $ for all large $r,$ then adopting the same
procedure as of Case IV of Theorem \ref{t9.14}, we obtain that $\overline{%
\tau }_{g_{1}\cdot g_{2}}^{\left( p,q\right) }\left( f_{1},\varphi \right) =%
\overline{\tau }_{g_{1}}^{\left( p,q\right) }\left( f_{1},\varphi \right) $
and $\overline{\tau }_{\frac{g_{1}}{g_{2}}}^{\left( p,q\right) }\left(
f_{1},\varphi \right) =\overline{\tau }_{gi}^{\left( p,q\right) }\left(
f_{1},\varphi \right) \mid i=1,2$.

\qquad Similarly if we consider that $\lambda _{g_{1}}^{\left( p,q\right)
}\left( f_{1},\varphi \right) >\lambda _{g_{2}}^{\left( p,q\right) }\left(
f_{1},\varphi \right) ,$ then one can easily verify that $\overline{\tau }%
_{g_{1}\cdot g_{2}}^{\left( p,q\right) }\left( f_{1},\varphi \right) $ $=$ $%
\overline{\tau }_{g_{2}}^{\left( p,q\right) }\left( f_{1},\varphi \right) $.

\qquad Therefore the second part of the theorem follows from Case III and
Case IV.

\qquad Proof of the third part of the Theorem is omitted as it can be
carried out in view of Theorem \ref{t9.12} , Theorem \ref{t9.12A} and the
above cases.
\end{proof}

\begin{theorem}
\label{t9.19} Let $f_{1},\,f_{2}$ be any two meromorphic functions and $%
g_{1} $, $g_{2}$ be any two entire functions.\newline
\textbf{(A)} The following condition is assumed to be satisfied:\newline
$\left( i\right) $ Either $\sigma _{g_{1}}^{\left( p,q\right) }\left(
f_{1},\varphi \right) \neq \sigma _{g_{1}}^{\left( p,q\right) }\left(
f_{2},\varphi \right) $ or $\overline{\sigma }_{g_{1}}^{\left( p,q\right)
}\left( f_{1},\varphi \right) \neq \overline{\sigma }_{g_{1}}^{\left(
p,q\right) }\left( f_{2},\varphi \right) $ holds;\newline
$\left( ii\right) $ $g_{1}$ satisfies the Property (A), then%
\begin{equation*}
\rho _{g_{1}}^{\left( p,q\right) }\left( f_{1}\cdot f_{2},\varphi \right)
=\rho _{g_{1}}^{\left( p,q\right) }\left( f_{1},\varphi \right) =\rho
_{g_{1}}^{\left( p,q\right) }\left( f_{2},\varphi \right) ~.
\end{equation*}%
\textbf{(B)} The following conditions are assumed to be satisfied:\newline
$\left( i\right) $ Either $\sigma _{g_{1}}^{\left( p,q\right) }\left(
f_{1},\varphi \right) \neq \sigma _{g_{2}}^{\left( p,q\right) }\left(
f_{1},\varphi \right) $ or $\overline{\sigma }_{g_{1}}^{\left( p,q\right)
}\left( f_{1},\varphi \right) \neq \overline{\sigma }_{g_{2}}^{\left(
p,q\right) }\left( f_{1},\varphi \right) $ holds;\newline
$\left( ii\right) $ $f_{1}$ is of regular relative $\left( p,q\right) $-$%
\varphi $ growth with respect to at least any one of $g_{1}$ or $g_{2}$.
Also $g_{1}\cdot g_{2}$ satisfy the Property (A). Then we have 
\begin{equation*}
\rho _{g_{1}\cdot g_{2}}^{\left( p,q\right) }\left( f_{1},\varphi \right)
=\rho _{g_{1}}^{\left( p,q\right) }\left( f_{1},\varphi \right) =\rho
_{g_{2}}^{\left( p,q\right) }\left( f_{1},\varphi \right) ~.
\end{equation*}
\end{theorem}

\begin{proof}
Let $f_{1},\,f_{2}$ be any two meromorphic functions and $g_{1}$, $g_{2}$ be
any two entire functions satisfying the conditions of the theorem.\medskip 
\newline
\textbf{Case I.} Suppose that $\rho _{g_{1}}^{\left( p,q\right) }\left(
f_{1},\varphi \right) =\rho _{g_{1}}^{\left( p,q\right) }\left(
f_{2},\varphi \right) $ $(0<\rho _{g_{1}}^{\left( p,q\right) }\left(
f_{1},\varphi \right) ,\rho _{g_{1}}^{\left( p,q\right) }\left(
f_{2},\varphi \right) <\infty )$ and $g_{1}$ satisfy the Property (A). Now
in view of Theorem \ref{t9.7}, it is easy to see that $\rho _{g_{1}}^{\left(
p,q\right) }\left( f_{1}\cdot f_{2},\varphi \right) \leq \rho
_{g_{1}}^{\left( p,q\right) }\left( f_{1},\varphi \right) =\rho
_{g_{1}}^{\left( p,q\right) }\left( f_{2},\varphi \right) .$ If possible let 
\begin{equation}
\rho _{g_{1}}^{\left( p,q\right) }\left( f_{1}\cdot f_{2},\varphi \right)
<\rho _{g_{1}}^{\left( p,q\right) }\left( f_{1},\varphi \right) =\rho
_{g_{1}}^{\left( p,q\right) }\left( f_{2},\varphi \right) ~.  \label{20.1}
\end{equation}

\qquad Let\textbf{\ }$\sigma _{g_{1}}^{\left( p,q\right) }\left(
f_{1},\varphi \right) \neq \sigma _{g_{1}}^{\left( p,q\right) }\left(
f_{2},\varphi \right) .$ Now in view of the first part of Theorem \ref{t9.17}
and $\left( \ref{20.1}\right) $ we obtain that $\sigma _{g_{1}}^{\left(
p,q\right) }\left( f_{1},\varphi \right) =\sigma _{g_{1}}^{\left( p,q\right)
}\left( \frac{f_{1}\cdot f_{2}}{f_{2}},\varphi \right) =\sigma
_{g_{1}}^{\left( p,q\right) }\left( f_{2},\varphi \right) $ which is a
contradiction. Hence $\rho _{g_{1}}^{\left( p,q\right) }\left( f_{1}\cdot
f_{2},\varphi \right) $ $=$ $\rho _{g_{1}}^{\left( p,q\right) }\left(
f_{1},\varphi \right) $ $=$ $\rho _{g_{1}}^{\left( p,q\right) }\left(
f_{2},\varphi \right) .$ Similarly with the help of the first part of
Theorem \ref{t9.17}, one can obtain the same conclusion under the hypothesis 
$\overline{\sigma }_{g_{1}}^{\left( p,q\right) }\left( f_{1},\varphi \right)
\neq \overline{\sigma }_{g_{1}}^{\left( p,q\right) }\left( f_{2},\varphi
\right) .$ This prove the first part of the theorem.\newline
\textbf{Case II.} Let us consider that $\rho _{g_{1}}^{\left( p,q\right)
}\left( f_{1},\varphi \right) =\rho _{g_{2}}^{\left( p,q\right) }\left(
f_{1},\varphi \right) $ $(0<\rho _{g_{1}}^{\left( p,q\right) }\left(
f_{1},\varphi \right) ,\rho _{g_{2}}^{\left( p,q\right) }\left(
f_{1},\varphi \right) <\infty )$, $f_{1}$ is of regular relative $\left(
p,q\right) $-$\varphi $ growth with respect to at least any one of $g_{1}$
or $g_{2}$. Also $g_{1}\cdot g_{2}$ satisfy the Property (A). Therefore in
view of Theorem \ref{t9.10}, it follows that $\rho _{g_{1}\cdot
g_{2}}^{\left( p,q\right) }\left( f_{1},\varphi \right) \geq \rho
_{g_{1}}^{\left( p,q\right) }\left( f_{1},\varphi \right) =\rho
_{g_{2}}^{\left( p,q\right) }\left( f_{1},\varphi \right) $ and if possible
let 
\begin{equation}
\rho _{g_{1}\cdot g_{2}}^{\left( p,q\right) }\left( f_{1},\varphi \right)
>\rho _{g_{1}}^{\left( p,q\right) }\left( f_{1},\varphi \right) =\rho
_{g_{2}}^{\left( p,q\right) }\left( f_{1},\varphi \right) ~.  \label{20.2}
\end{equation}

\qquad Further suppose that\textbf{\ }$\sigma _{g_{1}}^{\left( p,q\right)
}\left( f_{1},\varphi \right) \neq \sigma _{g_{2}}^{\left( p,q\right)
}\left( f_{1},\varphi \right) .$ Therefore in view of the proof of the
second part of Theorem \ref{t9.17} and $\left( \ref{20.2}\right) $, we
obtain that $\sigma _{g_{1}}^{\left( p,q\right) }\left( f_{1},\varphi
\right) $ $=$ $\sigma _{\frac{g_{1}\cdot g_{2}}{g_{2}}}^{\left( p,q\right)
}\left( f_{1},\varphi \right) $ $=$ $\sigma _{g_{2}}^{\left( p,q\right)
}\left( f_{1},\varphi \right) $ which is a contradiction. Hence $\rho
_{g_{1}\cdot g_{2}}^{\left( p,q\right) }\left( f_{1},\varphi \right) $ $=$ $%
\rho _{g_{1}}^{\left( p,q\right) }\left( f_{1},\varphi \right) $ $=$ $\rho
_{g_{2}}^{\left( p,q\right) }\left( f_{1},\varphi \right) ~.$ Likewise in
view of the proof of second part of Theorem \ref{t9.17}, one can obtain the
same conclusion under the hypothesis $\overline{\sigma }_{g_{1}}^{\left(
p,q\right) }\left( f_{1},\varphi \right) \neq \overline{\sigma }%
_{g_{2}}^{\left( p,q\right) }\left( f_{1},\varphi \right) .$ This proves the
second part of the theorem.
\end{proof}

\begin{theorem}
\label{t9.19A} Let $f_{1},\,f_{2}$ be any two meromorphic functions and $%
g_{1}$, $g_{2}$ be any two entire functions.\newline
\textbf{(A) }The following conditions are assumed to be satisfied:\newline
$\left( i\right) $ $\left( f_{1}\cdot f_{2}\right) $ is of regular relative $%
\left( p,q\right) $-$\varphi $ growth with respect to at least any one $%
g_{1} $ or $g_{2}$;\newline
$\left( ii\right) $ $\left( g_{1}\cdot g_{2}\right) $, $g_{1}$ and $g_{2}$
all satisfy the Property (A);\newline
$\left( iii\right) $ Either $\sigma _{g_{1}}^{\left( p,q\right) }\left(
f_{1}\cdot f_{2},\varphi \right) \neq \sigma _{g_{2}}^{\left( p,q\right)
}\left( f_{1}\cdot f_{2},\varphi \right) $ or $\overline{\sigma }%
_{g_{1}}^{\left( p,q\right) }\left( f_{1}\cdot f_{2},\varphi \right) \neq 
\overline{\sigma }_{g_{2}}^{\left( p,q\right) }\left( f_{1}\cdot
f_{2},\varphi \right) $;\newline
$\left( iv\right) $ Either $\sigma _{g_{1}}^{\left( p,q\right) }\left(
f_{1},\varphi \right) \neq \sigma _{g_{1}}^{\left( p,q\right) }\left(
f_{2},\varphi \right) $ or $\overline{\sigma }_{g_{1}}^{\left( p,q\right)
}\left( f_{1},\varphi \right) \neq \overline{\sigma }_{g_{1}}^{\left(
p,q\right) }\left( f_{2},\varphi \right) $;\newline
$\left( v\right) $ Either $\sigma _{g_{2}}^{\left( p,q\right) }\left(
f_{1},\varphi \right) \neq \sigma _{g_{2}}^{\left( p,q\right) }\left(
f_{2},\varphi \right) $ or $\overline{\sigma }_{g_{2}}^{\left( p,q\right)
}\left( f_{1},\varphi \right) \neq \overline{\sigma }_{g_{2}}^{\left(
p,q\right) }\left( f_{2},\varphi \right) $; then%
\begin{equation*}
\rho _{g_{1}\cdot g_{2}}^{\left( p,q\right) }\left( f_{1}\cdot f_{2},\varphi
\right) =\rho _{g_{1}}^{\left( p,q\right) }\left( f_{1},\varphi \right)
=\rho _{g_{1}}^{\left( p,q\right) }\left( f_{2},\varphi \right) =\rho
_{g_{2}}^{\left( p,q\right) }\left( f_{1},\varphi \right) =\rho
_{g_{2}}^{\left( p,q\right) }\left( f_{2},\varphi \right) ~.
\end{equation*}%
\textbf{(B) }The following conditions are assumed to be satisfied:\newline
$\left( i\right) $ $\left( g_{1}\cdot g_{2}\right) $ satisfy the Property
(A);\newline
$\left( ii\right) $ $f_{1}$ and $f_{2}$ are of regular relative $\left(
p,q\right) $-$\varphi $ growth with respect to at least any one $g_{1}$ or $%
g_{2}$;\newline
$\left( iii\right) $ Either $\sigma _{g_{1}\cdot g_{2}}^{\left( p,q\right)
}\left( f_{1},\varphi \right) \neq \sigma _{g_{1}\cdot g_{2}}^{\left(
p,q\right) }\left( f_{2},\varphi \right) $ or $\overline{\sigma }%
_{g_{1}\cdot g_{2}}^{\left( p,q\right) }\left( f_{1},\varphi \right) \neq 
\overline{\sigma }_{g_{1}\cdot g_{2}}^{\left( p,q\right) }\left(
f_{2},\varphi \right) $;\newline
$\left( iv\right) $ Either $\sigma _{g_{1}}^{\left( p,q\right) }\left(
f_{1},\varphi \right) \neq \sigma _{g_{2}}^{\left( p,q\right) }\left(
f_{1},\varphi \right) $ or $\overline{\sigma }_{g_{1}}^{\left( p,q\right)
}\left( f_{1},\varphi \right) \neq \overline{\sigma }_{g_{2}}^{\left(
p,q\right) }\left( f_{1},\varphi \right) $;\newline
$\left( v\right) $ Either $\sigma _{g_{1}}^{\left( p,q\right) }\left(
f_{2},\varphi \right) \neq \sigma _{g_{2}}^{\left( p,q\right) }\left(
f_{2},\varphi \right) $ or $\overline{\sigma }_{g_{1}}^{\left( p,q\right)
}\left( f_{2},\varphi \right) \neq \overline{\sigma }_{g_{2}}^{\left(
p,q\right) }\left( f_{2},\varphi \right) $; then%
\begin{equation*}
\rho _{g_{1}\cdot g_{2}}^{\left( p,q\right) }\left( f_{1}\cdot f_{2},\varphi
\right) =\rho _{g_{1}}^{\left( p,q\right) }\left( f_{1},\varphi \right)
=\rho _{g_{1}}^{\left( p,q\right) }\left( f_{2},\varphi \right) =\rho
_{g_{2}}^{\left( p,q\right) }\left( f_{1},\varphi \right) =\rho
_{g_{2}}^{\left( p,q\right) }\left( f_{2},\varphi \right) ~.
\end{equation*}
\end{theorem}

\qquad We omit the proof of Theorem \ref{t9.19A} as it is a natural
consequence of Theorem \ref{t9.19}.

\begin{theorem}
\label{t9.20.} Let $f_{1},\,f_{2}$ be any two meromorphic functions and $%
g_{1}$, $g_{2}$ be any two entire functions.\newline
\textbf{(A)} The following conditions are assumed to be satisfied:\newline
$\left( i\right) $ At least any one of $f_{1}$ or $f_{2}$ are of regular
relative $\left( p,q\right) $-$\varphi $ growth with respect to $g_{1}$;%
\newline
$\left( ii\right) $ Either $\tau _{g_{1}}^{\left( p,q\right) }\left(
f_{1},\varphi \right) \neq \tau _{g_{1}}^{\left( p,q\right) }\left(
f_{2},\varphi \right) $ or $\overline{\tau }_{g_{1}}^{\left( p,q\right)
}\left( f_{1},\varphi \right) \neq \overline{\tau }_{g_{1}}^{\left(
p,q\right) }\left( f_{2},\varphi \right) $ holds.\newline
$\left( iii\right) $ $g_{1}$ satisfy the Property (A), then%
\begin{equation*}
\lambda _{g_{1}}^{\left( p,q\right) }\left( f_{1}\cdot f_{2},\varphi \right)
=\lambda _{g_{1}}^{\left( p,q\right) }\left( f_{1},\varphi \right) =\lambda
_{g_{1}}^{\left( p,q\right) }\left( f_{2},\varphi \right) ~.
\end{equation*}%
\textbf{(B)} The following conditions are assumed to be satisfied:\newline
$\left( i\right) $ $f_{1}$ be any meromorphic function and $g_{1}$, $g_{2}$
be any two entire functions such that $\lambda _{g_{1}}^{\left( p,q\right)
}\left( f_{1},\varphi \right) $ and $\lambda _{g_{2}}^{\left( p,q\right)
}\left( f_{1},\varphi \right) $ exist and $g_{1}\cdot g_{2}$ satisfy the
Property (A);\newline
$\left( ii\right) $ Either $\tau _{g_{1}}^{\left( p,q\right) }\left(
f_{1},\varphi \right) \neq \tau _{g_{2}}^{\left( p,q\right) }\left(
f_{1},\varphi \right) $ or $\overline{\tau }_{g_{1}}^{\left( p,q\right)
}\left( f_{1},\varphi \right) \neq \overline{\tau }_{g_{2}}^{\left(
p,q\right) }\left( f_{1},\varphi \right) $ holds, then%
\begin{equation*}
\lambda _{g_{1}\cdot g_{2}}^{\left( p,q\right) }\left( f_{1},\varphi \right)
=\lambda _{g_{1}}^{\left( p,q\right) }\left( f_{1},\varphi \right) =\lambda
_{g_{2}}^{\left( p,q\right) }\left( f_{1},\varphi \right) ~.
\end{equation*}
\end{theorem}

\begin{proof}
Let $f_{1},\,f_{2}$ be any two meromorphic functions and $g_{1}$, $g_{2}$ be
any two entire functions satisfy the conditions of the theorem.\medskip 
\newline
\textbf{Case I.} Let $\lambda _{g_{1}}^{\left( p,q\right) }\left(
f_{1},\varphi \right) =\lambda _{g_{1}}^{\left( p,q\right) }\left(
f_{2},\varphi \right) $ $(0<\lambda _{g_{1}}^{\left( p,q\right) }\left(
f_{1},\varphi \right) ,\lambda _{g_{1}}^{\left( p,q\right) }\left(
f_{2},\varphi \right) <\infty ),$ $g_{1}$ satisfy the Property (A) and at
least $f_{1}$ or $f_{2}$ is of regular relative $\left( p,q\right) $-$%
\varphi $ growth with respect to $g_{1}$. Now in view of Theorem \ref{t9.8}
it is easy to see that $\lambda _{g_{1}}^{\left( p,q\right) }\left(
f_{1}\cdot f_{2},\varphi \right) $ $\leq $ $\lambda _{g_{1}}^{\left(
p,q\right) }\left( f_{1},\varphi \right) $ $=$ $\lambda _{g_{1}}^{\left(
p,q\right) }\left( f_{2},\varphi \right) .$ If possible let%
\begin{equation}
\lambda _{g_{1}}^{\left( p,q\right) }\left( f_{1}\cdot f_{2},\varphi \right)
<\lambda _{g_{1}}^{\left( p,q\right) }\left( f_{1},\varphi \right) =\lambda
_{g_{1}}^{\left( p,q\right) }\left( f_{2},\varphi \right) ~.  \label{20.3}
\end{equation}

\qquad Also let\textbf{\ }$\tau _{g_{1}}^{\left( p,q\right) }\left(
f_{1},\varphi \right) \neq \tau _{g_{1}}^{\left( p,q\right) }\left(
f_{2},\varphi \right) .$ Then in view of the proof of first part of Theorem %
\ref{t9.18} and $\left( \ref{20.3}\right) ,$ we obtain that $\tau
_{g_{1}}^{\left( p,q\right) }\left( f_{1},\varphi \right) =\tau
_{g_{1}}^{\left( p,q\right) }\left( \frac{f_{1}\cdot f_{2}}{f_{2}},\varphi
\right) =\tau _{g_{1}}^{\left( p,q\right) }\left( f_{2},\varphi \right) $
which is a contradiction. Hence $\lambda _{g_{1}}^{\left( p,q\right) }\left(
f_{1}\cdot f_{2},\varphi \right) $ $=$ $\lambda _{g_{1}}^{\left( p,q\right)
}\left( f_{1},\varphi \right) $ $=$ $\lambda _{g_{1}}^{\left( p,q\right)
}\left( f_{2},\varphi \right) .$ Analogously, in view of the proof of first
part of Theorem \ref{t9.18}$,$ one can derived the same conclusion under the
hypothesis $\overline{\tau }_{g_{1}}^{\left( p,q\right) }\left(
f_{1},\varphi \right) \neq \overline{\tau }_{g_{1}}^{\left( p,q\right)
}\left( f_{2},\varphi \right) $. Hence the first part of the theorem is
established.\newline
\textbf{Case II.} Let us consider that $\lambda _{g_{1}}^{\left( p,q\right)
}\left( f_{1},\varphi \right) =\lambda _{g_{2}}^{\left( p,q\right) }\left(
f_{1},\varphi \right) $ $(0<\lambda _{g_{1}}^{\left( p,q\right) }\left(
f_{1},\varphi \right) ,\lambda _{g_{2}}^{\left( p,q\right) }\left(
f_{1},\varphi \right) <\infty $ and $g_{1}\cdot g_{2}$ satisfy the Property
(A). Therefore in view of Theorem \ref{t9.9}, it follows that $\lambda
_{g_{1}\cdot g_{2}}^{\left( p,q\right) }\left( f_{1},\varphi \right) $ $\geq 
$ $\lambda _{g_{1}}^{\left( p,q\right) }\left( f_{1},\varphi \right) $ $=$ $%
\lambda _{g_{2}}^{\left( p,q\right) }\left( f_{1},\varphi \right) $ and if
possible let%
\begin{equation}
\lambda _{g_{1}\cdot g_{2}}^{\left( p,q\right) }\left( f_{1},\varphi \right)
>\lambda _{g_{1}}^{\left( p,q\right) }\left( f_{1},\varphi \right) =\lambda
_{g_{2}}^{\left( p,q\right) }\left( f_{1},\varphi \right) ~.  \label{20.4}
\end{equation}

\qquad Further let\textbf{\ }$\tau _{g_{1}}^{\left( p,q\right) }\left(
f_{1},\varphi \right) \neq \tau _{g_{2}}^{\left( p,q\right) }\left(
f_{1},\varphi \right) .$ Then in view of second part of Theorem \ref{t9.18}
and $\left( \ref{20.4}\right) $, we obtain that $\tau _{g_{1}}^{\left(
p,q\right) }\left( f_{1},\varphi \right) =\tau _{\frac{g_{1}\cdot g_{2}}{%
g_{2}}}^{\left( p,q\right) }\left( f_{1},\varphi \right) =\tau
_{g_{2}}^{\left( p,q\right) }\left( f_{1},\varphi \right) $ which is a
contradiction. Hence $\lambda _{g_{1}\cdot g_{2}}^{\left( p,q\right) }\left(
f_{1},\varphi \right) $ $=$ $\lambda _{g_{1}}^{\left( p,q\right) }\left(
f_{1},\varphi \right) $ $=$ $\lambda _{g_{2}}^{\left( p,q\right) }\left(
f_{1},\varphi \right) .$ Similarly by second part of Theorem \ref{t9.18}, we
get the same conclusion when $\overline{\tau }_{g_{1}}^{\left( p,q\right)
}\left( f_{1},\varphi \right) \neq \overline{\tau }_{g_{2}}^{\left(
p,q\right) }\left( f_{1},\varphi \right) $ and therefore the second part of
the theorem follows.
\end{proof}

\begin{theorem}
\label{t9.20A} Let $f_{1},\,f_{2}$ be any two meromorphic functions and $%
g_{1}$, $g_{2}$ be any two entire functions.\newline
\textbf{(A)} The following conditions are assumed to be satisfied:\newline
$\left( i\right) $ $g_{1}\cdot g_{2}$, $g_{1}$ and $g_{2}$ satisfy the
Property (A);\newline
$\left( ii\right) $ At least any one of $f_{1}$ or $f_{2}$ are of regular
relative $\left( p,q\right) $-$\varphi $ growth with respect to $g_{1}$ and $%
g_{2}$;\newline
$\left( iii\right) $ Either $\tau _{g_{1}}^{\left( p,q\right) }\left(
f_{1}\cdot f_{2},\varphi \right) \neq \tau _{g_{2}}^{\left( p,q\right)
}\left( f_{1}\cdot f_{2},\varphi \right) $ or $\overline{\tau }%
_{g_{1}}^{\left( p,q\right) }\left( f_{1}\cdot f_{2},\varphi \right) \neq 
\overline{\tau }_{g_{2}}^{\left( p,q\right) }\left( f_{1}\cdot f_{2},\varphi
\right) $;\newline
$\left( iv\right) $ Either $\tau _{g_{1}}^{\left( p,q\right) }\left(
f_{1},\varphi \right) \neq \tau _{g_{1}}^{\left( p,q\right) }\left(
f_{2},\varphi \right) $ or $\overline{\tau }_{g_{1}}^{\left( p,q\right)
}\left( f_{1},\varphi \right) \neq \overline{\tau }_{g_{1}}^{\left(
p,q\right) }\left( f_{2},\varphi \right) $;\newline
$\left( v\right) $ Either $\tau _{g_{2}}^{\left( p,q\right) }\left(
f_{1},\varphi \right) \neq \tau _{g_{2}}^{\left( p,q\right) }\left(
f_{2},\varphi \right) $ or $\overline{\tau }_{g_{2}}^{\left( p,q\right)
}\left( f_{1},\varphi \right) \neq \overline{\tau }_{g_{2}}^{\left(
p,q\right) }\left( f_{2},\varphi \right) $; then%
\begin{equation*}
\lambda _{g_{1}\cdot g_{2}}^{\left( p,q\right) }\left( f_{1}\cdot
f_{2},\varphi \right) =\lambda _{g_{1}}^{\left( p,q\right) }\left(
f_{1},\varphi \right) =\lambda _{g_{1}}^{\left( p,q\right) }\left(
f_{2},\varphi \right) =\lambda _{g_{2}}^{\left( p,q\right) }\left(
f_{1},\varphi \right) =\lambda _{g_{2}}^{\left( p,q\right) }\left(
f_{2},\varphi \right) ~.
\end{equation*}%
\textbf{(B)} The following conditions are assumed to be satisfied:\newline
$\left( i\right) $ $g_{1}\cdot g_{2}$ satisfy the Property (A);\newline
$\left( ii\right) $ At least any one of $f_{1}$ or $f_{2}$ are of regular
relative $\left( p,q\right) $-$\varphi $ growth with respect to $g_{1}\cdot
g_{2}$;\newline
$\left( iii\right) $ Either $\tau _{g_{1}\cdot g_{2}}^{\left( p,q\right)
}\left( f_{1},\varphi \right) \neq \tau _{g_{1}\cdot g_{2}}^{\left(
p,q\right) }\left( f_{2},\varphi \right) $ or $\overline{\tau }_{g_{1}\cdot
g_{2}}^{\left( p,q\right) }\left( f_{1},\varphi \right) \neq \overline{\tau }%
_{g_{1}\cdot g_{2}}^{\left( p,q\right) }\left( f_{2},\varphi \right) $ holds;%
\newline
$\left( iv\right) $ Either $\tau _{g_{1}}^{\left( p,q\right) }\left(
f_{1},\varphi \right) \neq \tau _{g_{2}}^{\left( p,q\right) }\left(
f_{1},\varphi \right) $ or $\overline{\tau }_{g_{1}}^{\left( p,q\right)
}\left( f_{1},\varphi \right) \neq \overline{\tau }_{g_{2}}^{\left(
p,q\right) }\left( f_{1},\varphi \right) $ holds;\newline
$\left( v\right) $ Either $\tau _{g_{1}}^{\left( p,q\right) }\left(
f_{2},\varphi \right) \neq \tau _{g_{2}}^{\left( p,q\right) }\left(
f_{2},\varphi \right) $ or $\overline{\tau }_{g_{1}}^{\left( p,q\right)
}\left( f_{2},\varphi \right) \neq \overline{\tau }_{g_{2}}^{\left(
p,q\right) }\left( f_{2},\varphi \right) $ holds, then%
\begin{equation*}
\lambda _{g_{1}\cdot g_{2}}^{\left( p,q\right) }\left( f_{1}\cdot
f_{2},\varphi \right) =\lambda _{g_{1}}^{\left( p,q\right) }\left(
f_{1},\varphi \right) =\lambda _{g_{1}}^{\left( p,q\right) }\left(
f_{2},\varphi \right) =\lambda _{g_{2}}^{\left( p,q\right) }\left(
f_{1},\varphi \right) =\lambda _{g_{2}}^{\left( p,q\right) }\left(
f_{2},\varphi \right) ~.
\end{equation*}
\end{theorem}

\qquad We omit the proof of Theorem \ref{t9.20A} as it is a natural
consequence of Theorem \ref{t9.20.}.

\begin{remark}
If we take $\frac{f_{1}}{f_{2}}$ instead of $f_{1}\cdot f_{2}$ and $\frac{%
g_{1}}{g_{2}}$ instead of $g_{1}\cdot g_{2}$ where $\frac{f_{1}}{f_{2}}$ is
meromorphic and $\frac{g_{1}}{g_{2}}$ is entire function, and the other
conditions of Theorem \ref{t9.19}, Theorem \ref{t9.19A}, Theorem \ref{t9.20.}
and Theorem \ref{t9.20A} remain the same, then conclusion of Theorem \ref%
{t9.19}, Theorem \ref{t9.19A}, Theorem \ref{t9.20.} and Theorem \ref{t9.20A}
remains valid.
\end{remark}

\end{document}